\documentclass[12pt,oneside]{amsart}
\usepackage[margin=1in]{geometry}

\usepackage{graphicx}
\usepackage{amssymb}
\usepackage{tikz}
\usepackage{tikz-cd}
\usepackage{forest}
\usetikzlibrary{trees, positioning}
\usepackage{mathtools}

\usepackage{setspace}
\usepackage{listings}
\usepackage{xcolor}

\lstdefinelanguage{Macaulay2}{
  keywords={ring, ideal, module, matrix, res, betti, hilbert, radical, saturate, quotient, kernel, image, coker},
  sensitive=true,
  comment=[l]{--},
  morestring=[b]",
}

\usepackage{amsmath,amsthm,amscd,amssymb}
\usepackage{latexsym}
\usepackage[colorlinks,citecolor=red,pagebackref,hypertexnames=false]{hyperref}
\usepackage{bookmark}
\hypersetup{%
  colorlinks = true,
  citecolor=red,
  linkcolor  = black
}

\numberwithin{equation}{section}
\theoremstyle{plain}
\newtheorem{theorem}{Theorem}[section]
\newtheorem*{theorem*}{Theorem}
\newtheorem{lemma}[theorem]{Lemma}

\newtheorem{proposition}[theorem]{Proposition}

\theoremstyle{definition}
\newtheorem{definition}[theorem]{Definition}

\newtheorem{example}[theorem]{Example}

\theoremstyle{remark}
\newtheorem{remark}[theorem]{Remark}

\newtheorem{notation}[theorem]{Notation}

\newtheorem{case[theorem]}{Case}

\def\R{\mathbb R}

\title[Strong Approximation on $\mathcal{S}_{A, B, C, D}$]{Strong Approximation for the Relative Character Variety of the Four-Times Punctured Sphere}
\author{Nathaniel Kingsbury-Neuschotz}

\date{\today}
\begin{document}
\begin{abstract}
We study the orbits of solutions to the Markoff-type equation
$$X^2 + Y^2 + Z^2 = XYZ + AX + BY + CZ + D$$
in $\mathbb{F}_p$ for fixed integers $A,B,C,D$ under the group $\Gamma$ generated by

\[\begin{split}&V_1: (x, y, z)\mapsto (A + yz - x, y, z),\\
&V_2: (x, y, z)\mapsto (x, B + xz - y, z),\text{ and}\\
&V_3: (x, y, z)\mapsto (x, y, C + xy - z).\end{split}\]

This equation defines the Relative Character Variety of the Four-Times Punctured Sphere, with $\Gamma$ arising from the Pure Mapping Class Group. Outside an explicit degeneracy locus, $\Gamma$ acts transitively on almost all solutions for a density-one set of primes independent of the parameters; the others split into small orbits reflecting finite orbits over $\mathbb{C}$. For ``degenerate'' parameters, we show there are two large orbits (most cases) or four (the rest, excluding $(0, 0, 0, 4)$) for the same primes. 

These results strengthen for two subfamilies. The first,

$$X^2 + Y^2 + Z^2 = XYZ + k,\,\,\,k\neq 4,$$

arises in the combinatorial group theory of $\text{SL}_2(\mathbb{F}_p)$. Combined with Chen's work, $\Gamma$ acts transitively for all large primes, uniformly in $k$, so the $Q$-classification conjecture of McCullough and Wanderley holds over $\mathbb{F}_p$ for $p$ large. By Martin's work, so do their Classification and $T$-Classification Conjectures. The second,

$$x_1^2 + x_2^2 + x_3^2 + a_1x_2x_3 + a_2x_1x_3 + a_3x_1x_2 = (3+a_1+a_2+a_3)x_1x_2x_3,$$

arises from generalized cluster algebras. Our degeneracy notion specializes to that of de Courcy-Ireland, Litman, and Mizuno, whose results similarly imply our orbit count (1, 2, or 4) for all large primes for all nondegenerate and some degenerate surfaces in that family.
\end{abstract}

\maketitle

\tableofcontents

\section{Introduction}
The Markoff equation is the classical Diophantine equation
\begin{equation}\label{Markoff}
\mathcal{M}:\, X^2+Y^2 + Z^2 = 3XYZ.
\end{equation}
By the work of Markoff and Hurwitz, positive integer solutions to this equation govern successive minima of indefinite binary quadratic forms and poorly approximable irrational numbers (see \cite{Markoff1} and \cite{Markoff2} for Markoff's work on quadratic forms, and \cite{BombieriContinuedFracs} and \cite{CasselsMarkoff} for connections to Diophantine approximation). This equation has an extremely rich group of symmetries, prominently including the Vieta involutions, which are given by
\[\begin{split}&V_1: (x, y, z)\mapsto (3yz - x, y, z),\\
&V_2: (x, y, z)\mapsto (x, 3xz - y, z),\text{ and}\\
&V_3: (x, y, z)\mapsto (x, y, 3xy - z).\end{split}\]
These involutions arise by fixing the values of two variables, treating the equation as quadratic in the third variable, and swapping the two roots thereof.\footnote{In addition to the Vieta involutions, one can of course also permute the coordinates of solutions to the equation, or negate two of the three coordinates of a solution. However, the Vieta involutions are the most significant symmetries---they generate a finite index normal subgroup isomorphic to $(\mathbb{Z}/2\mathbb{Z}) *(\mathbb{Z}/2\mathbb{Z}) *(\mathbb{Z}/2\mathbb{Z})$ (a free product of three copies of $\mathbb{Z}/2\mathbb{Z}$, cf. \cite{elHuti}, Theorem 1) of the full group of symmetries.} Starting from the root solution $(1, 1, 1)$, all solutions to (\ref{Markoff}) in positive integers may be generated by applying a sequence of Vieta involutions, giving them a tree structure, depicted up to permutations in Figure \ref{MarkoffTree}. In modern language, if $\Gamma=\langle V_1, V_2, V_3\rangle,$ then $\Gamma$ acts freely and transitively on $\mathcal{M}(\mathbb{Z}_{>0})$. There are a number of classical questions about the solutions to the Markoff equation, prominently including Frobenius' Uniqueness Conjecture (see \cite{Frobenius}, as well as \cite{Aigner} for a survey of topics related to this conjecture): does the largest coordinate of a solution to (\ref{Markoff}) in positive integers determine the values of the remaining two coordinates? 

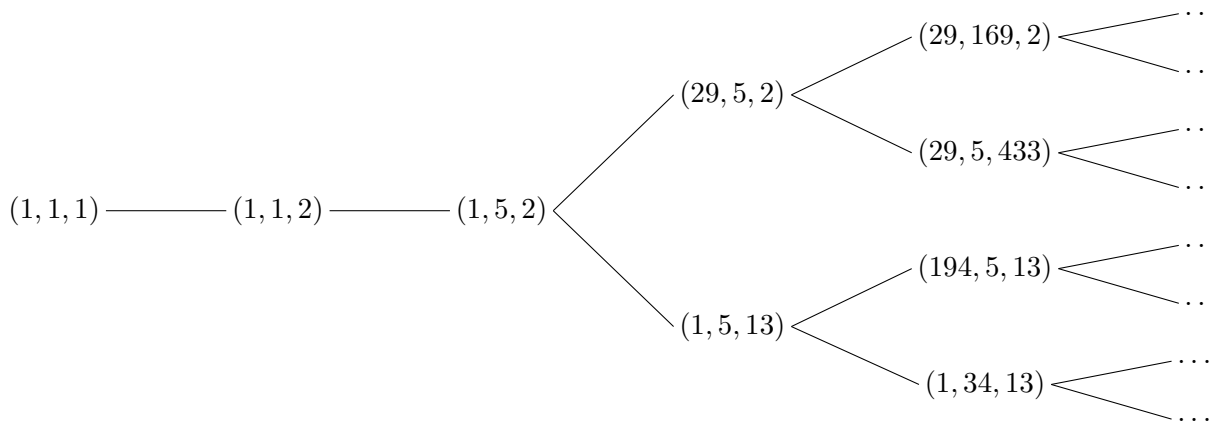
\begin{figure}
\centering
\begin{forest}
for tree={
  grow=east,
  parent anchor=east,
  child anchor=west,
  edge={draw},
  l sep=16mm,
  s sep+=3mm,
  font=\small,
  inner sep=2pt,
}
[{$(1,1,1)$}
  [{$(1,1,2)$}
    [{$(1,5,2)$}
      [{$(1,5,13)$}
        [{$(1,34,13)$}
          [$\dots$]
          [$\dots$]
        ]
        [{$(194,5,13)$}
          [$\dots$]
          [$\dots$]
        ]
      ]
      [{$(29,5,2)$}
        [{$(29,5,433)$}
          [$\dots$]
          [$\dots$]
        ]
        [{$(29,169,2)$}
          [$\dots$]
          [$\dots$]
        ]
      ]
    ]
  ]
]
\end{forest}
    \label{MarkoffTree}
    \caption{The Markoff Tree}
\end{figure}

Baragar (\cite{Baragar91}) conjectured that just as $\Gamma$ acts transitively on $\mathcal{M}(\mathbb{Z}_{>0})$, so too it acts transitively on $\mathcal{M}^*(\mathbb{F}_p)\vcentcolon = \mathcal{M}(\mathbb{F}_p) \setminus\{(0, 0, 0)\}$. If this is true, the reduction mod $p$ map $\mathcal{M}(\mathbb{Z})\to\mathcal{M}(\mathbb{F}_p)$ would be surjective, a form of \textit{strong approximation}. Indeed, given any nontrivial solution $(x, y, z)$ to (\ref{Markoff}) over $\mathbb{F}_p$, there would be a sequence of Vieta involutions taking $(1, 1, 1)$ to $(x, y, z)$; applying this same sequence of involutions in $\mathbb{Z}$ to the root solution $(1, 1, 1)$ then gives an integral solution congruent to $(x, y, z)$ mod $p$. Results of Bourgain, Gamburd, and Sarnak \cite{BGS} imply that this is true for density one of all primes, and a further result of Chen \cite{Chen} improved this to hold for all sufficiently large primes; in fact, it suffices to have $p> 3.45\times 10^{392}$ (\cite{MarkoffConnectivityBound}). Bourgain, Gamburd, and Sarnak (\cite{BGSPreprint}) extended the corresponding transitivity result to $\mathbb{Z}/N\mathbb{Z}$ for squarefree $N$ whose prime factors are all $\equiv 1$ (mod 4), and Meiri and Puder (\cite{MeiriPuder}) extended it to squarefree $N$ whose prime factors lie in a density-one set containing all primes $\equiv 1$ (mod 4). These results enabled Bourgain, Gamburd, and Sarnak to establish that almost all Markoff numbers are composite (\cite{BGS}, Theorem 3). Bourgain, Gamburd, and Sarnak phrase their transitivity results in terms of a family of graphs they call Markoff graphs, which they conjecture to be expander graphs; were expansion proven one could use sieve methods to establish that infinitely many Markoff numbers are almost primes, as in the affine sieve (see for instance \cite{BGSExpansion}, \cite{SarnakLinearSieveNotes}, and \cite{GolsefidySarnakAffLinSieve}). For broad, up-to-date surveys of material related to arithmetic and dynamics on the Markoff surface and some of its relatives, we refer the reader to \cite{SilvermanNotices} and \cite{GamburdICM}.

Since the work of Bourgain, Gamburd, and Sarnak, there has been interest in extending these types of results to broader classes of varieties. The simplest extension would be to the family given by
\begin{equation}\label{MarkoffPlusk}
    X^2 + Y^2 + Z^2 = XYZ + k,
\end{equation}
for $k \neq 4$, which arises in geometric group theory as the character variety of the once punctured torus, and which includes the classical Markoff equation (\ref{Markoff}) as the case $k=0$ after the usual rescaling of coordinates. The excluded case $k = 4$ is the Cayley cubic, which is known to be degenerate and to have a very large number of orbits; see \cite{LT} for its many finite orbits over $\mathbb{C}$, and \cite{deCourcyIrelandLee} for its orbit structure over $\mathbb{F}_p$.

We define $V_1, V_2,$ and $V_3$ by
\[\begin{split}&V_1: (x, y, z)\mapsto (yz - x, y, z),\\
&V_2: (x, y, z)\mapsto (x, xz - y, z),\text{ and}\\
&V_3: (x, y, z)\mapsto (x, y, xy - z),\end{split}\]
reflecting the rescaling from the Markoff case, and may study the structure of orbits on this variety under the action of the group $\Gamma$. Transitivity results for this action have important implications for the combinatorial group theory of $\text{SL}_2(\mathbb{F}_p)$; see Section \ref{GroupTheory} below for details. 

Another generalization of the Markoff equation is the family
\begin{equation}\label{GyodaMatsushitaEq}x_1^2 + x_2^2 + x_3^2 + a_1x_2x_3 + a_2x_1x_3 + a_3x_1x_2 = (3 + a_1 + a_2 + a_3)x_1x_2x_3\end{equation}
where $a_1,$ $a_2$, and $a_3$ are fixed parameters. Once again, one may define Vieta involutions in each variable by interchanging the two roots of the quadratic equation in that variable, though the resulting formulas are more complicated than in the Markoff case. This family was first introduced by Gyoda and Matsushita \cite{GyodaMatsushita} in connection with Generalized Cluster Algebras.\footnote{The Markoff tree has an important interpretation in terms of Cluster Algebras, which has been used for instance in the proof of Aigner's Monotonicity Conjecture, see \cite{AignerMono2} and \cite{AignerMono1}.} Much as with the Markoff equation, $\Gamma$ acts freely and transitively on the positive integral solutions to (\ref{GyodaMatsushitaEq}), giving them a tree structure, and the associated tree may be interpreted in terms of mutations of Generalized Cluster Algebras---we refer the reader to \cite{GyodaMatsushita} for a definition of this term. Orbits in the $\mathbb{F}_p$ points of (\ref{GyodaMatsushitaEq}) were studied by de Courcy-Ireland, Litman, and Mizuno in \cite{pDivisibilityClusterAlg}, where they extended Chen's theorem to most varieties in this family following the proof of Martin \cite{Martinpdivisibility}; see Section \ref{GeneralizedClusterAlgsSection} of this paper for details.

Another important family of Markoff-type varieties is the family of K3 surfaces of Markoff type, first identified by Fuchs, Litman, Silverman, and Tran in \cite{FuchsLitmanSilvermanTran}. While these surfaces have several important analogies with the Markoff surface, there are several significant differences which make their dynamics over $\mathbb{F}_p$ much harder to study---most significantly, that slicing them by a plane such as $X = x_0$ gives a homogeneous space of an elliptic curve, rather than of the algebraic torus $\mathbb{G}_m$, and the geometry and group theory of elliptic curves is much more complicated than that of $\mathbb{G}_m$. 

In this paper, we extend the results of Bourgain, Gamburd, and Sarnak to the family of surfaces $\mathcal{S}_{A, B, C, D}$ given by the equation 
\begin{equation}\label{MainEq}
    X^2 + Y^2 + Z^2 = XYZ + AX + BY + CZ + D,
\end{equation}
following a suggestion in the announcement of their results on the Markoff equation (\cite{BGSAnnouncement}). This surface arises geometrically as the \textit{relative character variety of the four-holed sphere}, see for instance \cite{BenedettoGoldman}, and the dynamics on $\mathcal{S}_{A, B, C, D}(\mathbb{C})$ and $\mathcal{S}_{A, B, C, D}(\mathbb{R})$ of the group generated by its Vieta involutions has been studied extensively by geometers and dynamicists; see for instance \cite{ErgodicTheoryModuliSpaces}, \cite{Cantat}, \cite{CantatLoray}, \cite{CantatDupontMartinBaillon}, and the references contained therein. This equation in some sense interpolates between equations (\ref{MarkoffPlusk}) and (\ref{GyodaMatsushitaEq}): equation (\ref{MarkoffPlusk}) arises from parameters $A = B = C = 0$, and equation (\ref{GyodaMatsushitaEq}) may be transformed into an equation of the form (\ref{MainEq}) by way of the change of variables described in (\ref{clusterAlgChangeVars}) below. As before, there is a rich group of symmetries generated by Vieta involutions which swap the two roots of the quadratic equations in each variable; for ease of reference we record them below:
\[\begin{split}&V_1: (x, y, z)\mapsto (A + yz - x, y, z),\\
&V_2: (x, y, z)\mapsto (x, B + xz - y, z),\text{ and}\\
&V_3: (x, y, z)\mapsto (x, y, C + xy - z).\end{split}\]
Geometrically, the group $\Gamma$ is exactly the extended pure mapping class group; generically it is the full group of automorphisms of $\mathcal{S}_{A, B, C, D}$ over $\mathbb{Z}$, and it has index at most 24 within the full group of automorphisms over $\mathbb{Z}$; see Remark \ref{rmk:GammaPrimeFullAut} below for details.

There are two important complications to a Bourgain-Gamburd-Sarnak type result, generalizing complications already present in equations (\ref{MarkoffPlusk}) and (\ref{GyodaMatsushitaEq}). The first is the presence of small orbits: in the case of the equation
$$X^2 + Y^2 + Z^2 = XYZ + k,$$
if $k$ is a quadratic residue mod $p$, then $\{(\sqrt{k}, 0, 0), (-\sqrt{k}, 0, 0)\}$ is a complete orbit under the action of $\Gamma$. For certain special values of the parameter $k$, there are additional small orbits, such as the orbit of ${(1, 1, 1)}$ in the case where $k = 2$. These orbits are quite well understood. In connection with $\text{SL}_2,$ they arise from the so-called ``nonessential triples'' studied in \cite{McCulloughWanderley} (see also \cite{McCulloughExceptionalSubgroups}). All of these small orbits arise from identities of algebraic numbers, and hence also form finite orbits on $\mathcal{S}_{0, 0, 0, k}(\mathbb{C});$ these finite orbits give rise to finite-branching solutions to certain Painlev\'e VI equations, and have been classified by Dubrovin and Mazzocco \cite{DubrovinMazzocco} as a means of studying algebraic solutions to this subfamily of Painlev\'e VI equations. In the literature on arithmetic and dynamics on Markoff type varieties, these orbits are discussed in \cite{deCourcyIrelandLee} and \cite{MartinBigPaper}. In order to precisely state our main theorem, it is necessary to classify and remove all of the small orbits on our more general $\mathcal{S}_{A, B, C, D}$. This is done for us by Lisovyy and Tykhyy, who extended the classification results of Dubrovin and Mazzocco to all Painlev\'e VI equations in \cite{LT}; we include an overview of their results in Section \ref{FiniteOrbitsSection} below. 

The second complication is the presence, for certain ``degenerate'' parameters, of multiple large orbits. This type of phenomenon was first noticed for a certain family of K3 surfaces of Markoff type in \cite{FuchsLitmanSilvermanTran}, and it was established in that setting in terms of a double cover in \cite{ODorneyObstruction}. It was later noted in the family (\ref{GyodaMatsushitaEq}), which maps into our family via (\ref{clusterAlgChangeVars}), by de Courcy-Ireland, Litman, and Mizuno in \cite{pDivisibilityClusterAlg}: if some $a_i^2 = 4$, then there are at least two large orbits, and if a second $a_j^2 = 4$, then there are at least four orbits. In their paper, they described $\Gamma$-invariant sets explaining this breakdown using explicit polynomial identities, leaving implicit the double-cover interpretation of this obstruction, though they did not show that $\Gamma$ acts transitively on these $\Gamma$-invariant sets. In our setting, this phenomenon also holds for some quadruples of parameters $(A, B, C, D)$ which do not arise via (\ref{clusterAlgChangeVars}) from the equation (\ref{GyodaMatsushitaEq}) with rational integer parameters; we classify these degenerate parameters below in Definition \ref{nondegeneracyCondition} and prove in Theorem \ref{QuadraticObstructions} that for degenerate parameters, there must be at least two large $\Gamma$-invariant sets, and sometimes even at least four large $\Gamma$-invariant sets, under the action of $\Gamma$. Our proof closely resembles that of de Courcy-Ireland, Litman, and Mizuno in \cite{pDivisibilityClusterAlg}. With these two complications resolved, we prove the following theorem. At the end of Section \ref{FiniteOrbitsSection}, in which we discuss the small orbits mentioned in Theorem \ref{th:mainFirstForm} and introduce notation for their complement, we shall restate this more compactly as Theorem \ref{th:mainRestated}. This theorem is restated yet again as Theorem \ref{th:mainBadPrimesExplicit}, where the zero-density set of primes excluded from the result is described in somewhat more detail.
\begin{theorem}\label{th:mainFirstForm}
    There is a density one set of primes $p$ for which, if $(A, B, C, D)$ is nondegenerate modulo $p$ in the sense of Definition \ref{nondegeneracyCondition}, $\mathcal{S}_{A, B, C, D}(\mathbb{F}_p)$ consists of a single giant orbit under $\Gamma$, possibly together with some orbits of bounded size. These small orbits are images under a suitable reduction map of finite orbits on $\mathcal{S}_{A, B, C, D}(\mathbb{C})$, in the sense made precise in Section \ref{FiniteOrbitsSection}. The density one set of primes does not depend on the parameters $A, B, C,$ and $D$.
\end{theorem}

In Section \ref{ModificationsDegenerate}, we shall modify this result, ultimately showing that (for the exact same density-one set of primes) the group $\Gamma$ acts transitively on each of the $\Gamma$-invariant sets identified in Theorem \ref{QuadraticObstructions}, so that, depending on the exact relations between the parameters, there are either exactly two large orbits or four large orbits. To the author's knowledge, this appears to be the first time a Markoff-type variety has been shown to have exactly two or exactly four orbits after excluding those that arise from finite orbits over $\mathbb{C}$---previous work has either shown the presence of exactly one orbit, or the presence of at least two or at least four.

The restriction to a density-one set of primes arises in the course of the so-called ``opening'' (Section \ref{openingSection}), in which we show that the set $\mathcal{N}$ of  primes for which we cannot establish connectivity satisfies $|\mathcal{N}\cap[1, z]| = O(z^\epsilon)$ for $\epsilon>0$. For the rare primes not covered by the arguments in Section \ref{openingSection}, we are able to establish results similar to those of Bourgain, Gamburd, and Sarnak in \cite{BGS}; namely, existence of a large connected component, and a logarithmic lower bound on the size of any mod $p$ orbit which is not the image of a finite orbit over $\mathbb{C}.$ In particular, we establish the following theorems:
\begin{theorem}\label{th:EndgamePlusMiddlegame}
    Suppose $p$ is sufficiently large, and that $(A, B, C, D)$ is nondegenerate modulo $p$. Then there is a large connected component $\mathcal{C}_{A, B, C, D}(p)$ (the $\Gamma$-orbit containing the cage defined in Subsection \ref{cageConnectivitySubsect}) such that for any $\epsilon>0,$
    $$|\mathcal{S}_{A, B, C, D}(\mathbb{F}_p)\setminus \mathcal{C}_{A, B, C, D}(p)| <<_{\epsilon} p^\epsilon.$$
    The implied constant depends only on $\epsilon$, and the statement holds for $p\geq p_0$ with $p_0$ absolute.
\end{theorem}
\begin{theorem}[c.f. Proposition \ref{AlgNTBound}]\label{th:LogarithmicLowerBoundFirstAppearance}
    Any $\Gamma$-orbit $F$ in $\mathcal{S}_{A, B, C, D}(\mathbb{F}_p)$ which is not the image of a finite orbit on $\mathcal{S}_{A, B, C, D}(\mathbb{C})$ under the reduction map has size
    $$|F|\geq \frac{1}{2}(\log_{(20 + 2|A| + 2|B| + 2|C| + |D|)}p)^\frac{1}{3}.$$
\end{theorem}
\begin{remark}[How large is large enough?]
    The restriction on small primes in Theorem \ref{th:EndgamePlusMiddlegame} comes from requiring the cage $\mathcal{K}$ to be connected via the arguments of the so-called ``endgame'' (Section \ref{EndgameSection}); the author has not explicitly computed the cutoff $p_0$ but expects that it's not particularly large---for instance, $p\geq 100$ likely suffices. However, when $\epsilon$ is small, the implicit constant is often quite large, due to the many iterations of the ``middlegame'' (Section \ref{middlegameSection}) which are needed. For example, even after optimizing the argument, the authors of \cite{MarkoffConnectivityBound} were only able to ensure $|\mathcal{S}_{0, 0, 0, 0}(p) \setminus \mathcal{C}_{0,0,0,0}(p)| < 4p$ for $p>3.45\times10^{392}.$ 
\end{remark}
\begin{remark}[Parameter dependence]
    The density-zero set of primes excluded from Theorem \ref{th:mainFirstForm} is independent of the parameters; the bound in \ref{th:LogarithmicLowerBoundFirstAppearance} is not. We close this gap with a ``spreading-out'' argument (Subsection \ref{subsect:SpreadingOut}): orbits which are not reductions of finite orbits over $\mathbb{C}$ have sizes tending to infinity with $p$ uniformly in $A, B, C, D$, though without a rate. The methods of \cite{KMSV} would likely improve the lower bound of Theorem \ref{th:LogarithmicLowerBoundFirstAppearance} but would not remove its parameter dependence.
\end{remark}

 We can also avoid the restriction to a density-one set of primes in our results for degenerate parameters. Theorem \ref{th:LogarithmicLowerBoundFirstAppearance} requires no modification, but Theorem \ref{th:EndgamePlusMiddlegame} does:
\begin{theorem}\label{th:DegenerateEndgamePlusMiddlegame}
    Suppose that $p$ is sufficiently large and that $(A, B, C, D)$ is degenerate modulo $p$, where $(A, B, C, D)\not\equiv(0, 0, 0, 4)\pmod{p}$. Then for each $\Gamma$-invariant set $\mathcal{S}_i(p)$ identified in Theorem \ref{QuadraticObstructions}, there is a large connected component $\mathcal{C}_i(p)$ such that for any $\epsilon>0,$
    $$|\mathcal{S}_i(p)\setminus \mathcal{C}_i(p)| <<_{\epsilon} p^\epsilon.$$
\end{theorem}
We note that the restriction to avoid $(A, B, C, D)\equiv (0, 0, 0, 4)\pmod{p}$ is imposed to avoid the Cayley cubic, see Remark \ref{rmk:CayleyParams} for details.

Remarkably, for all surfaces in the family (\ref{MarkoffPlusk}) and most surfaces in the family (\ref{GyodaMatsushitaEq}), all $\Gamma$-orbits other than the reductions of finite orbits on $\mathcal{S}_{A, B, C, D}(\mathbb{C})$ have size at least $p$; for (\ref{GyodaMatsushitaEq}) this is due to de Courcy-Ireland, Litman, and Mizuno \cite{pDivisibilityClusterAlg} (cf. Theorem \ref{pDivisibilityResult}), whereas for (\ref{MarkoffPlusk}) this is a very recent result of Chen \cite{ChenMarkPlusk}. Combining these strong lower bounds with our Theorem \ref{th:EndgamePlusMiddlegame} (or our Theorem \ref{th:DegenerateEndgamePlusMiddlegame} in the degenerate case) upgrades our main result to hold for all sufficiently large primes $p$. 

\begin{theorem}\label{th:MainSufficientlyLargeAllCases}
    For all sufficiently large primes $p$, $\Gamma$ acts transitively on $\mathcal{S}_{0, 0, 0, k}^*(p)$ for $k\not\equiv 4\pmod{p}.$ The threshold for $p$ does not depend on $k$.

    Similarly, for all sufficiently large primes $p$,
    \begin{enumerate}
        \item If $3 + a_1 + a_2 + a_3 \not\equiv 0\pmod{p}$ and $a_i^2\not\equiv  4\pmod{p}$ for $i = 1, 2, 3$, then the mod $p$ solutions to (\ref{GyodaMatsushitaEq}) other than the trivial solution $(x_1,x_2,x_3) = (0,0,0)$ form a single orbit under the action of the group generated by the three Vieta involutions.

        \item If instead $3 + a_1 +a_2 + a_3 \not\equiv 0\pmod{p}$ but $a_i^2 \equiv 4\pmod{p}$ for some $i$, and $2a_{i-1} \equiv a_ia_{i+1}\pmod{p}$, then the mod $p$ solutions to (\ref{GyodaMatsushitaEq}) other than the trivial solution $(x_1,x_2,x_3) = (0,0,0)$ form either two (if $a_{i-1}^2 \not\equiv 4\pmod{p}$) or four (if $a_{i-1}^2 \equiv 4\pmod{p}$) orbits under the action of the group generated by the three Vieta involutions.
    \end{enumerate} 
    The threshold for $p$ may be taken to not depend on $a_1, a_2, a_3$, or on whether or not any $a_i^2 \equiv 4\pmod{p}.$ 
\end{theorem}
\begin{remark}[How large is large enough? Reprise.]
    These full-connectivity-type results arise from playing Theorems \ref{th:EndgamePlusMiddlegame} and \ref{th:DegenerateEndgamePlusMiddlegame} off of the lower bounds on the size of orbits given by Theorems \ref{th:MarkoffPluskLowerBound} and \ref{pDivisibilityResult}. As our methods are effective, one can make the ``sufficiently large'' condition on $p$ explicit, though the author has not done so. The authors of \cite{MarkoffConnectivityBound} obtain the cutoff of $p>3.45\times 10^{392}$ for the classical Markoff equation after optimizing both the arguments of \cite{BGS} and amplifying the lower bound provided by Chen's congruence to $|\Gamma(x,y,z)| \geq 4p$. As our arguments are quite similar in shape to those of \cite{BGS} for the classical Markoff equation, we expect an explicit version of our arguments, even after the optimizations of \cite{MarkoffConnectivityBound}, would yield a cutoff at least as large, and in fact a cutoff which is even larger---in the case of the classical Markoff equation, the amplification of Chen's lower bound to $4p$ uses the fact that each orbit has size divisible by $p$ and that $\mathcal{S}_{0,0,0,0}$ has extra automorphisms not contained in $\Gamma$ (specifically the automorphisms $\operatorname{neg}_{xy}$, $\operatorname{neg}_{xz}$, and $\operatorname{neg}_{yz}$ described in Subsection \ref{ExtraAutsSubSect}). For the family \ref{MarkoffPlusk}, Chen has proven a lower bound of $p$ on the size of each nontrivial orbit, but these orbits do not all have size divisible by $p$; for the family (\ref{GyodaMatsushitaEq}), de Courcy-Ireland, Litman, and Mizuno established divisibility of orbit sizes by $p$, but these surfaces do not typically have the extra automorphisms which \cite{MarkoffConnectivityBound} uses.
\end{remark}

For the family (\ref{MarkoffPlusk}), this result has important consequences in combinatorial group theory, which we outline in Section \ref{GroupTheory}.

Finally, we note that since excluding finitely many primes does not affect density-one statements and asymptotic statements, we shall freely discard small primes throughout the proof. In particular, all arguments over $\mathbb{F}_p$ may be understood with $p$ odd.

\subsection{Equivalence of Parameters}\label{EquivalenceParamsSubect}

For the purposes of defining degenerate parameters and describing the small orbits in $\mathcal{S}_{A, B, C, D},$ it will be convenient to introduce an equivalence between some of the surfaces $\mathcal{S}_{A, B, C, D}.$ In particular, the following three maps commute with the action of the Vieta involutions on each surface:

$$\operatorname{neg}_{xy}: \mathcal{S}_{A, B, C, D}\rightarrow \mathcal{S}_{-A, -B, C, D}\text{ by } (x, y, z)\mapsto (-x, -y, z),$$
$$\operatorname{neg}_{xz}: \mathcal{S}_{A, B, C, D}\rightarrow \mathcal{S}_{-A, B, -C, D}\text{ by } (x, y, z)\mapsto (-x, y, -z),\text{ and }$$
$$\operatorname{neg}_{yz}: \mathcal{S}_{A, B, C, D}\rightarrow \mathcal{S}_{A, -B, -C, D}\text{ by } (x, y, z)\mapsto (x, -y, -z).$$

Additionally, the following three maps induce automorphisms of $\Gamma$ by permuting $V_1,$ $V_2,$ and $V_3$:

$$\tau_{xy}: \mathcal{S}_{A, B, C, D}\rightarrow\mathcal{S}_{B, A, C, D}\text{ by }(x,y,z)\mapsto(y,x,z),$$
$$\tau_{xz}: \mathcal{S}_{A, B, C, D}\rightarrow\mathcal{S}_{C, B, A, D}\text{ by }(x,y,z)\mapsto(z,y,x),\text{ and}$$
$$\tau_{yz}: \mathcal{S}_{A, B, C, D}\rightarrow\mathcal{S}_{A, C, B, D}\text{ by }(x,y,z)\mapsto(x,z,y).$$

It follows that if parameters $(A', B', C', D')$ are obtained from parameters $(A, B, C, D)$ by a sequence of moves of this sort, then $\mathcal{S}_{A', B', C', D'}$ and $\mathcal{S}_{A, B, C, D}$ have the same orbit structure. In the language of permutation groups, $(\Gamma, \mathcal{S}_{A, B, C, D})$ and $(\Gamma, \mathcal{S}_{A', B', C', D'})$ are permutation isomorphic, though they are not equivalent actions of $\Gamma$; see Chapter 1 of \cite{DixonMortimer} for the terminology used here. We are led to the following definition:
\begin{definition}\label{def:EquivalentParams}
    Two quadruples of parameters $(A, B, C, D)$ and $(A', B', C', D')$ are said to be \textit{equivalent} if $(A', B', C', D')$ may be obtained from $(A, B, C, D)$ by permuting $A, B,$ and $C$ and negating either 0 or 2 of $A$, $B,$ and $C$. $\Gamma$-invariant subsets $\mathcal{O}$ and $\mathcal{O}'$ of $\mathcal{S}_{A, B, C, D}$ and $\mathcal{S}_{A', B', C', D'}$ are said to be \textit{equivalent} if $(A', B', C', D')$ is equivalent to $(A, B, C, D)$ and an isomorphism $\mathcal{S}_{A, B, C, D}\rightarrow \mathcal{S}_{A', B', C', D'}$ arising from this equivalence maps $\mathcal{O}$ to $\mathcal{O}'$.
\end{definition}

\subsection{Extra Automorphisms}\label{ExtraAutsSubSect}
For some exceptional values of the parameters, these equivalences actually induce extra automorphisms of $\mathcal{S}_{A, B, C, D}$. If $A = B$, $A = C$, or $B = C$, then one of $\tau_{xy},$ $\tau_{xz},$ or $\tau_{yz}$ is an automorphism of $\mathcal{S}_{A, B, C, D},$ and if $A = B = C$, then they all are. If at least two of $A, B,$ and $C$ are zero, then the corresponding maps among $\operatorname{neg}_{xy},$ $\operatorname{neg}_{xz},$ and $\operatorname{neg}_{yz}$ are automorphisms of $\mathcal{S}_{A, B, C, D}.$ Slightly less obviously, if $A = -B,$ $A = -C$, or $B = -C$, then the corresponding maps among $\operatorname{neg}_{xy}\circ \tau_{xy}$, $\operatorname{neg}_{xz}\circ \tau_{xz}$, and $\operatorname{neg}_{yz}\circ \tau_{yz}$ are automorphisms of $\mathcal{S}_{A, B, C, D}$. We denote by $H$ the group generated by all automorphisms of the form $\operatorname{neg}_{xy}$, $\operatorname{neg}_{xz}$, and $\operatorname{neg}_{yz},$ and $\Gamma'$ the group generated by the Vieta involutions $V_1,$ $V_2,$ and $V_3$ together with all applicable automorphisms of the form $\operatorname{neg}$, $\tau,$ or $\operatorname{neg}\circ\tau$. 
\begin{remark}[$\Gamma'$ is the full automorphism group]\label{rmk:GammaPrimeFullAut}
The group $\Gamma'$ is not merely a convenient enlargement of $\Gamma$: it is the full group of polynomial automorphisms of $\mathcal{S}_{A,B,C,D}$ defined over $\mathbb{Z}$. This is due to \`El$'$-Huti \cite{elHuti}, in the form given by Cantat and Loray (\cite{CantatLoray}, Theorem 3.1; this is Theorem A of their introduction): $\Gamma\cong(\mathbb{Z}/2)^{*3}$ is a normal subgroup of $\operatorname{Aut}(\mathcal{S}_{A,B,C,D}(\mathbb{C}))$ of index at most $24$, and $\operatorname{Aut}(\mathcal{S}_{A,B,C,D}(\mathbb{C}))$ is generated by $\Gamma$ together with the affine transformations of $\mathbb{C}^3$ preserving the surface. Such a transformation is necessarily linear, of the form $\sigma\circ\varepsilon$ with $\sigma$ a permutation of the coordinates and $\varepsilon$ either trivial or a change of sign of two coordinates; these are exactly the maps $\operatorname{neg}$, $\tau$, and $\operatorname{neg}\circ\tau$ of this subsection, whichever of them preserve the surface for the given parameters. Their result is stated over $\mathbb{C}$, but it applies \textit{a fortiori} to automorphisms defined over $\mathbb{Z}$, since an automorphism over $\mathbb{Z}$ is in particular an automorphism over $\mathbb{C}$, and every element of $\Gamma'$ has integer coefficients. Over $\mathbb{F}_p$ or $\overline{\mathbb{F}}_p$ the surface may have further symmetries (for instance the Frobenius endomorphism acts over $\overline{\mathbb{F}}_p$), which play no role here.
 
We caution that these automorphisms do not all have the same origin. The group $\Gamma$ is the image of the pure mapping class group of the four-times punctured sphere, extended by the orientation-reversing classes: the Vieta involutions are orientation-reversing, and the orientation-preserving subgroup (the image of the pure mapping class group proper) is the subgroup of $\Gamma$ of index two consisting of words of even length (\cite{CantatLoray}, Proposition 2.5). The maps $\tau$ arise from mapping classes which permute the punctures (\cite{CantatLoray}, Subsection 2.4). In contrast, the maps $\operatorname{neg}$ do \textit{not} come from the action of the mapping class group at all: they arise from twisting a representation by a character of the fundamental group with values in the center $\{\pm I\}$ of $\operatorname{SL}_2$ (\cite{CantatLoray}, \S2.4; see also \cite{BenedettoGoldman}, \S3C, and \cite{LubotzkyMagid}, \S5), and hence are most naturally viewed as symmetries \textit{between} the surfaces $\mathcal{S}_{A,B,C,D}$ and $\mathcal{S}_{-A,-B,C,D}$, which become automorphisms of a single surface only for special parameters.
\end{remark}

We shall see in Section \ref{Comparison} below that for nondegenerate parameters, the distinction between $\Gamma$ and $\Gamma'$ is not so important. First, the extra automorphisms cannot merge the small orbits into the large orbit, as they only permute equivalent, and hence equal-size, $\Gamma$-invariant sets. Second, transitivity of $\Gamma'$ on the complement of the orbits of bounded size implies transitivity of $\Gamma$ on the same (Theorem \ref{comparisonOfGroupsThm}), and more generally transitivity of $\Gamma'$ on the large connected component $\mathcal{C}$ in Theorem \ref{th:EndgamePlusMiddlegame} implies transitivity of $\Gamma$ on the same (Proposition \ref{prop:ComparisonLargeComponent}). Similar results hold in the degenerate case when there are only two large orbits (Lemma \ref{lemma:Degenerate2OrbitsComparisonOfGroups} and the proof of Lemma \ref{lemma:TwoPieceS1CageConnectivityAZero}). However, we shall see in Section \ref{ObstructionsSection} that when there are four large orbits, $\Gamma'$ merges three of them.

With equivalence defined, we can compactly state our degeneracy condition. We will motivate this condition below in Section \ref{GeneralizedClusterAlgsSection} by discussing how it relates to the special families (\ref{MarkoffPlusk}) and (\ref{GyodaMatsushitaEq}).
\begin{definition}\label{nondegeneracyCondition}
    We say a quadruple of integer parameters is \textit{degenerate} if it is equivalent to a quadruple $(A, B, C, D)$ with $A = B$ and $4D + A^2 = 8C + 16$; otherwise, we say that the quadruple of integer parameters is nondegenerate.
    
    Similarly, for a fixed prime number $p$, we say that a quadruple is degenerate mod $p$ if it is equivalent to a quadruple $(A, B, C, D)$ with $A \equiv B\pmod{p}$ and $4D + A^2 \equiv 8C + 16\pmod{p}$.
\end{definition}
\subsection{Outline of the Paper}\label{subsect:Outline}
For the remainder of this introduction, we outline the rest of the paper, its primary technical accomplishments, and directions for future research. In Section \ref{FiniteOrbitsSection}, we outline the results of Lisovyy and Tykhyy \cite{LT}, which allow us to classify and remove small orbits from $\mathcal{S}_{A, B, C, D}.$ In Section \ref{GroupTheory}, we discuss the applications of our results to the combinatorial group theory of $\text{SL}_2(\mathbb{F}_p)$ via (\ref{MarkoffPlusk}), and how our result is strengthened by the results of Chen in \cite{ChenMarkPlusk}. In Section \ref{GeneralizedClusterAlgsSection}, we discuss how our nondegeneracy condition specializes to the equations (\ref{MarkoffPlusk}) and (\ref{GyodaMatsushitaEq}). 

We then move on to the proof of our main result following the techniques of Bourgain, Gamburd, and Sarnak. Sections \ref{ConicsSection}--\ref{openingSection} establish the nondegenerate case, and Sections \ref{ObstructionsSection} and \ref{ModificationsDegenerate} establish the degenerate case. In Section \ref{ConicsSection}, we carefully analyze the action of certain subgroups of $\Gamma$ on slices of $\mathcal{S}_{A, B, C, D}$ by planes parallel to the coordinate planes. In this section, an important technical difficulty arises, which we describe here in the case of a plane $X = x_0$. Unlike in the setting of the Markoff equation, the group generated by the map $V_3\circ V_2$, the analogue of the rotation map in \cite{BGS}, does not act transitively on $C_1(x_0) \coloneqq \mathcal{S}_{A, B, C, D}\cap V(X - x_0)$ for $x_0 \neq \pm2$, even when this map has maximal order. In fact, for $x_0 \neq \pm 2$, when $V_3\circ V_2$ has maximal order, the subgroup $\langle V_2, V_3\rangle$ acts transitively on $C_1(x_0)$ only about half of the time, and breaks $C_1(x_0)$ up into two orbits about half of the time (see Proposition \ref{prop:SingleOrbitAndKappa}). One of the key steps in Section \ref{ConicsSection} is the use of some combinatorics to find polynomial conditions guaranteeing transitivity (Proposition \ref{prop:SingleOrbitAndKappa}). In Section \ref{EndgameSection}, we construct a large connected component, containing essentially all points whose orbits under either $V_3\circ V_2$, $V_3\circ V_1,$ or $V_2\circ V_1$ have order at least $p^{\frac{1}{2} + \delta}$, following the strategy of the so-called endgame of \cite{BGS}, i.e., applying Weil's bound together with sieving. This section is more complicated than the endgame in \cite{BGS} because of the more complicated polynomial conditions coming out of Section \ref{ConicsSection}, and it is in proving irreducibility of the curves arising from these conditions that we must demand nondegeneracy. In the course of our endgame, it is sometimes quite helpful to include the extra automorphisms in $\Gamma'$; in Section \ref{Comparison} we show that we can safely do this without collapsing multiple large orbits into one.

The so-called middlegame of Bourgain, Gamburd, and Sarnak goes through with essentially no changes in Section \ref{middlegameSection}, connecting points of order at least $p^\epsilon$ to the large connected component of Section \ref{EndgameSection}, and thereby establishing Theorem \ref{th:EndgamePlusMiddlegame}. In Subsection \ref{subsect:AlgNTBound}, we derive the parameter-dependent lower bound of Theorem \ref{th:LogarithmicLowerBoundFirstAppearance} following the strategy of Bourgain, Gamburd, and Sarnak with only minimal changes. In order to achieve parameter-independence for the zero-density set of excluded primes, however, we need a different technique (``spreading out''); this is performed in Subsection \ref{subsect:SpreadingOut}. Subsection \ref{subsect:ProofCompletion} completes the proof of our main theorem, and describes the different sets of exceptional primes which arise depending on whether one uses the results of Subsection \ref{subsect:AlgNTBound} or Subsection \ref{subsect:SpreadingOut}.

In Section \ref{ObstructionsSection}, we establish the existence of multiple large orbits under the action of $\Gamma$ for degenerate parameters $A, B, C,$ and $D$, and investigate how some of them may sometimes combine under the action of the larger $\Gamma'.$ In Section \ref{ModificationsDegenerate} we modify the strategy used to prove our main theorem to apply to the candidate orbits we identify in Section \ref{ObstructionsSection}. At a high level, the opening and middlegame go through only very minor changes (see Lemma \ref{lemma:DegenerateInheritedReductions}), but the action on the conic sections is somewhat different, and as a result several modifications to the strategy of Section \ref{EndgameSection} are needed. 

Appendix \ref{appendix:Delta} provides the formula for a complicated polynomial $\Delta$ which arises at various points in our analysis, and Appendix \ref{appendix:Macaulay2} contains the Macaulay2 code verifying certain algebraic identities used in the proof.

The primary technical achievement of this paper is contained in Sections \ref{ConicsSection} and \ref{EndgameSection}, in which we overcome the obstacle mentioned above: there are almost no conic sections upon which the iterates of any single automorphism act transitively. Important secondary achievements are the new path to the opening used to achieve parameter-independence of our density zero set (Subsection \ref{subsect:SpreadingOut}), our discovery of sharp degeneracy conditions, and the modification of the proof of our main theorem to count large orbits for degenerate surfaces in the family in Section \ref{ModificationsDegenerate}. For this modification, we ultimately need an 11-case breakdown to handle all of the various modifications of the endgame arguments needed in the degenerate case. Due to our detailed study of conic sections in Section \ref{ConicsSection}, we also obtain a slight simplification, relative to the corresponding argument in \cite{pDivisibilityClusterAlg}, of the argument establishing the existence of multiple large orbits. Two other important accomplishments of somewhat lesser stature appear in Section \ref{Comparison} and Subsection \ref{subsect:AlgNTBound}: In Section \ref{Comparison} we show that action by $\Gamma'$ and action by $\Gamma$ are equivalent (at least in the nondegenerate case, and after excluding the small orbits), which has been unwritten folklore for some time; in Subsection \ref{subsect:AlgNTBound} we show how the unconditional lower bound on orbits obtained by Bourgain, Gamburd, and Sarnak in the classical Markoff setting applies in our more general context after removing the finite $\Gamma$-orbits. The author hopes that the discussions of other related varieties and their relations to various areas of mathematics in this introduction and Sections \ref{FiniteOrbitsSection}, \ref{GroupTheory}, and \ref{GeneralizedClusterAlgsSection} can serve as a useful guide to the literature on the many manifestations of Markoff-type varieties.

It would be interesting to study the action of $\Gamma$ on $\mathcal{S}_{A, B, C, D}(\mathbb{Z}/N\mathbb{Z})$ for $N$ squarefree, as in \cite{MeiriPuder}, $N = p^n$ as in \cite{BGSInPrep}, or even for general composite $N$, as well as the action on the $p$-adic points $\mathcal{S}_{A, B, C, D}(\mathbb{Z}_p)$ as in \cite{JangMinimality}, which is essentially equivalent to the study of points over $\mathbb{Z}/p^n\mathbb{Z}.$ See also \cite{CantatJang} for related work surrounding stationary measures in the $p$-adic context, and \cite{GolsefidyTamam} for results of a similar spirit in the much more general context of mapping class group actions on character varieties, with special focus on the situation over a complete archimedean or nonarchimedean field. In a much more difficult direction, Bourgain, Gamburd, and Sarnak conjecture that the Markoff graphs they define in \cite{BGS} form a family of expander graphs; it would certainly be interesting to investigate whether the same is true of the graphs which arise from other surfaces in our family.

\section{Finite Orbits in Characteristic Zero}\label{FiniteOrbitsSection}
Depending on the choice of parameters $A,$ $B,$ $C$, and $D$, there may be small orbits which arise in any field, or at least any field containing the roots of certain polynomials. These come in four infinite families, together with 45 exceptional orbits, classified by Lisovyy and Tykhyy in \cite{LT}. The simplest such family is a point fixed by each of $V_1, V_2,$ and $V_3$. Unlike the other infinite families, we record the precise conditions in a proposition, both because they are slightly more complicated and because it will be necessary to reference them later in Section \ref{GeneralizedClusterAlgsSection}:
\begin{proposition}[\cite{LT}, Lemma 39 part 1]\label{singletonOrbits}
The point $(x, y, z)$ lies on $\mathcal{S}_{A, B, C, D}$ and is fixed by each of $V_1,$ $V_2,$ and $V_3$ if and only if the parameters are given by:
$$A = 2x - yz$$
$$B = 2y - xz$$
$$C = 2z - xy$$
$$D = 2xyz - x^2 - y^2 - z^2$$
\end{proposition}
The orbits arising from the above will be known as \textit{Type I orbits}, following the terminology of \cite{LT}.
Similarly, there may also be orbits of size $2, 3,$ or $4$, described in the other parts of \cite{LT}, Lemma 39: 

\begin{itemize}
    \item If $B = C = 0$ and the polynomial $X^2 - AX - D$ has distinct roots $r_1$ and $r_2$ in $\mathbb{F}_p$, then $(r_1, 0, 0)$ and $(r_2, 0, 0)$ are solutions to (\ref{MainEq}) which are mapped to each other by $V_1$ and fixed by $V_2$ and $V_3$, so they form a $\Gamma$-orbit of size 2. In the case $A^2 + 4D = 0$, we have a triple fixed point, already classified above. If $A^2 + 4D$ is a nonzero perfect square, then there will be an orbit of size 2 in $\mathcal{S}_{A, 0, 0, D}(\mathbb{F}_p)$ for all primes $p$ not dividing $A^2 + 4D$. If $A^2 + 4D$ is not a perfect square, then nonetheless we shall have
    $$\left(\frac{A^2 + 4D}{p}\right) = 1$$
    for essentially half of the primes $p$, and for each such prime, there will be an orbit of size~2 in
$\mathcal{S}_{A,0,0,D}(\mathbb{F}_p)$, while for other primes there will not be such an orbit.
If the parameters $(A^{\prime},B^{\prime},C^{\prime},D^{\prime})$ are equivalent to $(A,0,0,D)$,
then we obtain an orbit of size 2 equivalent to the one above for at least a density-$\frac{1}{2}$ set of primes.
These orbits will be known henceforth as \textit{Type~II orbits}.
    \item If $D = -1,$ $A = -2$, and $B = C = k$, then rewriting (\ref{MainEq}) as
\begin{equation}(X^2 + 2X + 1) + Y(Y-k) + Z(Z-k)= XYZ\end{equation}
and keeping in mind that the Vieta involutions exchange roots of the quadratic polynomial in each variable, it is clear that $(-1, 0, 0)$ is fixed by $V_1$, mapped to $(-1, k, 0)$ by $V_2$, and mapped to $(-1, 0, k)$ by $V_3$, and further that $(-1, k, 0)$ and $(-1, 0, k)$ remain fixed by $V_1$. Additionally, one can check that $(-1, k, 0)$ is fixed by $V_3$ and that $(-1, 0, k)$ is fixed by $V_2$, so that, for $k\neq 0$, $\{(-1, 0, 0), (-1, k, 0), (-1, 0, k)\}$ is a $\Gamma$-orbit of size $3$. Of course, if $(A, B, C, D)$ is equivalent to $(-2, k, k, -1),$ we obtain an equivalent $\Gamma$-orbit of size 3. These orbits will be known henceforth as \textit{Type III orbits}.

\item Finally, if $A = B = C = k$ and $D = 4 + 3k$, then rewriting (\ref{MainEq}) as
\begin{equation}\label{Type4Orbits} (X+1)(X-k-2) + (Y+1)(Y-k-2) + (Z+1)(Z-k-2) = XYZ - X  - Y - Z - 2\end{equation}
one can see as in the previous case that one has, for $k\neq -3$, an orbit of size 4 consisting of 
$$\{(-1, -1, -1), (k+2, -1, -1), (-1, k+2, -1), (-1, -1, k+2)\}.$$
Indeed, the right-hand side of (\ref{Type4Orbits}) vanishes at all of these points, telling us that $V_1$, $V_2$, and $V_3$ map $(-1, -1, -1)$ to $(k+2, -1, -1),$ $(-1, k+2, -1),$ and $(-1, -1, k+2)$ respectively. One can then manually check that $(k+2, -1, -1)$ is fixed by $V_2$ and $V_3$, or use the conditions for a point to be fixed by two Vieta involutions in Proposition \ref{DoubleFixedPoints}; by symmetry $(-1, k+2, -1)$ and $(-1, -1, k+2)$ are also double fixed points. Again, for equivalent parameters, we get equivalent orbits of size 4. These orbits will be known henceforth as \textit{Type IV orbits}.
\end{itemize}

Lisovyy and Tykhyy prove that these are the only infinite families of finite orbits, but that for certain special values of the parameters $A,$ $B,$ $C,$ and $D$ there are some additional finite orbits. Depending on the particular parameters for which these orbits occur, and on the coordinates of the solutions in these orbits, we can get different arithmetic behavior. We illustrate with three examples. First, if $(A, B, C, D) = (0, -1, -1, 0),$ then $(1, -1, -1)$ is a solution to (\ref{MainEq}), and its orbit has size $5$, as depicted in Figure \ref{ExceptionalFiniteOrbit1}. As the root solution $(1, -1, -1)$ has integral coordinates, so do all of the solutions in its orbit, and thus this orbit will appear in every $\mathcal{S}_{0, -1, -1, 0}(\mathbb{F}_p).$

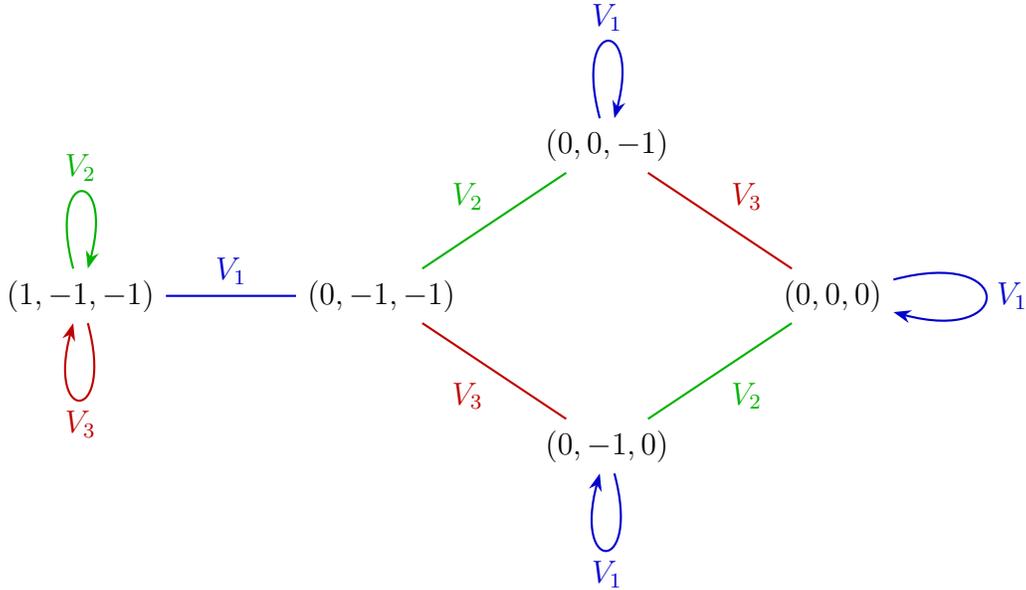
\begin{figure}
    \centering
\begin{tikzpicture}[
  >=Stealth,
  node distance=3cm,
  line width=0.8pt,
  v1/.style={blue!80!black},
  v2/.style={green!70!black},
  v3/.style={red!75!black},
  bigloop/.style={looseness=10, min distance=14mm}
]

\node (A) at (-6,0)  {$\left(1,-1,-1\right)$};
\node (B) at (-2,0)  {$\left(0,-1,-1\right)$};

\node (C) at (1,2)   {$\left(0,0,-1\right)$};
\node (D) at (4,0)   {$\left(0,0,0\right)$};
\node (E) at (1,-2)  {$\left(0,-1,0\right)$};

\draw[v2, ->] (A) to[loop above, bigloop] node {$V_2$} (A);
\draw[v3, ->] (A) to[loop below, bigloop] node {$V_3$} (A);

\draw[v1, ->] (C) to[loop above, bigloop] node {$V_1$} (C);
\draw[v1, ->] (E) to[loop below, bigloop] node {$V_1$} (E);

\draw[v1, ->] (D) to[loop right, bigloop] node {$V_1$} (D);

\draw[v1, -] (A) -- node[above] {$V_1$} (B);

\draw[v2, -] (B) -- node[above left] {$V_2$} (C);
\draw[v3, -] (C) -- node[above right] {$V_3$} (D);

\draw[v3, -] (B) -- node[below left] {$V_3$} (E);
\draw[v2, -] (E) -- node[below right] {$V_2$} (D);

\end{tikzpicture}

\caption{Exceptional finite orbit for the parameters $(A, B, C, D) = (0, -1, -1, 0)$ starting from $(1, -1, -1)$}
\label{ExceptionalFiniteOrbit1}    
\end{figure}

Second, for $(A, B, C, D) = (0, 0, 0, 3),$ we have that $(1, \sqrt{2}, \sqrt{2})$ is a solution to (\ref{MainEq}), and its orbit has size $12$, as depicted in Figure \ref{ExceptionalFiniteOrbit2}; note that as the coordinates of the root solution are not rational integers, no point in the orbit has coordinates that are all rational integers. It follows that for primes $p$ such that 
$$\left(\frac{2}{p}\right) = 1,$$
equivalently, for $p\equiv \pm 1\pmod{8}$, there is a $\Gamma$-orbit on $\mathcal{S}_{(0, 0, 0, 3)}(\mathbb{F}_p)$ of size $12$. However, modulo other primes, this orbit will not occur.

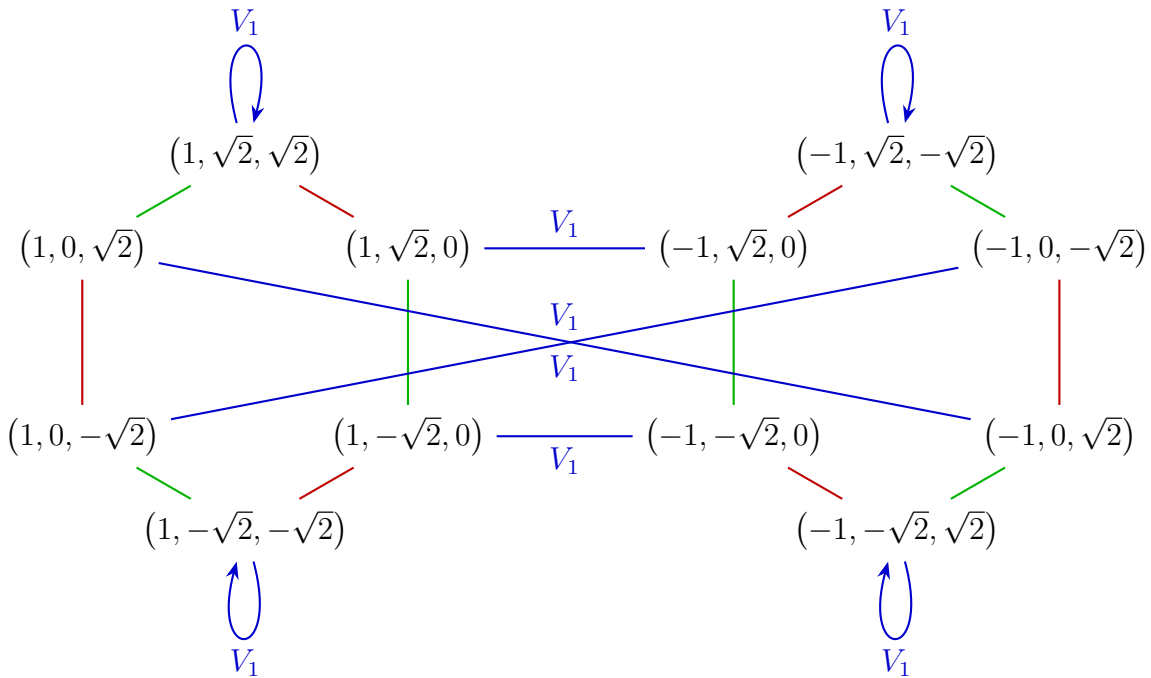
\begin{figure}
    \centering
\begin{tikzpicture}[
  >=Stealth,
  line width=0.8pt,
  v1/.style={blue!80!black},
  v2/.style={green!70!black},
  v3/.style={red!75!black},
  bigloop/.style={looseness=10, min distance=14mm}
]

\def\R{2.5}   
\def\H{1.732*\R}    

\node (L1) at ({-\H},{ \R}) {$\left(1,\sqrt2,\sqrt2\right)$};
\node (L2) at ({-\H-\R*0.866},{ \R*0.5}) {$\left(1,0,\sqrt2\right)$};
\node (L3) at ({-\H-\R*0.866},{-\R*0.5}) {$\left(1,0,-\sqrt2\right)$};
\node (L4) at ({-\H},{-\R}) {$\left(1,-\sqrt2,-\sqrt2\right)$};
\node (L5) at ({-\H+\R*0.866},{-\R*0.5}) {$\left(1,-\sqrt2,0\right)$};
\node (L6) at ({-\H+\R*0.866},{ \R*0.5}) {$\left(1,\sqrt2,0\right)$};

\node (R1) at ({\H},{ \R}) {$\left(-1,\sqrt2,-\sqrt2\right)$};
\node (R2) at ({\H+\R*0.866},{ \R*0.5}) {$\left(-1,0,-\sqrt2\right)$};
\node (R3) at ({\H+\R*0.866},{-\R*0.5}) {$\left(-1,0,\sqrt2\right)$};
\node (R4) at ({\H},{-\R}) {$\left(-1,-\sqrt2,\sqrt2\right)$};
\node (R5) at ({\H-\R*0.866},{-\R*0.5}) {$\left(-1,-\sqrt2,0\right)$};
\node (R6) at ({\H-\R*0.866},{ \R*0.5}) {$\left(-1,\sqrt2,0\right)$};

\draw[v1,->] (L1) to[loop above, bigloop] node {$V_1$} (L1);
\draw[v1,->] (L4) to[loop below, bigloop] node {$V_1$} (L4);

\draw[v1,->] (R1) to[loop above, bigloop] node {$V_1$} (R1);
\draw[v1,->] (R4) to[loop below, bigloop] node {$V_1$} (R4);

\draw[v2,-] (L1)-- node[below right] {$V_2$} (L2);
\draw[v3,-] (L2)-- node[right] {$V_3$} (L3);
\draw[v2,-] (L3)-- node[above right] {$V_2$} (L4);
\draw[v3,-] (L4)-- node[above left] {$V_3$} (L5);
\draw[v2,-] (L5)-- node[left] {$V_2$} (L6);
\draw[v3,-] (L6)-- node[below left] {$V_3$} (L1);

\draw[v2,-] (R1)-- node[below left] {$V_2$} (R2);
\draw[v3,-] (R2)-- node[left] {$V_3$} (R3);
\draw[v2,-] (R3)-- node[above left] {$V_2$} (R4);
\draw[v3,-] (R4)-- node[above right] {$V_3$} (R5);
\draw[v2,-] (R5)-- node[right] {$V_2$} (R6);
\draw[v3,-] (R6)-- node[below right] {$V_3$} (R1);

\draw[v1,-] (L6) -- node[above] {$V_1$} (R6);
\draw[v1,-] (L5) -- node[below] {$V_1$} (R5);

\draw[v1,-] (L2) -- node[above] {$V_1$} (R3);
\draw[v1,-] (L3) -- node[below] {$V_1$} (R2);

\end{tikzpicture}
    \caption{Exceptional finite orbit for the parameters $(A, B, C, D) = (0, 0, 0, 3)$ starting from $(1, \sqrt{2}, \sqrt{2})$}
    \label{ExceptionalFiniteOrbit2}
\end{figure}

Finally, for $(A, B, C, D) = (\sqrt{2}, 0, 0, 1),$ there is an orbit of size 6 containing the point $(\sqrt{2}, -1, -\sqrt{2})$, depicted in Figure \ref{ExceptionalFiniteOrbit3}. A fixed choice of integers $(A, B, C, D)$ can be equivalent, after reduction modulo $p$, to parameters of the form $(\sqrt{2}, 0, 0, 1)$ only for finitely many primes, namely those dividing $A^2 - 2$, $B$, $C$, and $D-1$, up to the equivalences of Definition \ref{def:EquivalentParams}. However, as we wish to allow the parameters of our surfaces to vary while demanding the set of excluded primes does not, we must remove this third type of orbit as well, since for any prime modulo which $2$ is a quadratic residue there will be a choice of $(A, B, C, D)$ over which this orbit of size 6 appears.

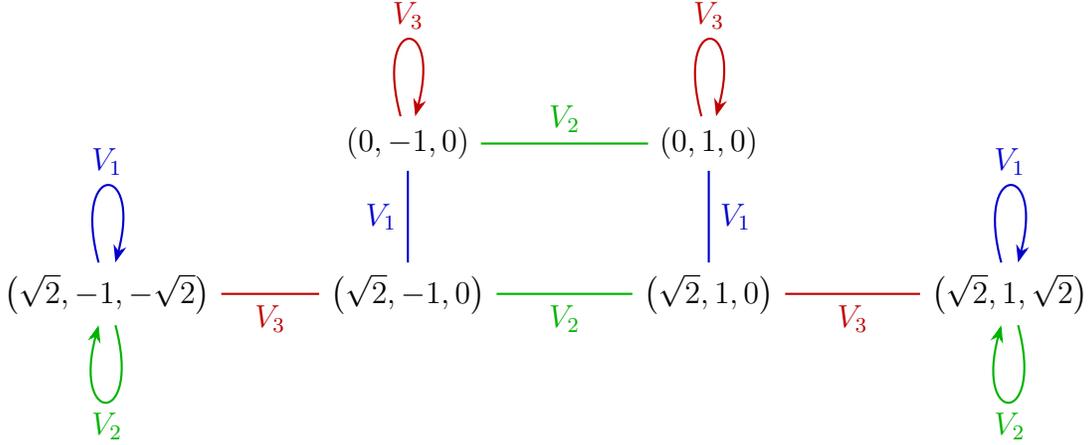
\begin{figure}
    \centering
    \begin{tikzpicture}[
  >=Stealth,
  line width=0.8pt,
  v1/.style={blue!80!black},
  v2/.style={green!70!black},
  v3/.style={red!75!black},
  bigloop/.style={looseness=10, min distance=14mm}
]

\node (A) at (-6,0) {$\left(\sqrt2,-1,-\sqrt2\right)$};
\node (B) at (-2,0) {$\left(\sqrt2,-1,0\right)$};
\node (C) at ( 2,0) {$\left(\sqrt2,1,0\right)$};
\node (D) at ( 6,0) {$\left(\sqrt2,1,\sqrt2\right)$};

\node (E) at (-2,2) {$\left(0,-1,0\right)$};
\node (F) at ( 2,2) {$\left(0,1,0\right)$};

\draw[v1,->] (A) to[loop above, bigloop] node {$V_1$} (A);
\draw[v2,->] (A) to[loop below, bigloop] node {$V_2$} (A);

\draw[v3,->] (E) to[loop above, bigloop] node {$V_3$} (E);
\draw[v3,->] (F) to[loop above, bigloop] node {$V_3$} (F);

\draw[v1,->] (D) to[loop above, bigloop] node {$V_1$} (D);
\draw[v2,->] (D) to[loop below, bigloop] node {$V_2$} (D);

\draw[v3,-] (A) -- node[below] {$V_3$} (B);
\draw[v2,-] (B) -- node[below] {$V_2$} (C);
\draw[v3,-] (C) -- node[below] {$V_3$} (D);

\draw[v1,-] (B) -- node[left] {$V_1$} (E);
\draw[v1,-] (C) -- node[right] {$V_1$} (F);

\draw[v2,-] (E) -- node[above] {$V_2$} (F);

\end{tikzpicture}
    \caption{Exceptional finite orbit for the parameters $(A, B, C, D) = (\sqrt{2}, 0, 0, 1)$ starting from $(\sqrt{2}, -1, -\sqrt{2})$}
    \label{ExceptionalFiniteOrbit3}
\end{figure}

All exceptional parameters and corresponding exceptional finite orbits have been classified in \cite{LT}, Table 4, up to equivalence. There are exactly 45 such orbits, and the examples given above are orbits number $1,$ $20,$ and $4$, respectively. For ease of reference we record this as a theorem.

\begin{theorem}[\cite{LT}, Theorem 1]\label{FiniteOrbitClassification}
For $(A, B, C, D) \neq (0, 0, 0, 4)$, every finite $\Gamma$-orbit on $\mathcal{S}_{A, B, C, D}(\mathbb{C})$ is either a Type I, Type II, Type III, or Type IV orbit, or is one of the 45 exceptional orbits listed in Table 4 of \cite{LT}. In our notation, the parameters of surfaces with exceptional orbits are given by
$$(A, B, C, D) = (-\omega_X, -\omega_Y, -\omega_Z, 4-\omega_4),$$
and each exceptional orbit has a point given by
$$(x, y, z) = (-2\cos(\pi r_X), -2\cos(\pi r_Y), -2\cos(\pi r_Z))$$

where $\omega_X, \omega_Y, \omega_Z, 4-\omega_4, r_X, r_Y,$ and $r_Z$ are the numbers given in the columns of \cite{LT}, Table 4.
\end{theorem}
\begin{remark}[Interpreting the cosines over $\mathbb{F}_p$]
As $r_X, r_Y,$ and $r_Z$ are rational numbers, the points in each exceptional orbit are algebraic integers. To interpret them over finite fields, let $K$ be the number field generated by the coordinates of the points of the orbit. For any prime $P$ of $K$ lying above a rational prime $p$ with residual degree one, reduction modulo $P$ gives a $\Gamma$-orbit over $\mathbb{F}_p$. Away from the finitely many primes $P$ for which two distinct points in the complex orbit become congruent modulo $P$, this orbit is isomorphic to the original orbit; in the remaining cases, the graph of the reduced orbit is a quotient of the graph of the complex orbit. We regard these quotient orbits as also arising from finite orbits over $\mathbb{C}$.
\end{remark}
\begin{remark}[The Cayley parameters]\label{rmk:CayleyParams}
    Recall that the case of $(A, B, C, D) = (0, 0, 0, 4)$ is the famously degenerate Cayley cubic. For this particular cubic surface, the Vieta involutions linearize under a change of variables, with the result that there are many orbits, all of them fairly straightforward to describe. Over $\mathbb{C}$, finite orbits are also classified in Theorem 1 of \cite{LT}; over $\mathbb{F}_p$, they are classified in Theorem 1.1 of \cite{deCourcyIrelandLee}.
\end{remark}
\begin{notation}\label{notation:FiniteOrbits}
    We let $\mathcal{E}(p) = \mathcal{E}_{A, B, C, D}(p)$ denote the orbits in $\mathcal{S}_{A, B, C, D}(\mathbb{F}_p)$ which arise from finite $\Gamma$-orbits over $\mathbb{C},$ that is, the collection of Type I, II, III, and IV orbits, together with any exceptional orbits that arise over $\mathbb{F}_p$ as in the discussion preceding Theorem \ref{FiniteOrbitClassification}, and let $E(p) = E_{A, B, C, D}(p) = \bigcup\limits_{\mathcal{O}\in\mathcal{E}(p)}\mathcal{O}$. Thus $\mathcal{E}(p)$ is a collection of orbits, while $E(p)$ is their union, a collection of excluded points not covered by our transitivity results. Finally, we let $\mathcal{S}^*_{A, B, C, D}(p) = \mathcal{S}_{A, B, C, D}(\mathbb{F}_p) \setminus E(p).$
\end{notation}
With this notation fixed, we can restate our main theorem more compactly as
\begin{theorem}\label{th:mainRestated}
    For all primes $p$ in a density one set, and any choice of integer parameters $(A, B, C, D),$ if $(A, B, C, D)$ is nondegenerate modulo $p$, then $\Gamma$ acts transitively on $\mathcal{S}^*_{A, B, C, D}(p)$.
\end{theorem}
\section{Applications to Group Theory}\label{GroupTheory}
An important special case of the equations considered in this paper is the equation (\ref{MarkoffPlusk}), namely

$$X^2 + Y^2 + Z^2 = XYZ + k,$$

which arises naturally in the combinatorial group theory of $\text{SL}_2(\mathbb{F}_p)$ as follows. Let $(M, N)$ and $(M', N')$ be two generating pairs of a group $G$. Each pair determines a surjective map $F_2\rightarrow G$, denoted by $\phi_1$ and $\phi_2$, respectively, defined by
$$\phi_1(a) = M,\; \phi_1(b) = N;$$
$$\phi_2(a) = M',\; \phi_2(b) = N'.$$
We say that $(M, N)$ and $(M', N')$ are \textit{Nielsen equivalent} if there is a $\rho\in\text{Aut}(F_2)$ such that 
$$\phi_2 = \phi_1\circ\rho.$$

A problem of interest in computational and combinatorial group theory is to determine the Nielsen equivalence classes of generating pairs of elements in important groups $G$. An important invariant of Nielsen equivalence, known, after its discoverer, as the \textit{Higman invariant}, is the so-called \textit{extended conjugacy class} of the commutator $[M, N]$, given by the union of the conjugacy classes of $[M, N]$ and $[M, N]^{-1} = [N, M].$

In the case $G = \text{SL}_2(\mathbb{F}_q)$ where $q$ is a prime power, McCullough and Wanderley conjectured in \cite{McCulloughWanderley} that the Higman invariant fully classifies generating pairs---if two generating pairs of $\text{SL}_2(\mathbb{F}_q)$ have the same Higman invariant, then they are Nielsen equivalent. They call this the \textit{Classification Conjecture}. Via the Fricke invariant, they introduce a weaker conjecture, the $Q$-Classification Conjecture, describing orbits in the solution set to (\ref{MarkoffPlusk}) under the action of $\Gamma$; see Sections 5, 7, and 8 of \cite{McCulloughWanderley} for details. In the special case of a prime field, their $Q$-Classification conjecture states that the solutions to (\ref{MarkoffPlusk}) over $\mathbb{F}_p$ which are not reductions of finite orbits over $\mathbb{C}$ in the sense of Section \ref{FiniteOrbitsSection} form a single orbit under the action of $\Gamma$. Building on their work and that of Campos-Vargas \cite{CamposVargas}, Martin showed established the following converse theorem:
\begin{theorem}[\cite{MartinBigPaper}, Theorem 1.5]
Suppose the $Q$-Classification conjecture holds over $\mathbb{F}_p$, i.e., for all $k\in\mathbb{F}_p$ with $k\neq 4$ the solutions to (\ref{MarkoffPlusk}) over $\mathbb{F}_p$ which are not reductions of finite orbits over $\mathbb{C}$ form a single orbit under the action of $\Gamma.$ Then the Higman invariant fully classifies Nielsen equivalence classes of generating pairs of $\text{SL}_2(\mathbb{F}_p)$.
\end{theorem}
In his paper, Martin also introduced an entirely new set of tools into the study of Markoff-type varieties, which enabled him to prove transitivity of the action of $\Gamma$ on $\mathcal{S}_{0, 0, 0, k}^*(p)$ uniformly in $k$ for all $p$ such that $p^2-1 \not\equiv 0 \pmod{N}$ by running a particular finite computation. Running this computation for various choices of $N$, he proved in particular that the Classification Conjecture holds over $\mathbb{F}_p$ for $> 99.99\%$ of all primes $p$. While Martin's technique cannot achieve a density-one result with finite computation, it is insensitive to the size of $p$, unlike the techniques of Bourgain, Gamburd, and Sarnak \cite{BGS} and the techniques of this paper, and thus establishes connectivity of mod $p$ Markoff graphs for many cases which had previously remained unproven.

By specializing our main theorem to the subfamily (\ref{MarkoffPlusk}) and applying Martin's converse theorem our techniques show that the Classification Conjecture holds over prime finite fields for a density-one set of primes. This result may be further strengthened using the following result of Chen \cite{ChenMarkPlusk}:
\begin{theorem}[\cite{ChenMarkPlusk}, consequence of Theorem 1.1]\label{th:MarkoffPluskLowerBound}
    Any $\Gamma$-orbit $\mathcal{O}$ in $\mathcal{S}_{0, 0, 0, k}^*(p)$ satisfies $|\mathcal{O}|\geq p.$
\end{theorem}

By combining this $k$-uniform lower bound with our Theorem \ref{th:EndgamePlusMiddlegame}, we can upgrade our main theorem to hold for all sufficiently large primes, giving the first case of Theorem \ref{th:MainSufficientlyLargeAllCases}.

\begin{theorem}\label{th:MarkoffPluskAllLargep}
    For all sufficiently large $p$, $\Gamma$ acts transitively on $\mathcal{S}_{0, 0, 0, k}^*(p)$ for all $k\not\equiv 4\pmod {p}.$ Consequently, the Higman invariant fully classifies Nielsen equivalence classes of generating pairs of $\text{SL}_2(\mathbb{F}_p)$ when $p$ is sufficiently large.
\end{theorem}

We conclude the section with two brief remarks, both drawing on \cite{McCulloughWanderley}. First, for most values of $Tr([A, B])$, said trace completely determines the Higman invariant, and in all cases there are at most two extended conjugacy classes of the commutator with a given trace. Second, the Classification Conjecture also implies a classification of $T$-systems (McCullough and Wanderley's ``$T$-Classification Conjecture'').

\section{Generalized Cluster Algebras and the Nondegeneracy Condition}\label{GeneralizedClusterAlgsSection}
Another important special case of our (\ref{MainEq}) is the equation 
$$x_1^2 + x_2^2 + x_3^2 + a_1x_2x_3 + a_2x_1x_3 + a_3x_1x_2 = (3 + a_1 + a_2 + a_3)x_1x_2x_3$$
also given above as (\ref{GyodaMatsushitaEq}), which was introduced in \cite{GyodaMatsushita} by Gyoda and Matsushita. They prove that the Vieta involutions act transitively on solutions to (\ref{GyodaMatsushitaEq}) in positive integers, and in fact that one can organize these solutions in a tree with root solution $(1, 1, 1)$, as with the Markoff equation. Gyoda and Matsushita show that this tree can be interpreted in terms of certain Generalized Cluster Algebras, much like the Markoff tree can be interpreted in terms of certain Cluster Algebras; see \cite{GyodaMatsushita} for the definitions of these terms. The uniform presence in this family of integer solutions stands in marked contrast to the family (\ref{MarkoffPlusk}), for which existence of integer solutions is a remarkably subtle problem---for instance, Ghosh and Sarnak prove in \cite{HassePrinciple} that the equation (\ref{MarkoffPlusk}) obeys the Hasse principle for almost all values of $k$, but fails it for infinitely many values of $k$ as well.

In \cite{pDivisibilityClusterAlg}, de Courcy-Ireland, Litman, and Mizuno study the dynamics of equation (\ref{GyodaMatsushitaEq}) over $\mathbb{F}_p$, proving several important results under the conditions that $(3+ a_1 + a_2 + a_3) \not\equiv 0\pmod{p}$ and either $a_i^2 \not \equiv 4 \pmod{p}$ for $i = 1, 2, 3$, or for some $i$ we have $a_i^2 \equiv 4 \pmod{p}$ and $2a_{i-1} \equiv a_ia_{i+1} \pmod{p}$. The main result of their paper can be stated as follows:

\begin{theorem}[\cite{pDivisibilityClusterAlg}, Theorem 1.1]\label{pDivisibilityResult}
    Suppose that $3 + a_1 + a_2 + a_3 \not\equiv 0\pmod{p}$, and that either $a_i^2 \not \equiv 4 \pmod{p}$ for $i = 1, 2, 3$ or, for some $i$, $a_i^2 \equiv 4 \pmod{p}$ and $2a_{i-1} \equiv a_ia_{i+1} \pmod{p}$. Then each orbit in the $\mathbb{F}_p$-points of the algebraic surface cut out by (\ref{GyodaMatsushitaEq}) other than $\{(0, 0, 0)\}$ has size divisible by $p$.
\end{theorem}

This is an analogue of Chen's theorem for the classical Markoff equation, first proven in \cite{Chen}; the proof in \cite{pDivisibilityClusterAlg} follows the elementary proof of Martin \cite{Martinpdivisibility}.

\begin{remark}[Geometric interpretation of the generalized cluster algebra structure]
    Theorem \ref{pDivisibilityResult} shows that in some sense the arithmetic dynamics of the varieties arising from the generalized cluster algebras are more closely analogous to those of the Markoff equation than the dynamics of the varieties arising from (\ref{MarkoffPlusk}) and (\ref{MainEq}). We can give this analogy a geometric interpretation by interpreting these varieties as character varieties via the change of variables (\ref{clusterAlgChangeVars}). Points of the character variety of the once punctured torus (\ref{MarkoffPlusk}) correspond to representations of the fundamental group of the once punctured torus, within which points of the Markoff equation (\ref{Markoff}) and the Cayley cubic ($k = 4$ within (\ref{MarkoffPlusk})) correspond respectively to representations where the trace of the matrix corresponding to the loop around the puncture is $-2$ or $2$.
    
    Similarly, points of the character variety of the four-times punctured sphere (\ref{MainEq}) correspond to representations of the fundamental group of the four-times punctured sphere, and when the parameters arise from the Generalized Cluster Algebra as in (\ref{clusterAlgChangeVars}), points correspond to representations where the trace of the matrix corresponding to a loop around one of the punctures is $\pm 2$. Thus, the Markoff-type equation arising from generalized cluster algebras (\ref{GyodaMatsushitaEq}) is to our equation (\ref{MainEq}) as the classical Markoff equation (\ref{Markoff}) is to the more general (\ref{MarkoffPlusk}). We remark that if the traces of the loops around each of the four punctures (taken in an arbitrary order) are $(a, b, c, d)$, then the parameters $(A, B, C, D)$ in (\ref{MainEq}) remain unchanged under $(a, b, c, d)\mapsto (-a, -b, -c, -d),$ so that the parameters $(A, B, C, D)$ do not determine the sign of the trace of the loop around any given puncture, which is why we state our condition in terms of a trace equal to $\pm 2$, not specifically $2$ or $-2$.
\end{remark}

Theorem \ref{pDivisibilityResult} is significant for this paper because, following de Courcy-Ireland, Litman, and Mizuno, when $s = 3+a_1 + a_2 + a_3 \not\equiv 0\pmod{p}$ the change of variables
\begin{equation}\label{clusterAlgChangeVars}u_i = sx_i - a_i\end{equation}
transforms (\ref{GyodaMatsushitaEq}) into 
$$u_1^2 + u_2^2 + u_3^2 = u_1u_2u_3 -(2a_1 + a_2a_3)u_1 -(2a_2 + a_1a_3)u_2 - (2a_3 + a_1a_2)u_3 - (2a_1a_2a_3 + a_1^2 + a_2^2 + a_3^2),$$
which is an equation of the form we consider with 
\begin{equation}\label{ChangeOfParams}
\begin{split}
A &= -2a_1 - a_2a_3\\
B &= -2a_2 - a_1a_3\\
C &= -2a_3 - a_1a_2\\
D &= -2a_1a_2a_3 -a_1^2 - a_2^2 - a_3^2.
\end{split}
\end{equation}

This change of variables is equivariant with respect to the Vieta involutions, so that the study of orbits on (\ref{GyodaMatsushitaEq}) and that of orbits on (\ref{MainEq}) with parameters arising via (\ref{ChangeOfParams}) are equivalent. For varieties arising this way, Theorem \ref{pDivisibilityResult} gives a strong lower bound on the sizes of orbits, just as Chen's recent result \ref{th:MarkoffPluskLowerBound} does for the family (\ref{MarkoffPlusk}). Combining the proof of our main theorem with this lower bound and the interpretation of our nondegeneracy condition described in Proposition \ref{RelationBetweenDegeneracyConditions} below, we obtain the second case of Theorem \ref{th:MainSufficientlyLargeAllCases}:
\begin{theorem}\label{th:MainSuffLargeGyodaMatsushita}
    For all sufficiently large primes $p$ and any integers $a_1, a_2,$ and $a_3,$ if $3 + a_1 + a_2 + a_3 \not\equiv 0\pmod{p}$ and $a_i^2\not\equiv  4\pmod{p}$ for $i = 1, 2, 3$, then the solutions to
    $$x_1^2 + x_2^2 + x_3^2 + a_1x_2x_3 + a_2x_1x_3 + a_3x_1x_2 = (3+a_1 +a_2 + a_3)x_1x_2x_3$$
    other than the trivial solution $(x_1,x_2,x_3) = (0,0,0)$ form a single orbit under the action of the group generated by the three Vieta involutions.

    If instead $3 + a_1 +a_2 + a_3 \not\equiv 0\pmod{p}$, $a_i^2 \equiv 4\pmod{p}$ for some $i$, and $2a_{i-1} \equiv a_ia_{i+1}\pmod{p}$, then the solutions to 
    $$x_1^2 + x_2^2 + x_3^2 + a_1x_2x_3 + a_2x_1x_3 + a_3x_1x_2 = (3+a_1 +a_2 + a_3)x_1x_2x_3$$
    other than the trivial solution $(x_1,x_2,x_3) = (0,0,0)$ form either two (if $a_{i-1}^2 \not\equiv 4\pmod{p}$) or four (if $a_{i-1}^2 \equiv 4\pmod{p}$) orbits under the action of the group generated by the three Vieta involutions.
\end{theorem}

While the $p$-divisibility result holds for some varieties in the family (\ref{GyodaMatsushitaEq}) with an $a_i = \pm 2$, transitivity of the group action fails in these cases. Indeed, de Courcy-Ireland, Litman, and Mizuno also show:

\begin{theorem}[\cite{pDivisibilityClusterAlg}, Theorem 1.2]\label{degeneracyClusterAlg}
    Suppose that $3 + a_1 + a_2 + a_3 \not\equiv 0\pmod{p}$, and that for some $i$ we have that $a_i^2 \equiv 4\pmod{p}$ and $2a_{i-1} \equiv a_ia_{i+1}\pmod{p}$, where indices are taken cyclically mod $3$. Then there are at least two $\Gamma$-orbits in the $\mathbb{F}_p$-points of the algebraic surface cut out by (\ref{GyodaMatsushitaEq}) aside from the trivial orbit $\{(0, 0, 0)\}$. If in addition we have $a_i^2 \equiv 4\pmod{p}$ for all $i$ then there are at least four large $\Gamma$-orbits.
\end{theorem}

This theorem is the starting point for our nondegeneracy condition. In fact, our proof of Theorem \ref{QuadraticObstructions} is based on their proof of this theorem, and, when specialized to the setting of (\ref{GyodaMatsushitaEq}), extends their obstruction to the case where some $a_i^2 \equiv 4\pmod{p}$ but $2a_{i-1} \not\equiv a_{i}a_{i+1}\pmod{p}$. We can explain our condition in these terms as follows:

Suppose that $a_3 = \pm 2$. Changing variables as in (\ref{clusterAlgChangeVars}), we find that
$$A = -2a_1 - a_2a_3 = -2a_1 \mp 2a_2 = \pm(-2a_2 \mp 2a_1) = \pm(-2a_2 - a_1a_3) = \pm B.$$

Similarly, if $a_2 = \pm2$ we get an equation with $A = \pm C$, and if $a_1 = \pm 2$ we get an equation with $B = \pm C$. This is not, however, the only way to get a variety of the form (\ref{MainEq}) from the equation (\ref{GyodaMatsushitaEq}) with such coincidences of parameters. For example, the equation $A = B$ becomes in the $(a_1, a_2, a_3)$-coordinates

$$-2a_1 - a_2a_3 = -2a_2 - a_1a_3,$$
or
$$-2(a_1 - a_2) = -a_3(a_1-a_2).$$

Thus, we can get $A = B$ from $a_3 = 2$ or from $a_1 = a_2$; similarly, we can get $A = -B$ from $a_3 = -2$ or from $a_1 = -a_2,$ $A = \pm C$ from $a_2 = \pm2$ or from $a_1 = \pm a_3,$ and $B = \pm C$ from $a_1 = \pm2$ or from $a_2 = \pm a_3$. This is where our second condition comes in. We present the case $A = B$ and $4D + A^2 = 8C + 16$; equivalent quadruples of parameters go through similarly. If $A = B,$ and the quadruple $(A, B, C, D)$ arises from a triple $(a_1, a_2, a_3)$ under (\ref{clusterAlgChangeVars}), then we have that
\[\begin{split}
    4D + A^2 - 8C - 16 &= (-8a_1a_2a_3 - 4a_1^2 - 4a_2^2 - 4a_3^2) + (4a_1^2 + 4a_1a_2a_3 + a_2^2a_3^2) + (16a_3 + 8a_1a_2) - 16\\
    &= -4a_1a_2a_3 + a_2^2a_3^2 - 4a_2^2 -4a_3^2 + 16a_3  +8a_1a_2  -16.
\end{split}\]
If $a_3 = 2$, then this simplifies as
\[\begin{split}
    4D + A^2 - 8C - 16 &= - 8a_1a_2 + 4a_2^2 - 4a_2^2 - 16 + 32 +8a_1a_2 - 16 = 0,
\end{split}\]
so that we have $4D + A^2 = 8C + 16.$
If on the other hand $a_3\neq 2$, then $A = B$ implies that $a_1 = a_2,$ so this simplifies as
\[\begin{split}
    4D + A^2 - 8C - 16 &= -4a_1^2a_3 + a_1^2a_3^2 - 4a_1^2 - 4a_3^2 + 16a_3 + 8a_1^2 - 16\\
    &= 4a_1^2 - 4a_1^2 a_3 + a_1^2a_3^2 -4a_3^2 + 16a_3 - 16\\
    &=a_1^2(a_3 - 2)^2 - 4(a_3 - 2)^2\\
    &=(a_1^2 - 4)(a_3- 2)^2.
\end{split}\]
Thus, if we know that $(A, B, C, D)$ arises from $(a_1, a_2, a_3)$ via (\ref{clusterAlgChangeVars}), the condition $A = B$ and $4D + A^2 = 8C + 16$ is equivalent to either $a_3 = 2$ or $a_1 = a_2$ together with $a_1 = \pm2$. Either way, we have that some $a_i = \pm 2$, as desired. The same computation, interpreted modulo $p$, proves the corresponding statement for degeneracy modulo $p$. We thus obtain the following result:

\begin{proposition}\label{RelationBetweenDegeneracyConditions}
    Suppose that the quadruple of parameters $(A, B, C, D)$ arises from the triple of parameters $(a_1, a_2, a_3)$ via (\ref{clusterAlgChangeVars}). Then $(A, B, C, D)$ is degenerate (respectively, degenerate mod $p$) if and only if $a_i^2 = 4$ for some $i$ (respectively, $a_i^2 \equiv 4\pmod{p}$ for some $i$). 
\end{proposition}

As a coda to this section, we discuss how the varieties arising from (\ref{GyodaMatsushitaEq}) lie within our family (\ref{MainEq}), and some alternative degeneracy conditions the author considered before finding the sharp conditions used in this paper. The astute reader may have noticed that the polynomials in (\ref{clusterAlgChangeVars}) are nearly the same as those in Proposition \ref{singletonOrbits}. This is not a coincidence: the equation (\ref{GyodaMatsushitaEq}) has a singleton orbit consisting only of the point $(0, 0, 0)$, which is mapped under (\ref{clusterAlgChangeVars}) to the point $(-a_1, -a_2, -a_3)$. In the same way, if $\mathcal{S}_{A, B, C, D}$ has a singleton orbit $(x, y, z)$, then Proposition \ref{singletonOrbits} together with the change of variables (\ref{clusterAlgChangeVars}) applied in reverse implies that $\mathcal{S}_{A, B, C, D}$ arises from (\ref{GyodaMatsushitaEq}) with $a_1 = -x$, $a_2 = -y$, and $a_3 = -z$, giving a dynamical characterization of those quadruples $(A, B, C, D)$ arising from the equation of the generalized cluster algebra---they are exactly the varieties $\mathcal{S}_{A, B, C, D}$ with a triple fixed point.

This dynamical characterization is also a geometric one: $V_1$ fixes $(x,y,z)$ if and only if $2x = A + yz$, which is the vanishing of $\partial F/\partial X$ for $F = X^2+Y^2+Z^2-XYZ-AX-BY-CZ-D$, and similarly for $V_2$ and $V_3$. Hence a point of $\mathcal{S}_{A, B, C, D}$ is fixed by all three Vieta involutions if and only if it is a singular point of the surface (cf. \cite[Lemma 3.3]{CantatLoray}). In particular, the quadruples arising from the generalized cluster algebra equation (\ref{GyodaMatsushitaEq}) over a field $k$ are exactly those for which $\mathcal{S}_{A,B,C,D}$ has a singular point with coordinates in $k$.

With a quick computation in Macaulay2 \cite{Macaulay2}, we can find a necessary algebraic condition for parameters $A, B, C,$ and $D$ to arise from the equation (\ref{GyodaMatsushitaEq}):
\begin{lemma}\label{ZariskiClose}
The Zariski closure of the image of the map $\mathbb{A}^3\rightarrow\mathbb{A}^4$ given by:
$$(a_1, a_2, a_3)\mapsto ( -2a_1 - a_2a_3,  -2a_2 - a_1a_3,  -2a_3 - a_1a_2, -2a_1a_2a_3 -a_1^2 - a_2^2 - a_3^2)$$
is the affine hypersurface cut out by the polynomial $\Delta(A, B, C, D)$ defined in Appendix A.
\end{lemma}

Remarkably, the polynomial $\Delta$ will arise in a completely separate way later on in the endgame (Lemma \ref{DeltaUniversality}) as the discriminant of three distinct univariate polynomials $f_1, f_2,$ and $f_3$ defined immediately prior to Lemma \ref{DeltaUniversality}. This alternate description allows us to prove a partial converse of Lemma \ref{ZariskiClose}, essentially due to Benedetto and Goldman \cite{BenedettoGoldman} and Cantat and Loray \cite{CantatLoray}.

\begin{proposition}\label{prop:SingularIffDelta}
Let $k$ be an algebraically closed field of characteristic different from $2$, and let $A,B,C,D\in k$. The following are equivalent:
\begin{enumerate}
    \item $\mathcal{S}_{A,B,C,D}$ has a singular point;
    \item some point of $\mathcal{S}_{A,B,C,D}(k)$ is fixed by each of $V_1$, $V_2$, and $V_3$;
    \item $\Delta(A,B,C,D)=0$;
    \item $(A,B,C,D)$ is the image of some $(a_1,a_2,a_3)\in k^3$ under the map of Lemma \ref{ZariskiClose}, i.e.\ $(A,B,C,D)$ arises from the generalized cluster algebra equation (\ref{GyodaMatsushitaEq}).
\end{enumerate}
\end{proposition}
\begin{proof}
The equivalence of (1) and (2) was noted in the above discussion, and that of (2) and (4) is Proposition \ref{singletonOrbits} with $(x,y,z) = (-a_1,-a_2,-a_3)$, as discussed after Proposition \ref{RelationBetweenDegeneracyConditions}. For the equivalence of (1) and (3) we follow the proof of \cite[Theorem 4.1]{BenedettoGoldman}, whose algebraic input, \cite[Lemmas 4.2 and 4.3]{BenedettoGoldman}, is proved over an arbitrary field of characteristic different from $2$, respectively an algebraically closed such field. Regard the equation of $\mathcal{S}_{A,B,C,D}$ as a monic quadratic in $Z$. By \cite[Lemma 4.2]{BenedettoGoldman}, the surface is singular if and only if the plane curve $\{f(x,y)=0\}$ cut out by the discriminant of this quadratic is singular. The polynomial $f$ is a quadratic in $y$ whose leading coefficient is, up to a nonzero constant, $x^2-4$, which has no multiple roots as $\operatorname{char}k\neq2$; so by \cite[Lemma 4.3]{BenedettoGoldman} the curve is singular if and only if the discriminant $g(x)$ of $f$ with respect to $y$ has a multiple root. A direct computation (\cite[p.~93]{BenedettoGoldman}) shows that, in our normalization, $g(x)$ is a nonzero constant times our $f_1(x)$. Hence $\mathcal{S}_{A,B,C,D}$ is singular if and only if $\operatorname{disc}(f_1)=0$, which is $\Delta(A,B,C,D)=0$ by Lemma \ref{DeltaUniversality}. (Over $\mathbb{C}$ this is the statement recorded in \cite[\S3.5]{CantatLoray} that the singular locus is cut out by the discriminant of the quartic $P_z$.)
\end{proof}

\begin{remark}
The singular locus $\{\Delta=0\}$, expressed in the boundary traces $(a,b,c,d)$, is $\{(a^2-4)(b^2-4)(c^2-4)(d^2-4)\,\mathrm{Res}(\psi_{a,b},\psi_{c,d})^2=0\}$ with $\psi_{a,b}(x) = x^2-abx+a^2+b^2-4$: this is the union of the locus where some boundary trace is $\pm2$ and the locus where a reducible representation exists (\cite[Theorem 1.1 and \S4D]{BenedettoGoldman}, \cite[Theorem 3.7]{CantatLoray}; see also \cite[\S2.1]{CantatDupontMartinBaillon}). We have stated our results in terms of $(A,B,C,D)$ because, as noted in Section 1.4 of \cite{CantatLoray}, the map $(a,b,c,d)\mapsto(A,B,C,D)$ is a ramified cover of degree $24$, and the surface does not depend on the choice of preimage. Within this geometric language, degenerate parameters are those which arise when at least two of the boundary traces are $\pm 2$, so that our degenerate parameters all satisfy $\Delta = 0$; we now give a simple algebraic proof of this fact.
\end{remark}

\begin{proposition}\label{DegeneratesAndDelta}
    If the quadruple of parameters $(A, B, C, D)$ is degenerate, then $\Delta(A, B, C, D) = 0.$
\end{proposition}
\begin{proof}
    By noting that every term of $\Delta(A, B, C, D)$ has either even degree in each of $A, B,$ and $C$ or odd degree in each of $A, B,$ and $C$, and that $\Delta(A, B, C, D)$ is symmetric in $A, B,$ and $C$, we find that if $(A, B, C, D)$ and $(A', B', C', D')$ are equivalent, then $\Delta(A, B, C, D) = \Delta(A', B', C', D').$ We remark that symmetry in $A, B,$ and $C$ also follows from the alternative descriptions of $\Delta$ in Lemma \ref{DeltaUniversality}.
    
    Now, it is straightforward to check with Macaulay2 \cite{Macaulay2} that $(A^2 - 8C + 4D - 16)\mid\Delta(A, A, C, D),$ so that for degenerate parameters $(A, B, C, D)$ with $A = B$ and $A^2 + 4D = 8C + 16,$ $\Delta(A, B, C, D) =0.$ It follows then that the same is true of any equivalent parameters, proving the proposition. 
\end{proof}

By Proposition \ref{prop:SingularIffDelta} above, any parameters $A, B, C,$ and $D$ such that $$\Delta(A, B, C, D) = 0$$ do arise from the equation (\ref{GyodaMatsushitaEq}) via the change of variables (\ref{clusterAlgChangeVars}) if we work over an algebraically closed field of characteristic different from $2$; however, this certainly does not happen if we wish to restrict the parameters $a_1, a_2,$ and $a_3$ to elements of $\mathbb{F}_p$. Thus, despite Proposition \ref{DegeneratesAndDelta}, our degeneracy condition and corresponding obstruction genuinely extend those of de Courcy-Ireland, Litman, and Mizuno beyond just removing their condition $2a_{i-1} = a_ia_{i+1}$. We illustrate this with the following example:
\begin{example}
    We study the case $(A, B, C, D) = (4, 4, -2, -4)$. This is a degenerate quadruple with $A = B$ and $4D + A^2 = 8C + 16$, and hence $\Delta(4, 4,-2, -4) = 0.$ If we assume that it arises from a triple $(a_1, a_2, a_3)$ via (\ref{clusterAlgChangeVars}), we find as above that either $a_3 = 2$ or $a_1 = \pm 2$; this second case may quickly be ruled out, as it would imply $B = \pm C$.
We then have that
$$-2a_1 -2a_2 = 4,$$
so that 
$$a_1 + a_2 = -2;$$
and 
$$-4-a_1a_2 = -2,$$
so that
$$a_1a_2 = -2.$$
It follows that
$$a_1^2 + 2a_1 - 2 = 0,$$
so that $a_1 = -1\pm\sqrt{3}$ and $a_2 = -1\mp\sqrt{3}$. Over fields containing $\sqrt{3}$, this variety arises from (\ref{GyodaMatsushitaEq}) in two different ways, while over fields not containing $\sqrt{3}$ this variety does not arise from (\ref{GyodaMatsushitaEq}). However, the obstruction to transitivity we establish in Theorem \ref{QuadraticObstructions} coming from the degeneracy of these parameters occurs even over finite fields not containing $\sqrt{3}$.
\end{example}

\begin{remark}[Some less precise nondegeneracy conditions]\label{rmk:DifferentNondegConds}
When the author first began this project, he originally used the much stronger nondegeneracy condition $A \neq \pm B$, $A \neq \pm C$, and $B \neq \pm C$, which had been suggested by numerical experiments. However, there were many varieties $\mathcal{S}_{A, B, C, D}$ which seemed to obey strong approximation but which failed this condition, including the important special case (\ref{MarkoffPlusk}). He later found the polynomial $\Delta$ after reading \cite{pDivisibilityClusterAlg}. Demanding merely $\Delta \neq 0$ includes all of the equations of the form (\ref{MarkoffPlusk}) except for the cases $k = 0$ and $k = 4$, because $\Delta(0, 0, 0, k) = 64k(k-4)^4$, but excludes all of the varieties (\ref{GyodaMatsushitaEq}) connected to generalized cluster algebras. Allowing either $A \neq \pm B$, $A \neq \pm C$, and $B \neq \pm C$ or $\Delta \neq 0$ includes most of the varieties (\ref{GyodaMatsushitaEq}) but excludes those with $a_i = \pm a_j$ for $i \neq j$ in addition to those with $a_i = \pm 2$, and continues to exclude the classical Markoff equation in addition to the Cayley cubic from the family (\ref{MarkoffPlusk}). (In fact, these two varieties are the only two in the family (\ref{MarkoffPlusk}) that have the structure of a Generalized Cluster Algebra: the Markoff equation of course corresponds to the parameters $a_1 = a_2 = a_3 = 0,$ while the Cayley cubic corresponds to the parameters $a_1 = a_2 = a_3 = -2.$)
 
 The refined condition that $(A, B, C, D)$ must be equivalent to a quadruple where $4D + A^2 = 8C + 16$ forces $k = 4$ in the family (\ref{MarkoffPlusk}), so that the only degenerate surface in this family is the Cayley cubic, and forces $a_i = \pm 2$ for some $i$ in the family (\ref{GyodaMatsushitaEq}), as expected from \cite{pDivisibilityClusterAlg}.
\end{remark}

\section{Analysis of the Conic Sections}\label{ConicsSection}
We now move into the proof of our main theorem. As in \cite{BGS}, we begin by restricting the dynamics to the curves (which are conic sections that may be degenerate) formed by intersecting the surface
$$X^2 + Y^2 + Z^2 = XYZ + AX + BY + CZ + D$$
with planes of the form $X = x_0$, $Y = y_0$, or $Z = z_0$ and studying the resulting orbits; over the course of Sections \ref{EndgameSection}, \ref{middlegameSection}, and \ref{openingSection} we shall study how orbits within these curves intersect within the surface to achieve transitivity, a strategy described as ``fiber jumping'' in \cite{FuchsLitmanSilvermanTran}.

We will need to study these dynamics both combinatorially and algebraically; in this section we will start with the combinatorial analysis, proceed to the algebraic, and then indicate how the two support one another. As in \cite{BGS} and \cite{LT}, it will be convenient to introduce a graph $G = G_{A, B, C, D}$ whose vertices are the $\mathbb{F}_p$-points of $\mathcal{S}_{A, B, C, D}^*$, with an edge connecting the two vertices $(x_1, y_1, z_1)$ and $(x_2, y_2, z_2)$ if one solution is transformed into the other by a Vieta involution; as in \cite{LT} it will further be useful to give $G$ the natural 3-edge-coloring given by coloring the edge from $(x_1, y_1, z_1)$ to $(x_2, y_2, z_2)$ with color one, two, or three depending on whether $(x_1, y_1, z_1)$ is taken to $(x_2, y_2, z_2)$ by $V_1$, $V_2$, or $V_3$; this is in fact a proper coloring as each $V_i$ is an involution.\footnote{\label{NoBigons}We remark that for $i\neq j$, if both $V_i$ and $V_j$ take $(x_1, y_1, z_1)$ to $(x_2, y_2, z_2)$, then in fact $(x_1, y_1, z_1) = (x_2, y_2, z_2)$ must be a fixed point of both; in this case, we will attach to the corresponding vertex a self-loop of color $i$ and a self-loop of color $j$. This has also been observed elsewhere, for instance as Proposition 1.3 in \cite{pDivisibilityClusterAlg}, where it is observed that this fact means these graphs have no bigons.} We note that the connected components of $G$ are exactly the $\Gamma$-orbits of $\mathcal{S}_{A, B, C, D}^*(\mathbb{F}_p)$.

As in \cite{BGS} and \cite{LT}, we begin with a study of the orbit of a point $(x, y, z)$ of $\mathcal{S}_{A, B, C, D}^*$ under the group generated by two of the Vieta involutions, say $V_i$ and $V_j$. We start with the following basic observation (cf. the discussion in \cite{LT} following lemma 13).

\begin{lemma}\label{twoColoredOrbits}
    Let $G_{ij}$ be the subgraph of $G$ whose vertices are those in the orbit of a single point $(x, y, z)$ under the group generated by $V_i$ and $V_j$, and whose edges are those between the given vertices and colored with color $i$ or color $j$. Then $G_{ij}$ is either a cycle of even order or a line segment with a self-loop on each end.
\end{lemma}
\begin{proof}
    As $G_{ij}$ is of finite order, connected, and the degree of any vertex without a loop attached is two, $G_{ij}$ can only be a cycle or a line segment with a self-loop on each end. If $G_{ij}$ is a cycle, then it has an equal number of vertices and edges; as it admits a two-edge-coloring (by the colors $i$ and $j$) it must have an even number of edges, and thus an even number of vertices. 
\end{proof}

Our argument will proceed by studying the action of the iterates of maps of the form $V_i\circ V_j$ for $i\neq j$, which may be interpreted as the maps on the surface $\mathcal{S}_{A, B, C, D}$ induced by \textit{Dehn twists} of the four-times punctured sphere. These maps play a similar role in our analysis to the \textit{rotations} in \cite{BGS}, which also arise from Dehn twists, and which are, in the symmetric Markoff setting, square roots of the maps considered in this paper. In our setting, as $\mathcal{S}_{A, B, C, D}$ is typically not symmetric under the interchange of the variables, we generally do not have access to such a square root, which will be a source of some complications in what follows. The following lemma relates the orbit of $(x, y, z)$ under the group generated by $V_i$ and $V_j$ to that under the iterates of the map $V_i\circ V_j$:

\begin{lemma}\label{comparisonOfTwoColoredOrbits}
    Let $G_{ij}$ be the subgraph of $G$ whose vertices are those in the orbit of a single point $(x, y, z)$ under the group generated by $V_i$ and $V_j$. If $G_{ij}$ is a line with a self-loop on each end, then $V_i \circ V_j$ acts transitively on the vertices of $G_{ij}$. If $G_{ij}$ is a cycle, then the vertices of $G_{ij}$ form two orbits under $V_i\circ V_j$. 
\end{lemma}
\begin{proof}
    Suppose first that $G_{ij}$ is a cycle. By Lemma \ref{twoColoredOrbits}, $G_{ij}$ must have even order, and the result follows immediately as $V_i\circ V_j$ rotates the cycle by two steps.
    
    Suppose next that $G_{ij}$ is a line with a self-loop on each end. Let $v_1$ be a vertex of $G_{ij}$ with a self-loop, and number the remaining vertices of $G_{ij}$ consecutively, so that $v_2$ is the unique vertex adjacent to $v_1$, $v_3$ is the vertex adjacent to $v_2$ other than $v_1,$ etc., until reaching the other vertex with a self-loop, which we label $v_n$ ($n = |G_{ij}|$). Assume without loss of generality that the self-loop of $v_1$ is of color $i$ (else, study the backwards iterates of $V_i\circ V_j$ instead of the forwards iterates, which interchanges the roles of $i$ and $j$). Then $V_i\circ V_j$ maps $v_k$ to $v_{k+2}$ for $k < n-1,$ and thus iterating $V_i\circ V_j$ allows us to map $v_1$ to any vertex of odd index. If $n$ is odd, then $v_1$ is mapped to $v_n$, which will be fixed by $V_j$, and thus mapped under $V_i\circ V_j$ to $v_{n-1}$; if $n$ is even, then $v_1$ will be mapped to $v_{n-1}$, which will then be mapped to $v_n$ by $V_j$, and $v_n$ will be fixed by $V_i$, so that $V_i \circ V_j$ will take $v_{n-1}$ to $v_n$. Either way, this takes $v_1$ to a vertex of even index, and iterating further will take $v_1$ to all vertices of even index.
\end{proof}

We now move from the combinatorial study of the two-colored orbit of a point $(x, y, z)$ to the algebraic study of this orbit. For definiteness, we will focus on the iteration of the map $D_1 \coloneq V_3\circ V_2;$ analogous results about the iteration of $D_2\coloneq V_3\circ V_1$ and $D_3\coloneq V_2\circ V_1$ follow after permuting $A, B,$ and $C$ by symmetry. For a fixed value of $x$, this map is the affine map

\[
\begin{pmatrix}
    y\\ z
\end{pmatrix}\mapsto
\begin{pmatrix}
    -1 & x\\
    -x & x^2-1
\end{pmatrix}\begin{pmatrix}
    y\\ z
\end{pmatrix} + \begin{pmatrix}
    B\\ Bx + C
\end{pmatrix}.
\]
For the purposes of studying the iterates of this map, it will be convenient to embed solutions in three dimensional space to linearize the map, obtaining:
\[
\begin{pmatrix}
    y\\ z\\ 1
\end{pmatrix}\mapsto
\begin{pmatrix}
    -1 & x & B\\
    -x & x^2-1 & Bx+C\\
    0 & 0 & 1
\end{pmatrix}\begin{pmatrix}
    y\\ z\\ 1
\end{pmatrix}.
\]

For $x \neq \pm 2$, this matrix diagonalizes as

\[\begin{gathered}
\begin{pmatrix}
    -1 & x & B\\
    -x & x^2-1 & Bx+C\\
    0 & 0 & 1
\end{pmatrix}\\[0.4em]
= 
\varphi^{-1}
\begin{pmatrix}
    1 & 1 & \varphi^{-2}(-2B - Cx)\\
    \chi & \chi^{-1} & \varphi^{-2}(-Bx - 2C)\\
    0 & 0 & 1
\end{pmatrix}\begin{pmatrix}
    \chi^2 & 0 & 0\\
    0 & \chi^{-2} & 0\\
    0 & 0 & 1
\end{pmatrix}
\begin{pmatrix}
    \chi^{-1} & -1 & \varphi^{-1}(B + C\chi^{-1})\\
    -\chi & 1 & \varphi^{-1}(B + C\chi)\\
    0 & 0 & \varphi
\end{pmatrix},
\end{gathered}\]
where $\chi$ is a root of $\lambda^2 - x\lambda + 1$ and $\varphi = \chi^{-1}-\chi$. If $B = C$ we can instead use the map $\tau_{yz}\circ V_2,$ which may be represented by the matrix

\[
\begin{pmatrix}
    y\\ z\\ 1
\end{pmatrix}\mapsto
\begin{pmatrix}
    0 & 1  & 0\\
    -1 & x & B\\
    0 & 0 & 1
\end{pmatrix}\begin{pmatrix}
    y\\ z\\ 1
\end{pmatrix},
\]

and if $B = -C$ we can instead use the map $\operatorname{neg}_{yz}\circ\tau_{yz}\circ V_2,$ which may be represented by the matrix

\[
\begin{pmatrix}
    y\\ z\\ 1
\end{pmatrix}\mapsto
\begin{pmatrix}
    0 & -1  & 0\\
    1 & -x & -B\\
    0 & 0 & 1
\end{pmatrix}\begin{pmatrix}
    y\\ z\\ 1
\end{pmatrix}.
\]

In each case, letting $\pm$ denote consistently the sign in $B = \pm C$, the corresponding matrix may be diagonalized as

\[\varphi^{-1}\begin{pmatrix}
    1 & 1 & \varphi^{-2}(-2B - Cx)\\
    \chi & \chi^{-1} & \varphi^{-2}(-Bx - 2C)\\
    0 & 0 & 1
\end{pmatrix}\begin{pmatrix}
    \pm\chi & 0 & 0\\
    0 & \pm\chi^{-1} & 0\\
    0 & 0 & 1
\end{pmatrix}\begin{pmatrix}
    \chi^{-1} & -1 & \varphi^{-1}(B + C\chi^{-1})\\
    -\chi & 1 & \varphi^{-1}(B + C\chi)\\
    0 & 0 & \varphi
\end{pmatrix},
\]

and gives a square root of the map $V_3\circ V_2$. We remark that it follows that $x = \chi + \chi^{-1}$, and that if $\chi\in\mathbb{F}_{p^2}\setminus\mathbb{F}_p,$ $\chi^{p+1} = N(\chi) = 1$. In all of these cases, $x$ will be said to be \textit{hyperbolic} if $\chi\in\mathbb{F}_p$ and to be \textit{elliptic} if $\chi\in\mathbb{F}_{p^2}\setminus \mathbb{F}_p.$ The coordinate change diagonalizing the Dehn twist $D_1$ also gives a particularly nice form for the conic $C_1(x)$ given by intersecting $\mathcal{S}_{A, B, C, D}$ with the plane $X = x$: indeed, writing

\begin{equation}\label{DiagonalizedComponents}
    \begin{pmatrix}
    \xi\\ \eta\\ 1
\end{pmatrix} = (\chi^{-1} - \chi)^{-1}\begin{pmatrix}
    \chi^{-1} & -1 & (\chi^{-1} - \chi)^{-1}(B + C\chi^{-1})\\
    -\chi & 1 & (\chi^{-1} - \chi)^{-1}(B + C\chi)\\
    0 & 0 & (\chi^{-1} - \chi)
\end{pmatrix}\begin{pmatrix}
    y\\ z\\ 1
\end{pmatrix},
\end{equation}
a routine computation shows that:
\[\begin{split}
    \xi\eta &= (\chi^{-1} - \chi)^{-2}\left(\chi^{-1}y - z + (\chi^{-1} - \chi)^{-1}(B + C\chi^{-1})\right)\left(-\chi y + z + (\chi^{-1} - \chi)^{-1}(B + C\chi)\right)\\
    &= (x^2 - 4)^{-1}\left(-y^2 - z^2 + xyz + By + Cz + (x^2 -4)^{-1}(B^2 + C^2 + BCx)\right)\\
    &= (x^2 - 4)^{-1}(x^2 - Ax - D + (x^2 - 4)^{-1}(B^2 + C^2 + BCx)).
\end{split}
\]
We shall denote by $\kappa_1(x)$ the all-important quantity
$$\kappa_1(x) = (x^2 - 4)^{-1}(x^2 - Ax - D + (x^2 - 4)^{-1}(B^2 + C^2 + BCx)).$$

If $\xi$ and $\eta$ in (\ref{DiagonalizedComponents}) both equal 0, then the orbit of $(x, y, z)$ under $V_3\circ V_2$ will be trivial. Fortunately, these are easily controlled:

\begin{proposition}[Double Fixed Points]\label{DoubleFixedPoints}
    Let $(x, y, z)$ be a point of $\mathcal{S}_{A, B, C, D},$ with $x\neq \pm2,$ and let $\xi$ and $\eta$ be defined as in (\ref{DiagonalizedComponents}). Then the following are equivalent:
    \begin{enumerate}
        \item $\xi = \eta = 0;$
        \item $y = \frac{-2B - Cx}{x^2-4}$ and $z = \frac{-2C - Bx}{x^2 -4}.$
        \item $y = \frac{B + xz}{2}$ and $z = \frac{C + xy}{2};$ and
        \item $(x, y, z)$ is fixed by both $V_2$ and $V_3$.
    \end{enumerate}
    There are at most four such points.
\end{proposition}
\begin{remark}
    In the setting of (\ref{GyodaMatsushitaEq}), de Courcy-Ireland, Litman, and Mizuno prove a similar result in their Lemma 2.1 (\cite{pDivisibilityClusterAlg}). These double fixed points also arise \cite{LT}, where they simply called ``bad points.''
\end{remark}
\begin{proof}
    The equivalence of (3) and (4) is clear from the formulas defining $V_2$ and $V_3$.
    Assuming (3), we find that
    $$y = \frac{B + xz}{2} = \frac{x^2y + Cx + 2B}{4},$$
    so that
    $$y = \frac{-2B-Cx}{x^2-4};$$
    in the same way (3) also implies that
    $$z = \frac{-2C - Bx}{x^2 -4},$$
    so that (3) implies (2). We also have that (2) implies (3), for
    $$B + x\frac{-2C - Bx}{x^2-4} = \frac{Bx^2 - 4B - 2Cx - Bx^2}{x^2-4} = 2\frac{-2B - Cx}{x^2-4},$$
    so that (2) implies
    $$y = \frac{B + xz}{2};$$
    in the same way (2) also implies that
    $$z = \frac{C + xy}{2},$$
    so that (2) implies (3).
    Now, if we assume $\xi = \eta = 0$, we have that
    $$0 = (\chi^{-1} - \chi)^2\xi = \chi^{-2}y - y - \chi^{-1}z + \chi z + B + C\chi^{-1},\text{ and}$$
    $$0 = (\chi^{-1}-\chi)^2\eta = -y + \chi^2y + \chi^{-1}z - \chi z + B + C\chi.$$
    Adding these equations we get
    $$(\chi^{-2} - 2 + \chi^{2})y + 2B + C(\chi + \chi^{-1}) = 0,$$
    so that $$y = \frac{-2B-Cx}{x^2-4}.$$
    Subtracting these equations gives
    $$(\chi^{-2}-\chi^2)y - 2\chi^{-1}z + 2\chi z + C(\chi^{-1}-\chi) = 0,$$
    so that
    $$(\chi + \chi^{-1})y - 2z + C = 0$$
    and
    $$z = \frac{C + xy}{2} = \frac{-2C-Bx}{x^2-4}.$$
    Thus, (1) implies (2).
    Finally, plugging the formulas for $y$ and $z$ in (2) into the formulas for $\xi$ and $\eta$ from (\ref{DiagonalizedComponents}) shows that (2) implies (1). This proves the equivalence. For the bound, for any double fixed point $(x, y, z)$, we have $\xi\eta = \kappa_1(x) = 0$, whence $f_1(x) = 0$ as well, so there are at most 4 possible values of $x$. As the value of $x$ uniquely determines the values of $y$ and $z$ via (2), this gives at most 4 such points.
\end{proof}
We also note that if $\chi\in\mathbb{F}_{p^2}-\mathbb{F}_p,$ $\chi^{p+1} = 1$ implies that $\xi^p = \eta$. To simplify our treatment of the elliptic case, we introduce the following notation:

\begin{notation}
    Let $\mu_{p+1}$ denote the \textit{norm-one subgroup of} $\mathbb{F}_{p^2},$ i.e., \[
\mu_{p+1}=\{u\in\mathbb F_{p^2}^*:u^{p+1}=1\}.
\]
\end{notation}

All in all, the above discussion proves the following pair of lemmas:

\begin{lemma}\label{hyperbolics}
    Let $x$ be hyperbolic, with $\kappa_1(x)\not\equiv 0\pmod{p}$, and write 
    $x = \chi + \chi^{-1}$
    for $\chi\in \mathbb{F}_p.$ Then $C_1(x)$ is a hyperbola with $p-1$ points. Let
    $$\widetilde{C}_1(x) = \left\{\left(t, \frac{\kappa_1(x)}{t}\right)\mid t\in\mathbb{F}_p^*\right\}.$$
    Then the map
    $$\left(t, \frac{\kappa_1(x)}{t}\right)\mapsto \left(t + \frac{\kappa_1(x)}{t} + (x^2 - 4)^{-1}(-2B - Cx), \chi t + \frac{\kappa_1(x)}{\chi t} + (x^2 - 4)^{-1}(-Bx-2C)\right)$$
    is a bijection from $\widetilde{C}_1(x)$ to $C_1(x)$. In these coordinates, the map $D_1 = V_3\circ V_2$ acts as:
    $$\left(t, \frac{\kappa_1(x)}{t}\right)\mapsto\left(\chi^2 t, \frac{\kappa_1(x)}{\chi^2 t}\right).$$
\end{lemma}
\begin{lemma}\label{elliptics}
    Let $x$ be elliptic, with $\kappa_1(x)\not\equiv0\pmod{p}$, and write 
    $x = \chi + \chi^{-1}$
    for $\chi\in \mu_{p+1} - \mathbb{F}_p$. Then $C_1(x)$ is an ellipse with $p+1$ points. Let
    $$\widetilde{C}_1(x) = \left\{\left(t, \frac{\kappa_1(x)}{t}\right)\mid t\in\mathbb{F}_{p^2}^*, t^{p+1} = \kappa_1(x)\right\}.$$
    Then the map
    $$\left(t, \frac{\kappa_1(x)}{t}\right)\mapsto \left(t + \frac{\kappa_1(x)}{t} + (x^2 - 4)^{-1}(-2B - Cx), \chi t + \frac{\kappa_1(x)}{\chi t} + (x^2 - 4)^{-1}(-Bx-2C)\right)$$
    is a bijection from $\widetilde{C}_1(x)$ to $C_1(x)$. In these coordinates, the map $D_1 = V_3\circ V_2$ acts as:
    $$\left(t, \frac{\kappa_1(x)}{t}\right)\mapsto\left(\chi^2 t, \frac{\kappa_1(x)}{\chi^2 t}\right).$$
\end{lemma}
In an important difference from the setting of \cite{BGS}, even when $\chi$ has maximal order, the Dehn twist does not act transitively on $C_1(x)$. Fortunately, our above discussion of the combinatorial structure of orbits provides a partial remedy.
\begin{proposition}\label{prop:SingleOrbitAndKappa}
    Suppose that $x\neq \pm 2,$ $\kappa_1(x)\neq 0$, and $x = \chi + \chi^{-1},$ with $\chi$ of maximal order, i.e. $p-1$ (if hyperbolic) or $p+1$ (if elliptic). Then $C_1(x)$ forms a single orbit under the action of the group generated by $V_2$ and $V_3$ if and only if
    $$\left(\frac{\kappa_1(x)}{p}\right) = -1.$$
\end{proposition}
\begin{proof}
    As $x$ is of maximal order with $\kappa_1(x)\neq 0$, $C_1(x)$ splits into two orbits of size $\frac{p\pm 1}{2}$ under the action of $D_1 = V_3\circ V_2$. Fix $(x, y, z)\in C_1(x)$. By Lemma \ref{comparisonOfTwoColoredOrbits}, these orbits either jointly form a single two-colored cycle in $G$, or remain separate as two lines each with a self-loop on each end. In the latter case, $C_1(x)$ must contain four points each of which is fixed by one of $V_2$ or $V_3.$ Now, for a fixed point of $V_2$, we have that:
    \begin{equation}\label{fixedPointV2}y = \frac{B}{2} + \frac{xz}{2}.\end{equation}
    Substituting this into the Markoff equation, one finds that
    \begin{equation}\label{conditionOnX}
        \left(1-\frac{x^2}{4}\right)z^2 + \left(-C - \frac{Bx}{2}\right)z + \left(x^2 - Ax - D - \frac{B^2}{4}\right) = 0.
    \end{equation}
    For a fixed value of $x\neq \pm2$, this equation is quadratic in $z$, so that there are at most two solutions for $z$; also, as $y$ and $z$ are linearly related by (\ref{fixedPointV2}), each value of $z$ solving the above gives rise to only one fixed point in $C_1(x)$, so that $V_2$ has at most two fixed points. Similarly, $V_3$ can have at most two fixed points. It follows that under the action of the group generated by $V_2$ and $V_3$, $C_1(x)$ decomposes into two orbits if and only if it contains four points each of which is fixed by one of $V_2$ or $V_3$, if and only if it contains two fixed points under $V_2$ and two fixed points under $V_3$, and that the group generated by $V_2$ and $V_3$ acts transitively on $C_1(x)$ if and only if neither $V_2$ nor $V_3$ has a fixed point in $C_1(x)$; in particular, $C_1(x)$ decomposes into two orbits under the action of the group generated by $V_2$ and $V_3$ if and only if (\ref{conditionOnX}) has a solution. This will occur if and only if
    \[\begin{aligned}&\left(-C-\frac{Bx}{2}\right)^2 - 4\left(1-\frac{x^2}{4}\right)\left(x^2 - Ax - D - \frac{B^2}{4}\right) \\
    &=x^4 - Ax^3 - (D + 4)x^2 + (4A + BC)x + (4D + C^2 + B^2)\end{aligned}\]
    is a square modulo $p$. Finally, we note that
    \[\begin{split}(x^2-4)^2\kappa_1(x) &= (x^2 - 4)(x^2 - Ax - D) + (B^2 + C^2 + BCx)\\
    &= x^4 - Ax^3 - (D + 4)x^2 + (4A + BC)x + (4D + B^2 + C^2);\end{split}\]
    the result follows.
    
\end{proof}
\begin{remark}[Another path to this result]
    The key input to the above was the relationship between fixed points and the number of orbits under the action of the group generated by $V_2$ and $V_3$ arising from the combinatorics of Lemma \ref{comparisonOfTwoColoredOrbits}. This could also be proven via Burnside's lemma, using the fact that the group generated by $V_2$ and $V_3$ acts on the conic section as a dihedral group of order $2N$, where $N$ is the order of $D_1,$ with $V_3$ (say) chosen as a reflection, using that all group elements of the form $V_3\circ D_1^k$ will have the same number of fixed points. This approach was inspired by some of the analysis in \cite{pDivisibilityClusterAlg}.
\end{remark}
We also must describe the action of $D_1 = V_3\circ V_2$ on $C_1(\pm 2)$. For these values of $x$, the equation cutting out $\mathcal{S}_{A, B, C, D}$ becomes
$$y^2 + z^2 + 4 = \pm2yz \pm2A + By + Cz + D;$$
for a fixed value of $y$, this has two solutions if and only if
\begin{equation}\label{parabolicCondition}
    C^2 \pm4Cy + 4y^2 - 4(y^2 \mp2A - By - D + 4) = (4B\pm 4C)y \pm 8A + C^2 + 4D - 16
\end{equation}
is a quadratic residue and one if and only if (\ref{parabolicCondition}) is 0. Our analysis splits into several cases. Throughout the case breakdown, equalities and inequalities are all understood in $\mathbb{F}_p$:

First, assume that $B \neq \pm C$. In this case, $|C_1(\pm 2)| = p$. Furthermore, $D_1$ acts transitively on $C_1(\pm 2),$ and may be parameterized by putting the matrix of $D_1$ into Jordan canonical form, as in the following, where $\pm$ is the same sign as $x = \pm 2$ and $\mp$ is the opposite sign:

$$\begin{pmatrix}
    -1 & \pm2 & B\\
    \mp2 & 3 & \pm2B+C\\
    0 & 0 & 1
\end{pmatrix} =\begin{pmatrix}
    1 & 0 & B(4B\pm 4C)^{-1}\\
    \pm1 & \pm\frac{1}{2} & 0\\
    0 & 0 & (2B \pm 2C)^{-1}
\end{pmatrix}\begin{pmatrix}
    1 & 1 & 0\\
    0 & 1 & 1\\
    0 & 0 & 1
\end{pmatrix}\begin{pmatrix}
    1 & 0 & -\frac{B}{2}\\
    -2 & \pm2 & B\\
    0 & 0 & 2B\pm 2C
\end{pmatrix},$$
which certainly has order $p$. It further follows that the action of $D_1^t$ on $C_1(\pm2)$ may be written as
$$\begin{pmatrix}
    1 & 0 & B(4B\pm 4C)^{-1}\\
    \pm1 & \pm\frac{1}{2} & 0\\
    0 & 0 & (2B \pm 2C)^{-1}
\end{pmatrix}\begin{pmatrix}
    1 & t & \binom{t}{2}\\
    0 & 1 & t\\
    0 & 0 & 1
\end{pmatrix}\begin{pmatrix}
    1 & 0 & -\frac{B}{2}\\
    -2 & \pm2 & B\\
    0 & 0 & 2B\pm 2C
\end{pmatrix}.$$

If on the other hand we assume $B = \pm C,$ then the situation becomes more complicated, and depends on the relative signs of $B/C$ and $x/2.$ We break into cases depending on this sign.

If $B = C\neq 0,$ then for $x = 2$ the situation is as before: the number of $z$ such that $(2, y, z)$ is a solution is
$$1 + \left(\frac{(4B + 4C)y + 8A + C^2 + 4D - 16}{p}\right),$$
so that $|C_1(2)| = p,$ and $D_1$ may be put into Jordan canonical form as
$$\begin{pmatrix}
    -1 & 2 & B\\
    -2 & 3 & 2B+C\\
    0 & 0 & 1
\end{pmatrix} =\begin{pmatrix}
    1 & 0 & B(4B+ 4C)^{-1}\\
    1 & \frac{1}{2} & 0\\
    0 & 0 & (2B + 2C)^{-1}
\end{pmatrix}\begin{pmatrix}
    1 & 1 & 0\\
    0 & 1 & 1\\
    0 & 0 & 1
\end{pmatrix}\begin{pmatrix}
    1 & 0 & -\frac{B}{2}\\
    -2 & 2 & B\\
    0 & 0 & 2B+ 2C
\end{pmatrix},$$
and the action of $D_1^t$ on $C_1(2)$ is as before. However, for $x = -2$, the situation is different, because $B - C = 0.$ In this case, as well as when $B = C = 0$, setting $x = -2$ in the equation for $\mathcal{S}_{A, B, C, D}$ gives
$$y^2 + 2yz + z^2 - By - Bz + \frac{B^2}{4} = \left(y+z-\frac{B}{2}\right)^2= D + \frac{B^2}{4} - 2A - 4.$$
It follows that the number of pairs $(y, z)$ is 
$$\left(1 + \left(\frac{- 8A + C^2 + 4D - 16}{p}\right)\right)p,$$
so that there are either $0$ or $2p$ solutions if $(A, B, C, D)$ is nondegenerate modulo $p$, and exactly $p$ solutions if $(A, B, C, D)$ is degenerate modulo $p$. The Jordan canonical form of $D_1$ is
$$\begin{pmatrix}
    -1 & -2 & B\\
    2 & 3 & -B\\
    0 & 0 & 1
\end{pmatrix} =\begin{pmatrix}
    1 & 0 & \frac{B}{2}\\
    -1 & -\frac{1}{2} & 0\\
    0 & 0 & 1
\end{pmatrix}\begin{pmatrix}
    1 & 1 & 0\\
    0 & 1 & 0\\
    0 & 0 & 1
\end{pmatrix}\begin{pmatrix}
    1 & 0 & -\frac{B}{2}\\
    -2 & -2 & B\\
    0 & 0 & 1
\end{pmatrix}.$$
Although the matrix has order $p$ as a matrix, its action on $C_1(-2)$ has smaller order precisely when
$$(-2, -2, B)\cdot (y, z, 1) = -2y - 2z + B = 0,$$
that is, exactly when $4D + B^2 = 8A + 16$, in which case $(A, B, C, D)$ is degenerate modulo $p$. In this case, each point $P$ of $C_1(-2)$ is fixed by $V_3\circ V_2,$ so that $P$ and $V_2P$ form a bigon and thus $P$ is fixed by each of $V_2$ and $V_3$ (cf. footnote \ref{NoBigons}). The action of $D_1^t$ on $C_1(-2)$ may be written as
$$\begin{pmatrix}
    1 & 0 & \frac{B}{2}\\
    -1 & -\frac{1}{2} & 0\\
    0 & 0 & 1
\end{pmatrix}\begin{pmatrix}
    1 & t & 0\\
    0 & 1 & 0\\
    0 & 0 & 1
\end{pmatrix}\begin{pmatrix}
    1 & 0 & -\frac{B}{2}\\
    -2 & -2 & B\\
    0 & 0 & 1
\end{pmatrix}.$$

Similarly, if $B = -C \neq 0,$ then for $x = -2,$ the situation is as in the $B \neq \pm C$ case, while if $x = 2$, including also the case $B = C = 0$, $C_1(2)$ may consist of $0$ or $2p$ points if $(A, B, C, D)$ is nondegenerate modulo $p$, and exactly $p$ points if $(A, B, C, D)$ is degenerate modulo $p$, depending on the value of $p$:
$$|C_1(2)| = \left(1  + \left(\frac{8A + C^2 + 4D -16}{p}\right)\right)p.$$
The Jordan canonical form of $D_1$ may be written as
$$\begin{pmatrix}
    -1 & 2 & B\\
    -2 & 3 & B\\
    0 & 0 & 1
\end{pmatrix} =\begin{pmatrix}
    1 & 0 & \frac{B}{2}\\
    1 & \frac{1}{2} & 0\\
    0 & 0 & 1
\end{pmatrix}\begin{pmatrix}
    1 & 1 & 0\\
    0 & 1 & 0\\
    0 & 0 & 1
\end{pmatrix}\begin{pmatrix}
    1 & 0 & -\frac{B}{2}\\
    -2 & 2 & B\\
    0 & 0 & 1
\end{pmatrix},$$
which again has order $p$ as a matrix; as before the action on $C_1(2)$ has order $p$ as well in the nondegenerate modulo $p$ case, while each point of $C_1(2)$ is a double fixed point in the degenerate modulo $p$ case. The Jordan canonical form of $D_1^t$ is given by
$$\begin{pmatrix}
    1 & 0 & \frac{B}{2}\\
    1 & \frac{1}{2} & 0\\
    0 & 0 & 1
\end{pmatrix}\begin{pmatrix}
    1 & t & 0\\
    0 & 1 & 0\\
    0 & 0 & 1
\end{pmatrix}\begin{pmatrix}
    1 & 0 & -\frac{B}{2}\\
    -2 & 2 & B\\
    0 & 0 & 1
\end{pmatrix}.$$
Ultimately, the above discussion proves the following lemma:
\begin{lemma}\label{Parabolics}
    Suppose first that $B\neq C$ for $x=-2$, or that $B\neq -C$ for $x=2$. In each of these cases, $C_1(x)$ is a parabola containing $p$ points. The group generated by $V_2$ and $V_3$ acts transitively on $C_1(x)$, and the associated subgraphs of $G$ are each a line with a self-loop on each end. Algebraically, writing
    $x=\mp2$ and taking the corresponding condition $B\neq\pm C$, if $(\mp2,y,z)\in C_1(\mp2)$ then $C_1(\mp2)$ may be parameterized as
    $$C_1(\mp2) = \{(\mp2, y \mp 2t(z\pm y) + Bt^2 \mp C(t^2-t), z + 2t(z\pm y) \mp B(t^2 + t) + Ct^2)| t\in\mathbb{F}_p\}.$$
    
    If $B = C,$ then
    $$|C_1(-2)| = \left(1  + \left(\frac{-8A + C^2 + 4D -16}{p}\right)\right)p,$$
    so that $C_1(-2)$ consists of $0$ or $2p$ points if $(A, B, C, D)$ is nondegenerate modulo $p$, and exactly $p$ points if $(A, B, C, D)$ is degenerate modulo $p$ with $4D + C^2 = 8A + 16.$ In the latter case, each point in $C_1(-2)$ is fixed by each of $V_2$ and $V_3$, whereas in the former case each orbit in $C_1(-2)$ under $D_1$ has size $p,$ and may be parameterized from an arbitrary point $(-2, y, z)$ on $C_1(-2)$ as
    $$\mathcal{O}_x(-2, y, z) = \{(-2, y + t(-2y-2z + B), z - t(-2y-2z+B))\mid t\in\mathbb{F}_p\}.$$
    In either case, $C_1(-2)$ consists of the union of $0, 1,$ or $2$ lines; in the case of $0$ or $2$ lines each orbit under $D_1$ is a line in $\mathbb{F}_p^3$ with $p$ points.
    Similarly, if $B = -C,$ then
    $$|C_1(2)| = \left(1  + \left(\frac{8A + C^2 + 4D -16}{p}\right)\right)p,$$
    so that $C_1(2)$ consists of $0$ or $2p$ points if $(A, B, C, D)$ is nondegenerate modulo $p$, and exactly $p$ points if $(A, B, C, D)$ is degenerate modulo $p$ with $4D + C^2 = -8A + 16.$ In the latter case, each point in $C_1(2)$ is fixed by each of $V_2$ and $V_3$, whereas in the former case each orbit in $C_1(2)$ under $D_1$ has size $p,$ and may be parameterized from an arbitrary point $(2, y, z)$ on $C_1(2)$ as
        $$\mathcal{O}_x(2, y, z) = \{(2, y + t(-2y+2z + B), z + t(-2y+2z+B))\mid t\in\mathbb{F}_p\}.$$
    In either case, $C_1(2)$ consists of the union of $0, 1,$ or $2$ lines; in the case of $0$ or $2$ lines each orbit under $D_1$ is a line in $\mathbb{F}_p^3$ with $p$ points.
    
\end{lemma}
\begin{proof}
    In all cases, the count of points in $C_1(\pm2)$ has been established above, which proves transitivity under $V_2$ and $V_3$ in the cases where $D_1$ acts transitively (in the other cases, no assertion is made for the action of $V_2$ and $V_3$). The parameterizations follow from acting on $(\pm2, y, z)$ by $D_1^t,$ whose matrix has been computed. All that remains to be established is the characterization of the subgraphs of $G$ in the cases where the conic has exactly $p$ points; this follows from the fact that $C_1(\pm2)$ has odd order, and thus cannot form a cycle, which must have even order by Lemma \ref{twoColoredOrbits}.
\end{proof}

\begin{remark}
    In the case where $B = \pm C,$ a careful choice of base points $(2, y, z)$ for the orbits shows that when $C_1(\mp 2)$ consists of two orbits, they are not interchanged by the action of $\tau_{yz}$ (if $B = C$) or $\operatorname{neg}_{yz}\circ\tau_{yz}$ (if $B = -C$). If $B = C = 0$ then $\tau_{yz}$ interchanges the orbits within $C_1(2)$ while $\operatorname{neg}_{yz}\circ\tau_{yz}$ interchanges the orbits within $C_1(-2)$.  
\end{remark}
Finally, we remark that the same results hold for the actions of the maps $V_3\circ V_1$ and $V_2\circ V_1$ on the conics $C_2(y)$ and $C_3(z)$ that are given by fixing $Y = y$ and $Z = z$, respectively, modulo the appropriate interchange of the constants $A, B,$ and $C$. In particular, one must replace the function $\kappa_1(x)$ with either
$$\kappa_2(y) = (y^2 - 4)^{-1}(y^2 - By - D + (y^2 - 4)^{-1}(A^2 + C^2 + ACy))$$
or
$$\kappa_3(z) = (z^2 - 4)^{-1}(z^2 - Cz - D + (z^2 - 4)^{-1}(A^2 + B^2 + ABz)).$$

We also introduce three closely related polynomials, which arise naturally in the endgame:

$$f_1(x) = (x^2 - 4)^2 \kappa_1(x) = x^4 - Ax^3 - (D + 4)x^2 + (4A + BC)x + (4D + B^2 + C^2);$$
$$f_2(y) = (y^2 - 4)^2 \kappa_2(y) = y^4 - By^3 - (D + 4)y^2 + (4B + AC)y + (4D + A^2 + C^2);$$
$$f_3(z) = (z^2 - 4)^2 \kappa_3(z) = z^4 - Cz^3 - (D + 4)z^2 + (4C + AB)z +(4D + A^2 + B^2).$$

The following computation was done using Macaulay2 \cite{Macaulay2}:
\begin{lemma}\label{DeltaUniversality}
Let $\Delta_i = \Delta_i(A, B, C, D)$ denote the discriminant of $f_i$ with respect to its variable. Then we have
$$\Delta_1 = \Delta_2 = \Delta_3 = \Delta,$$
where $\Delta = \Delta(A, B, C, D)$ is the polynomial defined in Appendix A.
\end{lemma}
\begin{remark}
    The polynomial $\Delta$ also plays a role in understanding which of the surfaces $\mathcal{S}_{A, B, C, D}$ arise from generalized cluster algebras as in \cite{GyodaMatsushita} and \cite{pDivisibilityClusterAlg}; see Lemma \ref{ZariskiClose} and Proposition \ref{prop:SingularIffDelta} for details.
\end{remark}
\begin{remark}[Relation to the work of Cantat and Loray]\label{rmk:CantatLorayF}
The polynomials $f_i$ and the polynomial $\Delta$ already appear in the work of Cantat and Loray. Under the change of variables relating their normalization to ours (replace $(x,y,z,A,B,C)$ by $(-x,-y,-z,-A,-B,-C)$ in \cite[(1.4)]{CantatLoray}), their quartic $P_z$ of \cite[\S4.1]{CantatLoray} is our $f_3$, up to $z\mapsto -z$. They show there (Proposition 4.1 and Lemma 4.2) that for $z_0\neq\pm2$ the automorphism $g_z=V_2\circ V_1$ has a fixed point in the plane $\{Z=z_0\}$ lying on the surface if and only if $P_z(z_0)=0$, in which case the conic $\mathcal{S}_{A,B,C,D}\cap\{Z=z_0\}$ is a pair of lines meeting at that point, which is therefore fixed by both $V_1$ and $V_2$; this is the description of points fixed by two Vieta involutions given in Proposition \ref{DoubleFixedPoints}. They also observe that $P_z$ has a multiple root at each singular point, and (\cite[\S3.6]{CantatLoray}) in the discussion following Theorem 3.7, where the result is attributed to Benedetto and Goldman \cite{BenedettoGoldman}), that the set of parameters for which the surface is singular is the zero locus of the discriminant of $P_z$. Moreover they compute that this discriminant, pulled back by the cover $\Pi:(a,b,c,d)\mapsto(A,B,C,D)$ of degree $24$, equals $\frac{1}{16}(a^2-4)(b^2-4)(c^2-4)(d^2-4)\Delta_{BG}^2$, which is symmetric in $a,b,c,d$; as permuting $a,b,c,d$ permutes $A,B,C$, this implies (though they do not state it in this form) the equality of the discriminants of the three quartics $f_1,f_2,f_3$ of our Lemma \ref{DeltaUniversality}. The author arrived at the polynomials $f_i$ and $\Delta$ independently, in the course of the endgame and in the search for a sharp degeneracy condition respectively, but priority for these observations belongs to Cantat and Loray.
\end{remark}

\section{Endgame}\label{EndgameSection}
By the \textit{order} of an element $x \neq \pm2$ of $\mathbb{F}_p$, we mean the order of $\chi$ in $\mathbb{F}_p^\times$ or $\mu_{p+1}$ according as $x$ is hyperbolic or elliptic, where $x = \chi + \chi^{-1}$. If $x = \pm2$ we say its order is $p$. By the \textit{order} of a point $(x, y, z)\in \mathcal{S}_{A, B, C, D}^*(p)$ we mean the maximum of the orders of its coordinates. In this section, we show that the induced subgraph on the vertices of order at least $p^{\frac{1}{2} + \delta}$, save for a few ``forbidden'' vertices whose relevant high-order coordinate lies in a bounded exceptional set, is connected. There are two key steps, encoded in Propositions \ref{endgame} and \ref{cageConnectivity}.

Before stating and proving these Propositions, we must remark on a slight technical point. Recall from the introduction that we have defined $\Gamma$ to be the group generated by $V_1, V_2,$ and $V_3,$ $H$ to be the group generated by all double sign flips that preserve the set $\mathcal{S}_{A, B, C, D}$ (these occur only when at least two of $A$, $B$, or $C$ are zero), and $\Gamma'$ to be the group generated by $\Gamma$, $H$, and all transpositions or negated transpositions that preserve $\mathcal{S}_{A, B, C, D}$. Throughout this section, in cases where $\Gamma'$ is larger than $\Gamma$, it will be convenient to admit transformations from the larger group $\Gamma'$ when trying to join points, as when we can include a transposition or negated transposition, the dynamics on the conic sections becomes simpler. For example, as we saw above in Section \ref{ConicsSection}, if $B = C$, then the map $\tau_{yz}\circ V_2$ is a square root of $V_3\circ V_2$, and $\tau_{yz}\circ V_2$ acts transitively on $C_1(x)$ when it has maximal order, while $V_3\circ V_2$ never can. We will show below in Section \ref{Comparison}, specifically in Theorem \ref{comparisonOfGroupsThm}, that this cheat does not impact the form of our results---indeed, we shall show that if $\Gamma'$ acts transitively on $\mathcal{S}_{A, B, C, D}^*(p)$ for $p$ sufficiently large, so does $\Gamma$. 

The proof of the following proposition is based on that of Proposition 10 in \cite{BGS}; we largely mimic their notation as well. 
\begin{proposition}\label{endgame}
There exists an absolute constant $N$, independent of both $p$ and the parameters, such that for all but at most $N$ values of $x$, we have that if $(x,y,z)\in \mathcal{S}_{A, B, C, D}^*(p)$ with $(A, B, C, D)$ nondegenerate modulo $p$ and the order of the induced rotation in $x$ is at least $p^{\frac{1}{2} + \delta}$ ($\delta > 0$, fixed), then for $p$ sufficiently large (depending on $\delta$), $(x, y, z)$ is joined via edges from $\Gamma'$ to a point $(x', y', z')$ such that the group generated by $V_1$ and $V_3$ (and, if $A = \pm C,$ also $\pm\tau_{xz}$) acts transitively on $C_2(y')$. Moreover, any point whose relevant coordinate lies in the bounded exceptional set but has order at least $p^{3/4}$ may be connected to a point with a non-exceptional coordinate of order at least $p^{5/8}$ (and thus $>p^{1/2 + \delta}$).
\end{proposition}
\begin{proof}
We first write the proof in the case $A\neq \pm B$, $A\neq \pm C$, $B \neq \pm C$, and $x$ is either hyperbolic or elliptic; at the end of the proof we shall outline the other cases. From the very outset we discard the at most four values of $x$ such that $f_1(x) = 0$. Consider first the case where $x = \chi + \chi^{-1} $ is hyperbolic. In light of Lemma \ref{hyperbolics}, $(x, y, z)$ is connected by $V_3\circ V_2$ to the points of the form
$$(x, \alpha_1 t + \alpha_2 t^{-1} + \alpha_3, \alpha_4 t +\alpha_5 t^{-1} + \alpha_6)$$
with $t\in H \vcentcolon =\langle \chi^2 \rangle$. Let $e_H = \frac{p-1}{|H|}$. We must produce $t$'s in $H$ such that
$$\alpha_1 t + \alpha_2 t^{-1} + \alpha_3 = s + s^{-1}$$
with $s$ a primitive root in $\mathbb{F}_p$ such that
$$\left(\frac{f_2(s + s^{-1})}{p}\right) = -1;$$
let $P(H)$ be the number of such solutions. To do so, we use inclusion/exclusion on the subgroups of $\mathbb{F}_p^\times$. Each such subgroup is determined by its size $K$ (necessarily a divisor of $p-1$); let $d_K = \frac{p-1}{|K|}$. Let $f_H(K) = f_H(d_K)$ be the number of solutions to
$$\alpha_1 t + \alpha_2 t^{-1} + \alpha_3 = s + s^{-1},\, t\in H, s\in K,$$
with
$$\left(\frac{f_2(s + s^{-1})}{p}\right) = -1.$$
Letting $\omega$ denote a fixed quadratic nonresidue in $\mathbb{F}_p$, this last condition may be written
$$\omega u^2 = f_2(s + s^{-1}),\, u \neq 0.$$
By inclusion/exclusion we have
$$P(H) = \sum\limits_{d_K|(p-1)}\mu(d_K)f_H(d_K)$$
where $\mu$ is the M\"obius function. To show $P(H) > 0$ we therefore must estimate $f_H(d_K)$. We do this with Weil's bound. The map
$$\xi \mapsto \xi^{d_K},\;\eta\mapsto\eta^{e_H}$$
sends solutions of
$$C_{\alpha_1, \alpha_2, \alpha_3}:\;\alpha_1\eta^{e_H} + \alpha_2\eta^{-e_H} + \alpha_3 = \xi^{d_K} + \xi^{-d_K},\; \omega u^2 = f_2(\xi^{d_K} + \xi^{-d_K})$$
to solutions of
$$\alpha_1 t + \alpha_2 t^{-1} + \alpha_3 = s + s^{-1},\, t\in H, s\in K, \; \omega u^2 = f_2(s + s^{-1}),$$
and is $2e_Hd_K$ to 1.
As we prove below (Lemma \ref{endgameIrreducibility}), for all but at most 22 values of $x$, $C_{\alpha_1, \alpha_2, \alpha_3}$ is irreducible of genus $O(e_Hd_K)$, so that applying Weil's bound gives
$N(C_{\alpha_1, \alpha_2, \alpha_3}) = (p-1) + O(e_Hd_K\sqrt{p}).$
The points with $u=0$, as well as points over the poles/ramification points of the parameterizations, contribute $O(e_Hd_K)$, which is absorbed into the error term after dividing by $2e_H d_K$. Hence, 
$$f_H(K) = \frac{p}{2e_Hd_K} + O(\sqrt{p}),$$
and inclusion/exclusion yields
\[\begin{split}
    P(H) &= \sum\limits_{d_K|(p-1)}\mu(d_K)\left(\frac{|H|}{2d_K} + O(\sqrt{p})\right)\\
    &= |H|\sum\limits_{d_K|(p-1)}\frac{\mu(d_K)}{2d_K} + O_\epsilon\left(p^{\frac{1}{2} + \epsilon}\right)\\
    &= |H|\frac{\varphi(p-1)}{2(p-1)} + O_\epsilon\left(p^{\frac{1}{2} + \epsilon}\right)
\end{split}\]
where $\varphi$ is the Euler Totient function. As $\varphi$ satisfies $\varphi(n) = \Omega_{\epsilon}(n^{1-\epsilon}),$ the assumption $|H|\geq p^{\frac{1}{2} + \delta}$ implies that for $p$ large enough, $P(H) > 0.$

Now, suppose that $x$ is elliptic, so $x = v + v^p$ for $v\in\mathbb{F}_{p^2}\setminus\mathbb{F}_p$ with $v^{p+1} = 1$. If we were to proceed as in the hyperbolic case, we would need to count solutions to the system
$$\alpha_1\eta^{e_H} + \alpha_2\eta^{-e_H} + \alpha_3 = \xi^{d_K} + \xi^{-d_K},\; \omega u^2 = f_2(\xi^{d_K} + \xi^{-d_K}),\, \eta^{p+1} = 1$$
where $\eta$ lies in $\mathbb{F}_{p^2}$ but $\xi$ must lie in $\mathbb{F}_p$, which is not the sort of counting problem that can be directly attacked using the Weil bound. To get around this, we reparameterize the relevant curve in such a way that the solutions we're looking for may be described as $\mathbb{F}_p$-points of this curve. As the reparameterized curve is isomorphic to the original one, at least over $\overline{\mathbb{F}}_p,$ the proof of geometric irreducibility and the bounds on its genus from Lemma \ref{endgameIrreducibility} carry over.

To this end, fix a quadratic nonresidue $\omega\in\mathbb{F}_p$ and a choice of $\sqrt{\omega}\in\mathbb{F}_{p^2}$. In the setting of Lemma \ref{elliptics}, we write $t = a + b\sqrt{\omega}$ with $a, b\in \mathbb{F}_p$; the condition $t^{p+1} = N(t) = \kappa_1(x)$ then becomes
$$a^2 - \omega b^2 = \kappa_1(x).$$
Similarly, a subgroup $H$ of $\mu_{p+1}$ may be parameterized in an $e_H\vcentcolon = \frac{p+1}{|H|}$-to-one way as
$$\{(\xi + \eta\sqrt{\omega})^{e_H}\mid \xi, \eta\in\mathbb{F}_p,\, \xi^2 - \omega\eta^2 = 1\}.$$
Crucially, by the binomial theorem, if $g_{n}(\xi, \eta)$ and $h_n(\xi, \eta)$ are the integral polynomials defined by

\begin{equation}\label{ellipticsModificationPolys}\begin{split}
    g_n(\xi, \eta) &= \sum\limits_{i = 0}^{\lfloor\frac{n}{2}\rfloor} \binom{n}{2i}\omega^i\xi^{n-2i}\eta^{2i},\\
    h_n(\xi, \eta) &= \sum\limits_{i = 0}^{\lfloor\frac{n-1}{2}\rfloor} \binom{n}{2i+1}\omega^i\xi^{n-2i-1}\eta^{2i+1}
\end{split}\end{equation}
then
$$(\xi + \eta\sqrt{\omega})^n = g_n(\xi, \eta) + h_n(\xi, \eta)\sqrt{\omega}.$$

The second coordinates of the elements of the orbit of our base point then take the form

$$\text{Tr}\left((a + b\sqrt{\omega})(g_{e_H}(\xi, \eta) + h_{e_H}(\xi, \eta)\sqrt{\omega})\right) + \alpha_3 = 2ag_{e_H}(\xi, \eta) + 2b\omega h_{e_H}(\xi, \eta) + \alpha_3.$$

It then follows that for $x$ elliptic, with $H = \langle\chi^2\rangle$ and $K$ a subgroup of $\mathbb{F}_p^\times,$ the number of points on $\mathcal{S}_{A, B, C, D}(\mathbb{F}_p)$ with $X = x$ and $Y = y$ such that $y = t + t^{-1}$ with $t\in K$ and $f_2(y)$ a nonresidue, is given by $\frac{N(e_H, d_K)}{2e_Hd_K}$ where $N(e_H, d_K)$ is the number of solutions over $\mathbb{F}_p$ to the system

\begin{align*}
    &\xi^2 - \omega\eta^2 = 1\\
    & y = 2ag_{e_H}(\xi, \eta) + 2b\omega h_{e_H}(\xi, \eta) + \alpha_3\\
    & y = r^{d_K} + r^{-d_K}\\
    & f_2(y) = \omega v^2,
\end{align*}

where $a$ and $b$ satisfy
\[\begin{split}
    &a^2 - \omega b^2 = \kappa_1(x)\\
    &y_0 = 2a + \alpha_3.
\end{split}\]

This curve is isomorphic over $\overline{\mathbb{F}}_p$ to the curve considered in the hyperbolic case, so that, as before, for all but a bounded number of values of $x$, this is an irreducible curve, and applying Weil's bound and inclusion/exclusion gives the result.

We now discuss the case where extra transpositions are available. As long as the quadruple of parameters $(A, B, C, D)$ is nondegenerate, the above argument goes through even if $A = \pm B$ or $B = \pm C$ so long as $A \neq \pm C$, which is the hypothesis needed for Lemma \ref{endgameIrreducibility}. If on the contrary $A = \pm C$, then the map
\[
    \tau:(x, y, z)\mapsto (\pm z, y, \pm x)\\
\]
will also preserve $\mathcal{S}_{A, B, C, D}^*(p)$. In this case, the group generated by $V_1, V_3,$ and $\tau$ will act transitively on $C_2(y)$ as long as the corresponding value of $\pm y$ has maximal order: simply replace $V_3\circ V_1$ with $\tau\circ V_1$ as outlined in Section \ref{ConicsSection}. Thus, we can remove the condition for $f_2(y)$ to be a quadratic nonresidue, resulting in the kind of system considered in \cite{BGS}, Lemma 12; this will give an irreducible curve for all but a bounded number of values of $x$.

If $x = \pm2$, or if $x$ has order at least $p^\frac{3}{4},$ then so long as we can connect it to a non-forbidden hyperbolic or elliptic point of order at least $p^\frac{5}{8}$, we can connect it to the cage by way of this point. But this is straightforward: in either case, the set of $2^\text{nd}$-coordinates of points in the orbit of $(x, y, z)$ has size at least $cp^\frac{3}{4}$ for some absolute constant $c$, by the parameterizations in Lemmas \ref{hyperbolics}, \ref{elliptics}, and \ref{Parabolics}, whereas the collection of elements of $\mathbb{F}_p$ that are parabolic or have order at most $p^\frac{5}{8}$ has size at most 
$$2 + \frac{1}{2}\sum\limits_{\substack{d|(p^2-1)\\d\leq p^\frac{5}{8}}}d \leq 2 + p^\frac{5}{8}d(p^2 - 1) = O(p^{\frac{5}{8} + \epsilon}),$$
so the second coordinate of at least one of the points in an orbit of a parabolic point or a forbidden hyperbolic or elliptic point of order at least $p^\frac{3}{4}$ must be a non-forbidden hyperbolic or elliptic point of order at least $p^\frac{5}{8}.$

\end{proof}
\begin{lemma}\label{NonresidueCondition}
    Let $\omega\in\mathbb{F}_p$ be a fixed quadratic nonresidue, and assume that $(A, B, C, D)$ is nondegenerate modulo $p$. Then for $i = 1, 2, \text{ and }3$ the equation

    $$\omega \eta^2 = f_i(\xi)$$

    defines a geometrically irreducible affine curve $E$ which may be extended to a projective curve isomorphic to either an elliptic curve, or the nodal or cuspidal cubic. The map from $E\rightarrow \mathbb{P}^1$ given by $(\xi, \eta)\mapsto \xi$ is a two-sheeted ramified cover, ramified over the roots of odd multiplicity of the polynomial $f_i(\xi).$
\end{lemma}
\begin{proof}
    We prove the lemma for $i = 1;$ the other cases are identical. Under a change of variables defined over $\overline{\mathbb{F}}_p,$ the given equation may be written as
    \[\begin{split}\eta^2 &= \xi^4 - A\xi^3 - (D + 4)\xi^2 + (4A + BC)\xi + (4D + B^2 + C^2).\end{split}\]
    The proof is particularly easy if $\Delta(A, B, C, D) \neq 0$. Indeed, this implies that the roots of $f_1(\xi)$ are all distinct, so that $E$ is an irreducible smooth curve of genus $1$, and the specified map is ramified over the roots of $f_1$ and nowhere else by the argument in the proof of Proposition 19.5.2 of \cite{Vakil}. For the general case, the curve will be irreducible if and only if $f_1$ is not the square of some other polynomial, in which case at least one root of $f_1$ must have multiplicity one; moving this to infinity will then give singular Weierstrass equations, proving the rest of the lemma. Thus, to prove the lemma, it suffices to show that
    $$f_1(\xi) \neq (\xi^2 + a\xi + b)^2$$
    for any $a, b\in\overline{\mathbb{F}}_p.$
    If on the contrary $f_1(\xi) = (\xi^2 + a\xi + b)^2,$ then we have

    \[\begin{split}
        2a &= -A\\
        2b + a^2 &= -(D+4)\\
        2ab & = 4A + BC\\
        b^2 &= 4D + B^2 + C^2.
    \end{split}\]

    From the first and third equations, we have that

    $$-Ab = 4A + BC.$$

    It follows that if $A = 0,$ then either $B = 0$ or $C = 0.$ If $B=0$, then squaring the second equation and comparing it with the fourth we find that
    $$D^2 + 8D + 16 = 4b^2= 16D + 4C^2,$$
    whence
    $$4C^2 = (D-4)^2$$
    and so
    $$\pm2C = D-4,$$
    and
    $$4D + A^2 = \pm8C + 16,$$
    so that $(A, B, C, D)$ is degenerate. This argument also establishes degeneracy if $C = 0$. Thus, we may assume $A \neq 0,$ so that

    \[\begin{split}
        a &= -\frac{A}{2};\\
        b &= -4-\frac{BC}{A}.
    \end{split}\]

    We use the second equation to solve for $D$ in terms of $A, B,$ and $C$:

    \[\begin{split}
        D &= -4 - 2b - a^2\\
        &= 4 + \frac{2BC}{A} -\frac{A^2}{4}.
    \end{split}\]
    The fourth equation then implies
    $$4D + B^2 + C^2 = b^2 = 16 + \frac{8BC}{A} + \frac{B^2C^2}{A^2};$$

    Inputting our expression for $D$ gives
    $$16 + \frac{8BC}{A} - A^2 + B^2 + C^2 = 16 + \frac{8BC}{A} + \frac{B^2C^2}{A^2},$$

    which simplifies to

    $$A^2B^2 + A^2C^2 - A^4 - B^2C^2 = 0.$$

    Finally, we factor as
    $$-(A^2 - B^2)(A^2 - C^2) = 0$$

    so that $A = \pm B$ or $A = \pm C$. If $A = \pm B \neq 0$, then the third equation implies that

    $$b = -4 \mp C.$$

    Squaring and comparing with the final equation gives

    $$16 \pm 8C + C^2 = 4D + A^2 + C^2,$$
    once again implying degeneracy; a similar argument gives degeneracy under the condition $A = \pm C$. This proves the lemma. 
\end{proof}

We are now in a position to prove the irreducibility result used in the endgame. Our proof is by monodromy; cf. \cite{CasselsPolynomials} and \cite{Pakovich} for the technique over the complex numbers, and \cite{Fried} and Chapter 6 of \cite{LMT} for modifications in characteristic $p$.

\begin{lemma}\label{endgameIrreducibility}
Assume that $(A, B, C, D)$ is nondegenerate modulo $p$ and $A \neq \pm C$, and that $e$ and $d$ are positive integers with $p\nmid ed$. Then for all but at most $22$ values of $x$, the system of equations
$$y = t^d + t^{-d}$$
$$y = \alpha_1 s^e + \alpha_2 s^{-e} + \alpha_3$$
$$\omega\eta^2 = f_2(y)$$

where $\omega$ is a fixed quadratic nonresidue mod $p$, $\alpha_1\alpha_2 = \kappa_1(x),$ and $\alpha_3 = (x^2 - 4)^{-1}(-2B - Cx)$ defines a geometrically irreducible curve of genus $O(ed)$. The forbidden values of $x$ are independent of the values of $e$ and $d$.
\end{lemma}
\begin{proof}
Let $g(s) = s^d + s^{-d}$ and $h(t) = \alpha_1t^e + \alpha_2 t^{-e} + \alpha_3$. The curve $X$ defined by the above system of equations may be viewed as the fiber product $\mathbb{P}_{t}^1\times_{\mathbb{P}_y^1}\mathbb{P}_{s}^1\times_{\mathbb{P}_y^1} E$ where $E$ is the curve from Lemma \ref{NonresidueCondition} with the cover as described there, and the covers $\mathbb{P}_{t}^1\rightarrow \mathbb{P}_{y}^1$ and $\mathbb{P}_{s}^1\rightarrow \mathbb{P}_{y}^1$ are given by
\[\begin{split}
g: \mathbb{P}_{s}^1&\rightarrow \mathbb{P}_{y}^1\; \;\;\;\;\;\;g(s) - y = 0\\
h: \mathbb{P}_{t}^1&\rightarrow \mathbb{P}_{y}^1\; \;\;\;\;\;\; h(t) - y = 0.
\end{split}\]
The condition $p\nmid ed$ ensures that ramification is tame. As in \cite{BGS}, Lemma 12, the branch points for $g$ are $\{-2, 2, \infty\}$ with branch cycles given by
\begin{equation}
\begin{cases} \sigma_{-2} = (12)(34)\dots(2d-1\, 2d)\\
\sigma_2 = (1\,2d)(23)\dots(2d-2\, 2d-1)\\
\sigma_\infty = (135\dots2d-1)(246\dots2d),\end{cases}
\end{equation}
while the branch points for $h$ are $\{-2\sqrt{\alpha_1\alpha_2} + \alpha_3, 2\sqrt{\alpha_1\alpha_2} + \alpha_3, \infty\}$ with branch cycles given by
\begin{equation}
\begin{cases} \sigma_{-2\sqrt{\alpha_1\alpha_2} + \alpha_3} = (12)(34)\dots(2e-1\, 2e)\\
\sigma_{2\sqrt{\alpha_1\alpha_2}+\alpha_3} = (1\,2e)(23)\dots(2e-2\, 2e-1)\\
\sigma_\infty = (135\dots2e-1)(246\dots2e).\end{cases}
\end{equation}
Finally, the branch points for $E\rightarrow \mathbb{P}_y^1$ are given by the roots of $f_2(y)$, as in Lemma \ref{NonresidueCondition}, with branch cycle about root $r$ given by $(12)^m$ where $m$ is the multiplicity of $r$ as a root of $f_2(y)$. 

The curve $X$ will be irreducible if and only if the subgroup of the product of the monodromy groups generated by the branch cycles acts transitively on the $2\cdot 2d\cdot 2e$ sheeted covering. This is immediate so long as $\pm2\sqrt{\alpha_1\alpha_2} + \alpha_3 \neq \pm 2$, and at least one root of $f_2(y)$ of odd multiplicity is not $\pm 2$ or $\pm2\sqrt{\alpha_1\alpha_2} + \alpha_3$. The first condition is equivalent to $4\alpha_1\alpha_2 = (\pm 2 - \alpha_3)^2.$ Keeping in mind that $\alpha_1\alpha_2 = \kappa_1(x)$ is a degree 4 rational function of $x$, shown not to be a perfect square in the proof of Lemma \ref{NonresidueCondition}, the first condition forbids at most 16 values of $x$. As for the second, if $\pm 2$ is a root of $f_2(y)$, then

\[\begin{split}0 &= 16 \mp 8B - 4(D+4) \pm 2(4B + AC) + (4D + A^2 + C^2) \\
&= A^2 \pm 2AC + C^2
\end{split}\]
violating the assumptions of the lemma. If $\Delta \neq 0$, this suffices, as $f_2(y)$ will have four distinct roots, each of (odd) multiplicity one. If $\Delta = 0$, then $f_2(y)$ has exactly two roots of odd multiplicity by Lemma \ref{NonresidueCondition}, and these might indeed be $\pm2\sqrt{\alpha_1\alpha_2} + \alpha_3$. However, if they are, then the polynomial
$$y^2 - 2\alpha_3y + (\alpha_3^2 - 4\alpha_1\alpha_2)$$
must divide $f_2(y)$ (and in fact be the polynomial having as roots all roots of $f_2(y)$ with odd multiplicity). Keeping in mind that $\alpha_3$ and $\alpha_1\alpha_2$ are both rational functions of $x$, with $\alpha_3$ of degree 1 and $\alpha_1\alpha_2$ of degree 4, this can occur for at most 6 values of $x$. This proves irreducibility.
Bounds on genus follow from the Riemann-Hurwitz formula after analyzing the ramification of the map $X\rightarrow\mathbb{P}_y^1$. If $X$ is not smooth, then we may pass to its normalization, and still find that the genus is $O(ed).$ As the normalization map is generically one to one, this will not disrupt the application of the Weil bound in Proposition \ref{endgame}, though it may increase the implicit constant. 
\end{proof}
\subsection{Connectivity of the Cage}\label{cageConnectivitySubsect}
We say that a point $(x_1, x_2, x_3)$ is $i$-connecting if the subgraph induced by $C_i(x_i)$ is connected. This occurs exactly in the parabolic case when $x_i = \pm2$ and $C_i(x_i)$ has exactly $p$ points, in the hyperbolic and elliptic cases when $x_i$ has order $p\pm1$ with 
$$\left(\frac{f_i(x_i)}{p}\right) = -1,$$
or, if the linear coefficients on $x_{i-1}$ and $x_{i+1}$ (indices taken cyclically mod 3) are equal or equal to each other's negatives, simply when $x_i$ is hyperbolic or elliptic of maximal order; see Section \ref{ConicsSection} and especially Proposition \ref{prop:SingleOrbitAndKappa}. We define the \textit{cage} $\mathcal{K}_{A, B, C, D}(p)$ to be the set of points that are $1$-, $2$-, or $3$-connecting; when the parameters are fixed we will simply use $\mathcal{K}(p).$ In Proposition \ref{endgame}, it is shown that any point of sufficiently high order connects to the cage. In this subsection, we show that the cage is connected; this serves as the ``skeleton'' of the graph, to which almost all other points will ultimately be connected. 
\begin{proposition}\label{cageConnectivity}
Suppose that $(x, y, z)$ and $(x', y', z') \in\mathcal{K}(p)$. Then there is a path from $(x, y, z)$ to $(x', y', z'),$ and moreover, one can assume this path only involves vertices of order at least $p^\frac{5}{8}.$
\end{proposition}
\begin{proof}
It suffices to show this assuming that $(x, y, z)$ is 1-connecting and $(x', y', z')$ is 3-connecting: by symmetry, one can then connect any $i$-connecting point to any $j$-connecting point so long as $i\neq j$; then given a pair of $i$-connecting points, one can connect each to a common, arbitrarily chosen $j$-connecting point, and thereby to each other. Further, we may assume that $x$ and $z'$ are hyperbolic, as parabolic and elliptic connecting coordinates are connected to hyperbolic connecting elements in the proof of Proposition \ref{endgame}.

We therefore seek a value of $\tilde{y}$ such that there is a
\[P \in C_1(x) \cap C_2(\tilde y)\]
\[Q \in C_3(z') \cap C_2(\tilde y)\]
and $\tilde y$ is 2-connecting, i.e. $\tilde y$ has maximal order and $f_2(\tilde y)$ is a quadratic nonresidue. As in the proof of Proposition \ref{endgame}, we can do this using inclusion/exclusion so long as we can describe the conditions
\[P \in C_1(x) \cap C_2(\tilde y) \]
\[Q \in C_3(z') \cap C_2(\tilde y)\]
with a system of polynomial equations, and then show that this system, together with the equations 
\[\begin{split}
y &= t^\ell + t^{-\ell}\\
\omega u^2 &= f_2(y)
\end{split}\] 
with $\omega$ a fixed quadratic nonresidue, cut out an irreducible curve.

By the quadratic formula, there exists $\tilde y$ with
\[P \in C_1(x) \cap C_2(\tilde y) \]
if and only if for some $v\in\mathbb{F}_p$ we have
$$(x^2 - 4)\tilde y^2 + (2Cx + 4B)\tilde y - v^2 = (4x^2 - 4Ax - 4D - C^2).$$
Dividing through by $(x^2 - 4)$, we arrive at a curve of the form
$$\tilde y^2 + A_1\tilde y - B_1 v^2 = C_1.$$
Similarly, the condition 
\[Q \in C_3(z') \cap C_2(\tilde y)\]
leads to an equation of the form
$$\tilde y^2 + A_2\tilde y - B_2w^2 = C_2.$$
We are thus led to consider the system of equations
\[\begin{split}
    &\omega u^2 = f_2(\tilde y)\\
    &\tilde y^2 + A_1\tilde y - B_1 v^2 = C_1\\
    &\tilde y^2 + A_2\tilde y - B_2w^2 = C_2\\
    &\tilde y = t^\ell + t^{-\ell}\\
\end{split}\]
As in the proof of Lemma \ref{endgameIrreducibility}, this system describes a fiber product of four ramified covers of $\mathbb{P}_{\tilde y}^1$; if the fiber product is smooth and irreducible, the Riemann-Hurwitz formula implies that the genus is $O(\ell),$ and if the fiber product is nonsmooth, we may pass to its normalization while preserving this bound on the genus and still apply the sieving argument.
Each of the first three equations corresponds to a double-sheeted cover. The first equation corresponds to a cover ramified at the roots of $f_2(\tilde y)$ of odd multiplicity, with each ramification cycle swapping the two sheets. The second and third are ramified at
$$\frac{-A_1\pm\sqrt{A_1^2 + 4C_1}}{2}$$
and
$$\frac{-A_2\pm\sqrt{A_2^2 + 4C_2}}{2},$$
respectively, with each ramification cycle swapping the two sheets, unless $A_1^2 = -4C_1$ (respectively, unless $A_2^2 = -4C_2$). 
The last equation is ramified at $\pm 2$ and $\infty.$

The first three double covers are not ramified at $\infty$, so that as long as none of the ramification points of the first three equations are $\pm2$, and as long as each of the first three equations has a ramification point not shared by any of the other equations, it is trivial that the product of the monodromy groups of each ramified cover acts transitively on the product of the sheets, and so the curve is irreducible. As in the proof of Lemma \ref{endgameIrreducibility}, $\pm 2$ is not a root of $f_2(\tilde y)$ unless $A = \pm C$, in which case acting on $C_2(\tilde y)$ by $\tau\circ V_3$ instead of $V_1\circ V_3$, and demanding that $-\tilde{y}$ has maximal order if $A = C\neq 0$ as in the proof of Proposition \ref{endgame} gives a transitive action, and we may remove the first equation from the system entirely. 

Now, $A_1$ and $C_1$ are rational functions of $x$, each with denominator $x^2 - 4$ and with numerators of degree 1 and 2, respectively, and similarly $A_2$ and $C_2$ are rational functions of $z'$, each with denominator $z'^2 - 4$ and with numerators of degree 1 and 2, respectively. Thus, by avoiding $O(1)$ values of $x$ and $z'$, we may ensure $A_1^2 \neq -4C_1$ and $A_2^2 \neq -4C_2.$ Similarly, we shall show that for all but $O(1)$ values of $x$ and all but $O(1)$ values of $z'$, $\pm 2$ is not a ramification point of the cover associated to the second and third equations, nor are any of the roots of $f_2(\tilde y).$ Indeed, the condition for fixed $c$ to be a ramification point of the cover associated to the second equation has the shape
$$(2c + A_1)^2 = A_1^2 + 4C_1,$$
which simplifies to
$$cA_1 + c^2 = C_1,$$
which after multiplication by $(x^2 - 4)$ gives a degree $2$ polynomial in $x$ unless $c = \pm 2.$ For $c = \pm2$, it leads to a linear equation in $x$ unless $A = \mp C$, in which case the equation reduces to the degeneracy condition
$$4D + C^2 = \mp8B + 16,$$
which we have assumed does not hold. Therefore avoiding a bounded number of values of $x$ allows one to ensure that none of the ramification points of the second equation is $\pm2,$ and none of the ramification points of the second equation is also a ramification point of the first.

By avoiding a bounded number of values of $z'$ one can assume the same for the first and third equations. Finally, unless 
$$A_1 = A_2$$
and
$$C_1 = C_2,$$
the pair of ramification points associated to the second equation cannot be the same as the pair of ramification points associated to the third equation. $A_1 = A_2$ only occurs if $A = B = C = 0,$ in which case $C_1 = C_2$ requires $D = 4,$ the generally forbidden Cayley parameters. It follows that for each permitted value of $x$, for all but $O(1)$ values of $z'$, the relevant curve is irreducible, and $(x, y, z)$ connects to $(x', y', z')$ in the cage.

In the endgame (Proposition \ref{endgame} and its proof), it is shown that any point of order at least $p^{\frac{3}{4} + \delta}$ connects to at least $cp^{\frac{5}{8}-\epsilon} + O_\epsilon(p^{\frac{1}{2}+\epsilon})$ points in the cage.\footnote{This count of points is the worst case behavior, where we have to pass to a non-forbidden point of order $p^\frac{5}{8}$ before applying the Weil bound. Typically, the endgame shows one can connect a point of order $p^\frac{1}{2} + \delta$ to at least $cp^{\frac{1}{2} + \delta -\epsilon} + O_\epsilon(p^{\frac{1}{2}+\epsilon})$.} Thus, if either $x$ or $z'$ is in the globally forbidden set of values, then we can connect it to a point in the cage with no coordinate in the forbidden set via the endgame, and similarly, if $z'$ is in the set of $x$-dependent forbidden values, we can connect it to a point in the cage not having problematic coordinates via the endgame. This proves the proposition.

\end{proof} 

\section{Comparison of  \texorpdfstring{$\Gamma$}{Gamma} and  \texorpdfstring{$\Gamma '$}{Gamma'}}\label{Comparison}
 As discussed at the beginning of Section \ref{EndgameSection}, throughout the endgame we have availed ourselves of the potentially larger group of transformations $\Gamma'$, including when necessary double sign-flips and transpositions in addition to Vieta involutions. The following result lets us transfer transitivity from the action of the group $\Gamma'$ to our primary group of interest $\Gamma$. The proof is inspired by the remarks at the end of the proof of Theorem 4.7 in \cite{MartinBigPaper}.
\begin{theorem}\label{comparisonOfGroupsThm}
    Suppose that \((A, B, C, D)\) is nondegenerate modulo $p$, that $\Gamma'$ acts transitively on $\mathcal{S}_{A, B, C, D}^*(p)$, and that $p$ is sufficiently large (independent of $A, B, C,$ and $D$). Then $\Gamma$ also acts transitively on $\mathcal{S}_{A, B, C, D}^*(p).$
\end{theorem}
\begin{proof}    
    Our proof proceeds in steps: first, we show that the extension of $\Gamma$ by $H$ is harmless; then, we shall deal with transpositions, first in the case where only one is present, and then in the case where all three are present. 
    \begin{lemma}
        Using the notation defined above, we have that $\langle \Gamma, H\rangle\cong\Gamma\times H.$
    \end{lemma}
    \begin{proof}
        Elements of $H$ commute with each other and with the Vieta involutions. As $\Gamma\cong (\mathbb{Z}/2)*(\mathbb{Z}/2)*(\mathbb{Z}/2)$ (\cite{elHuti}, Theorem 1), it has a trivial center, so that $\Gamma\cap H$ is trivial. The result follows by elementary group theory. 
    \end{proof}
    \begin{lemma} \label{lemma:RemovalOfH}
        Suppose that $\Gamma\times H$ acts transitively on $\mathcal{S}_{A, B, C, D}^*(p).$ Then so does $\Gamma$.
    \end{lemma} 
    \begin{proof}
        As $\Gamma \unlhd \Gamma \times H$, we have that $(\Gamma\times H)/\Gamma \cong H$ acts on the $\Gamma$ orbits in $\mathcal{S}_{A, B, C, D}^*(p).$ By our assumption that $\Gamma\times H$ acts transitively on $\mathcal{S}_{A, B, C, D}^*(p),$ $H$ acts transitively on the set of $\Gamma$-orbits; we aim to show that there is only one $\Gamma$-orbit. By the orbit-stabilizer theorem, there can be either one, two, or four orbits. Further, if there is more than one orbit, there must be some element of $H$ which acts nontrivially on every single orbit. This is clear if there are two orbits (some element of $H$ must interchange them); if there are four orbits then $|H| = 4$ and the action of $H$ on the collection of $\Gamma$-orbits must be isomorphic to that of $H$ on itself by translation, and thus no nonidentity element can stabilize any $\Gamma$-orbit. 

        However, using the parameterizations of the conic sections from Lemmas \ref{hyperbolics} and \ref{elliptics}, we can show that each element of $H$ stabilizes some $\Gamma$-orbit; we will write the argument for $\operatorname{neg}_{yz}$ (implicitly assuming we're in the setting where $B = C = 0$). For $p$ sufficiently large, the arguments of the endgame guarantee the existence of a point $(x,y,z)\in\mathcal{S}_{A,B,C,D}^*(p)$ with $x=\chi+\chi^{-1}$ and $4\mid\operatorname{ord}(\chi)$.  In this setting, the parameterizations in Lemmas \ref{hyperbolics} and \ref{elliptics} give us that if $x = \chi + \chi^{-1}$ with the order of $\chi$ divisible by 4, then $(x, y, z)$ can be connected to $(x, -y, -z)$ using only iterations of $V_2\circ V_3.$ In particular, the $\Gamma$-orbit of $(x, y, z)$ must be stabilized by $\operatorname{neg}_{yz}$. A similar argument shows that $\operatorname{neg}_{xy}$ and $\operatorname{neg}_{xz}$ each stabilize a $\Gamma$-orbit whenever the parameters are such that these automorphisms are able to act on $\mathcal{S}_{A, B, C, D}.$ This proves the result.
    \end{proof}

    Next, we show that we can remove the transpositions and negated transpositions. We shall need the following lemma:
    \begin{lemma}\label{lemma:TranspositionFixedPointsExist}
    Suppose that $A = B \neq 0$. Then for $p$ sufficiently large (independent of the parameters) there is a point $v\in\mathcal{S}_{A, B, C, D}^*(p)$ which is fixed by $\tau_{xy},$ i.e., $v = (x, x, z).$ 

    Suppose that $A = -B \neq 0.$ Then for $p$ sufficiently large (independent of the parameters) there is a point $v\in\mathcal{S}_{A, B, C, D}^*(p)$ which is fixed by $\operatorname{neg}_{xy}\circ\tau_{xy},$ i.e., $v = (x, -x, z).$

    Suppose that $A = B = 0$ (so that $A = -B$ as well). Then for $p$ sufficiently large (independent of the parameters) there is a point $v\in\mathcal{S}_{A, B, C, D}^*(p)$ which is fixed by one of $\tau_{xy}$ or $\operatorname{neg}_{xy}\circ\tau_{xy},$ i.e., $v = (x, x, z)$ or $v=(x, -x, z).$
    \end{lemma}
    \begin{proof}
    This is a straightforward application of the Weil bound---points of the desired form on $\mathcal{S}_{A, B, C, D}(\mathbb{F}_p)$ correspond to points on curves whose geometric irreducibility we can establish, and whose degree (and thus genus) is bounded independent of the parameters and of the prime $p$; these curves must then have roughly $p$ points, so that for $p$ large enough there are points which survive the removal of small orbits in the passage to $\mathcal{S}_{A, B, C, D}^*(p).$ Suppose first $A = B \neq 0$. Points fixed by $\tau_{xy}$ correspond to the points of the curve
    \begin{equation}\label{eq:tauFixed}
        2X^2 + Z^2 - X^2Z - 2AX - CZ - D = 0,
    \end{equation}
    a proof of whose geometric irreducibility we now sketch. As $A \neq 0,$ the polynomial on the left-hand side is not divisible by any polynomial of the form $Z - z_0$ or $X - x_0$, as evaluation of $X$ or $Z$ at a constant value cannot eliminate the $Z^2$ and $2AX$ terms. By considering total degree, $X$ degree, and $Z$ degree one may show that the only possible remaining factorization is one of the form
    $$(aX + bZ + c)(dX + eZ + fXZ + g) = 2X^2 + Z^2 - X^2Z- 2AX - CZ - D,$$
    and looking at terms of highest degree we find that $af = -1,$ $be = 1$, yet $bf = 0,$ plainly a contradiction.
    
    If instead $A = -B \neq 0$, the relevant curve is of the form
    \begin{equation}\label{eq:negTauFixed}
        2X^2 + Z^2 + X^2Z - 2AX - CZ - D = 0
    \end{equation}
    and we can argue exactly as before. 
    
    If $A = B = 0$, then the only part of the preceding argument which fails is the portion ruling out divisibility of the polynomials (\ref{eq:tauFixed}) and (\ref{eq:negTauFixed}) by polynomials of the form $Z-z_0.$ By evaluation we see the only such possibilities are for $Z-2$ to divide the left hand side of (\ref{eq:tauFixed}) and for $Z + 2$ to divide the left hand side of (\ref{eq:negTauFixed}). Again by evaluation, the former requires $D = -2C + 4,$ while the latter requires $D = 2C + 4$. These can only occur simultaneously if $C = 0$ and $D = 4$, in which case the surface $\mathcal{S}_{A, B, C, D}$ is the Cayley cubic, which is generally forbidden.
    \end{proof}
    
    With this lemma now in hand, we can start the passage from $\Gamma'$ to $\Gamma\times H$. Once again, the key fact is that $(\Gamma\times H)\unlhd \Gamma';$ we will use the (transitive) action of the quotient $\Gamma'/(\Gamma\times H)$ on the $(\Gamma\times H)$-orbits in $\mathcal{S}_{A, B, C, D}^*(p)$ together with orbit-stabilizer. 
    \begin{lemma}\label{lemma:OneTranspositionCase}
        Suppose that $\Gamma'/(\Gamma\times H)\cong\mathbb{Z}/2;$ i.e., there is only one transposition or negated transposition in $\Gamma'.$ Suppose also that $\Gamma'$ acts transitively on $\mathcal{S}_{A, B, C, D}^*(p).$ Then for $p$ sufficiently large (independent of the parameters), so does $\Gamma\times H$.
    \end{lemma}
    \begin{proof}
        By orbit-stabilizer, there are at most two $(\Gamma\times H)$-orbits, and if there are two, the transposition or negated transposition interchanges them. By Lemma \ref{lemma:TranspositionFixedPointsExist}, for $p$ sufficiently large, there is a point $w$ of $\mathcal{S}_{A, B, C, D}^*(p)$ fixed by this transposition or negated transposition. The (possibly negated) transposition cannot act nontrivially on the orbit containing $w$, so there cannot be two $(\Gamma\times H)$-orbits. The result follows.
    \end{proof}

    If there is more than one transposition, then we shall need to work a bit more structurally. We certainly have that $A = \pm B = \pm C$; any setting where this is the case is equivalent to one where $A = B$ and $C = \pm B;$ we write the proof in this case. Here we have that $\Gamma'/(\Gamma\times H)\cong S_3$, containing the three involutions $\tau_{xy}, \pm\tau_{yz},$ and $\pm\tau_{xz},$ where the sign is chosen consistently with $C = \pm B.$ Under the assumption that $\Gamma'$ acts transitively, there are either one, two, three, or six $(\Gamma\times H)$-orbits. In fact, the following lemma implies that the six-orbit and two-orbit cases are impossible.
    
    \begin{lemma}\label{lemma:TwoAndSixOrbitsImpossible}
        Suppose that $A = B = \pm C,$ and that $\Gamma'$ acts transitively on $\mathcal{S}_{A, B, C, D}^*(p).$ Then there cannot be exactly two or exactly six $(\Gamma\times H)$-orbits.
    \end{lemma}\begin{proof}
        Any transitive action of $S_3$ on a six-element set must have trivial point-stabilizers by the orbit-stabilizer theorem, and any transitive action of $S_3$ on a two-element set factors through the quotient $C_2$. In particular, in these cases, no transposition can stabilize any orbit. But this contradicts Lemma \ref{lemma:TranspositionFixedPointsExist}, as in the proof of Lemma \ref{lemma:OneTranspositionCase}.
    \end{proof}  
    We are left with the task of eliminating the case of three orbits. This requires a bit more effort, which is not a coincidence: in the degenerate four-orbit case, three of the orbits ($\mathcal{S}_2,\mathcal{S}_3$, and $\mathcal{S}_4$) are joined when passing to the larger action of $\Gamma'$, and all of the properties of the action of $\Gamma'$ on nondegenerate $\mathcal{S}_{A, B, C, D}^*(p)$ which we have used so far (namely, normality of $\Gamma\times H$ and existence of fixed points under each $\tau$) obtain for this combined orbit in the degenerate case as well. We break the argument into two lemmas:

    \begin{lemma}\label{lemma:oddOrderSymmetryAvoidsTranspositions}
        Suppose that $A = B = \pm C,$ and that the parameters $(A, B, C, D)$ are nondegenerate. Suppose there exists a point of the form $(x, \pm z, z)\in\mathcal{S}_{A, B, C, D}^*(p)$ where $z$ is hyperbolic or elliptic, of odd order, and $\kappa_3(z) \neq 0$ (the sign being taken consistently with that in $A = \pm C$; if $A = B = C = 0,$ we can take either sign). Then there is a path from $(x, \pm z, z)$ to $(\pm z, x, z)$ using only edges arising from Vieta involutions. 
    \end{lemma}
    In order to use this lemma, we need another:
    \begin{lemma}\label{lemma:oddOrderSymmetryExists}
         Suppose that $A = B = \pm C,$ and that the parameters $(A, B, C, D)$ are nondegenerate. Then for $p$ large enough (independent of the parameters) there exists a point of the form $(x, \pm z, z)\in\mathcal{S}_{A, B, C, D}^*(p)$ where $z$ is hyperbolic or elliptic, of odd order, and $\kappa_3(z) \neq 0$.
    \end{lemma}
    Once we establish these two lemmas, we can rule out the case of 3 $(\Gamma\times H)$-orbits relatively quickly: by Lemmas \ref{lemma:oddOrderSymmetryAvoidsTranspositions} and \ref{lemma:oddOrderSymmetryExists}, there is a $(\Gamma\times H)$-orbit containing points fixed by each of the two transpositions $\pm\tau_{yz}$ and $\pm\tau_{xz},$ so that these two maps also stabilize this $(\Gamma\times H)$-orbit setwise. But in any transitive action of $S_3$ on a three-element set, each element is stabilized by exactly one of the transpositions in $S_3.$
    \begin{remark}[Comparison with the four-orbit degenerate case]
        Points satisfying the conclusions of Lemma \ref{lemma:oddOrderSymmetryExists} do exist when $(A, B, C, D)$ are degenerate parameters producing four orbits. However, they are all confined in $\mathcal{S}_1,$ so that the arguments of this section do not apply to the $\Gamma'$-orbit $\mathcal{S}_2\cup\mathcal{S}_3\cup\mathcal{S}_4$.
    \end{remark}

    We now prove the lemmas.
    \begin{proof}[Proof of Lemma \ref{lemma:oddOrderSymmetryAvoidsTranspositions}]
        Using the parameterizations of the conic sections in Lemmas \ref{hyperbolics} and \ref{elliptics}, we find that our point $P = (x, \pm z, z)$ corresponds either to some $t\in\mathbb{F}_p^*$ or to some $t\in\mathbb{F}_{p^2}$ with $t^{p+1} = \kappa_3(z),$ with the identities
        \begin{align*}
            x &= t + \frac{\kappa_3(z)}{t} + (z^2 - 4)^{-1}(-2A - Az)\\
            \pm z &= \chi t + \frac{\kappa_3(z)}{\chi t} + (z^2 - 4)^{-1}(-2A - Az)
        \end{align*}
        where $z = \chi + \chi^{-1},$ and the map $D_3$ acts as $t\mapsto \chi^2t.$ As $\chi$ has odd order, we have that $\chi^{2k+1} = 1$ for some positive integer $k$, whence the first coordinate of $D_3^{k+1}P$ is
        $$\chi^{2k+2}t + \frac{\kappa_3(z)}{\chi^{2k+2}t} + (z^2 - 4)^{-1}(-2A - Az) = \chi t + \frac{\kappa_3(z)}{\chi t} + (z^2 - 4)^{-1}(-2A - Az) = \pm z.$$
        If the second coordinate of $D_3^{k+1}P$ is $x$, we are done; if not, since $x$ is certainly a solution to the quadratic obtained from the equation defining $\mathcal{S}_{A, B, C, D}$ by setting $X = \pm z$ and $Z = z$, $V_2D_3^{k+1}P$ will have second coordinate $x$ and first and third coordinates unchanged. The result follows.
    \end{proof}

    \begin{proof}[Proof of Lemma \ref{lemma:oddOrderSymmetryExists}]
        Choose $s = \pm1$ so that $C = sA$. As in the proof of Lemma \ref{lemma:TranspositionFixedPointsExist}, to find a point of the form $(x, sz, z),$ we seek points of the curve
        \begin{equation}\label{eq:FixedByATransposition}
            X^2 + 2Z^2 - sXZ^2 - AX - 2sAZ - D = 0,
        \end{equation}
        but we now want to add in the demand that $z$ have odd order. Depending on the congruence class of $p\pmod{4}$, we will choose to make $z$ hyperbolic or elliptic to make this easy to achieve. First suppose that $p\equiv 3\pmod{4}.$ Then $\frac{p-1}{2}$ is odd, so that if $z = t^2 + t^{-2}$ for $t\in\mathbb{F}_p^*,$ $z$ will be hyperbolic of odd order. As in the proof of Lemma \ref{lemma:TranspositionFixedPointsExist}, the constraints $A = B = \pm C$ and $(A, B, C, D)\neq (0, 0, 0, 4)$ following from nondegeneracy imply that the curve cut out by (\ref{eq:FixedByATransposition}) is irreducible. We therefore seek to use the Weil bound to show that the system
        \begin{equation*}
        \begin{split}
            &Z = t^2 + t^{-2}\\
            &X^2 + 2Z^2 - sXZ^2 - AX - 2sAZ - D = 0
        \end{split}
        \end{equation*}
        has approximately $p$ solutions, so that there are approximately $p/4$ choices of $x$ and $z$ satisfying the conclusions of this lemma in $\mathcal{S}_{A, B, C, D}(\mathbb{F}_p)$, and some such choices remain after removing the finite orbits and the points for which $z=\pm2$ or $\kappa_3(z) = 0.$ As in the endgame, we realize this system as the fiber product of ramified covers of $\mathbb{P}^1_Z$. As in the endgame, the equation
        $$Z = t^2 + t^{-2}$$
        ramifies over $2, -2$, and $\infty$; one can show that the equation (\ref{eq:FixedByATransposition}) is a double cover of $\mathbb{P}^1_Z$ which extends to have two points at infinity, so that the cover does not ramify over $\infty$. By re-writing (\ref{eq:FixedByATransposition}) as
        $$X^2 + (-A- sZ^2)X + (2Z^2 - 2sAZ - D) = 0,$$
        we find that the finite ramification points of (\ref{eq:FixedByATransposition}) over $\mathbb{P}^1_Z$ are exactly the solutions of
        \begin{equation}\label{eq:RamificationAnnoyingLemma}
            (-A - sZ^2)^2 - 4 (2Z^2 - 2sAZ - D) = A^2 + 4D + sA(2Z^2 + 8Z) + (Z^4 - 8Z^2) = 0.
        \end{equation}
        As long as at least one of $2$ or $-2$ is not a root of (\ref{eq:RamificationAnnoyingLemma}), the fiber product is geometrically irreducible, as in the arguments of the endgame, and we are done. Nondegeneracy guarantees we do not have ramification over $Z = -2$: substituting $Z = -2$ into (\ref{eq:RamificationAnnoyingLemma}) gives
        $$A^2 + 4D - 8sA - 16 = 0,$$
        which may be rearranged as
        $$4D + A^2 = 8sA + 16 = 8C + 16,$$
        which is exactly the condition for the parameters $(A, B, C, D)$ to be degenerate.

        If in contrast $p \equiv 1 \pmod{4},$ then $\frac{p+1}{2}$ is odd, so that if $z = t^2 + t^{-2}$ for $t\in\mathbb{F}_{p^2}$ with $t^{p+1} = 1,$ $z$ will be elliptic of odd order. As in the elliptic case of the endgame, we merely need to re-write the system in such a way that the coefficients are all over $\mathbb{F}_p,$ such as by re-writing it in the form
        \begin{align*}
            &X^2 + 2Z^2 - sXZ^2 - AX - 2sAZ - D = 0\\
            &Z = 2 g_{2}(\xi, \eta)\\
            &\xi^2 - \omega\eta^2 = 1 
        \end{align*}
        where $g_2(\xi, \eta)$ is as in the elliptic case of the endgame. As this system is equivalent to the one appearing in the hyperbolic case of the proof of this lemma over an algebraically closed field, this curve is also geometrically irreducible, and the Weil bound produces points of the desired form just as in the hyperbolic case for $p\equiv 3 \pmod{4}$.
    \end{proof}
    With the lemmas proved, the proof of Theorem \ref{comparisonOfGroupsThm} is now complete.
\end{proof}

Theorem \ref{comparisonOfGroupsThm} assumes that $\Gamma'$ is transitive on all of $\mathcal{S}^*_{A,B,C,D}(p)$. For the primes at which transitivity is not known, we shall need the following version for the large connected component  $\mathcal{C}$ constructed in Sections \ref{EndgameSection} and \ref{middlegameSection}, the $\Gamma'$-orbit containing the cage $\mathcal{K}$.
 
\begin{proposition}\label{prop:ComparisonLargeComponent}
Suppose that $(A,B,C,D)$ is nondegenerate modulo $p$ and $p$ is sufficiently large (independent of $A,B,C,D$). Let $\mathcal{C} = \mathcal{C}(p)$ be the $\Gamma'$-orbit of the cage $\mathcal{K}(p)$. Then $\mathcal{C}$ is a single $\Gamma$-orbit.
\end{proposition}
\begin{proof}
We run the proof of Theorem \ref{comparisonOfGroupsThm} with $\mathcal{C}$ in place of $\mathcal{S}^*_{A,B,C,D}(p)$. Note first that $\Gamma\times H\unlhd\Gamma'$ and $\mathcal{C}$ is a single $\Gamma'$-orbit, so $\Gamma'/(\Gamma\times H)$ acts transitively on the set of $(\Gamma\times H)$-orbits contained in $\mathcal{C}$, and likewise $H$ acts transitively on the $\Gamma$-orbits contained in a $(\Gamma\times H)$-orbit; the orbit-stabilizer arguments in the proofs of Lemma \ref{lemma:RemovalOfH}, Lemma \ref{lemma:OneTranspositionCase}, and Lemma \ref{lemma:TwoAndSixOrbitsImpossible}, and the final three-orbit argument, use transitivity only through these facts. The only other input in that proof is the existence of points fixed by, or stabilizing the orbit under, a given element of $H$ or a given transposition (Lemmas \ref{lemma:TranspositionFixedPointsExist}, \ref{lemma:oddOrderSymmetryExists}, and the point with $4\mid\operatorname{ord}(\chi)$ in the proof of Lemma \ref{lemma:RemovalOfH}). It therefore suffices to know that such points can be found \textit{inside} $\mathcal{C}$. But in each case the Weil bound produces $\gg p$ such points in $\mathcal{S}_{A,B,C,D}(\mathbb{F}_p)$, and only $O(p^{4\epsilon})$ points of $\mathcal{S}_{A,B,C,D}(\mathbb{F}_p)$ have all three coordinates of order less than $p^\epsilon$ (there are $O(p^{2\epsilon})$ elements of $\mathbb{F}_p$ of order less than $p^{\epsilon}$, and $z$ is determined by $x,y$ up to two choices). After discarding these, the points of $E(p)$, and the at most $12$ double fixed points, we are left with points of order at least $p^\epsilon$ which are not double fixed, and these lie in $\mathcal{C}$ by the middlegame. (For the point with $4\mid\operatorname{ord}(\chi)$ we may simply choose $\chi$ of order at least $p^\epsilon$.) The conclusion of the proof of Theorem \ref{comparisonOfGroupsThm} then applies verbatim, showing that $\mathcal{C}$ is a single $\Gamma$-orbit.
\end{proof}

\section{Middlegame}\label{middlegameSection}
In the endgame, it is shown that the cage is connected, that almost all points of order at least $p^{\frac{1}{2} + \delta}$ connect to the cage in a single move, and that \textit{all} points of order at least $p^\frac{3}{4}$ connect to the cage in at most two moves. In this section we allow arbitrarily many moves, and prove the following:
\begin{proposition}\label{prop:middlegame}
    Each point of order at least $p^\epsilon$ in $\mathcal{S}^*_{A, B, C, D}(p)$ which is not a double fixed point (i.e. a fixed point of both $V_i$ and $V_j$ for $i\neq j$) connects to the cage (for $p$ sufficiently large depending on $\epsilon$).
\end{proposition}
\begin{remark}[Counting double fixed points]
    In the nondegenerate case, there are at most 12 double fixed points: Propositions \ref{DoubleFixedPoints} and \ref{Parabolics} establish that there are at most four points fixed by both $V_3$ and $V_2,$ and the same holds by symmetry for bounding the number of points fixed by $V_3$ and $V_1$ or by $V_2$ and $V_1$.
\end{remark} 

The key to our argument is the following result of Corvaja and Zannier (\cite{CorvajaZannier}, Corollary 2):
\begin{proposition}\label{CorvajaZannierBound}
Let $X\subset \mathbb{G}_m^2$ be a geometrically irreducible plane curve of Euler characteristic $\chi$, not the translate of a subtorus. Suppose it is defined by an equation $f(x, y) = 0$ of bidegree $(d_1, d_2)$. Then
$$|X \cap (\mu_{m_1} \times \mu_{m_2} )| \leq \max\left(3\sqrt[3]{2}(m_1m_2d_1d_2\chi)^\frac{1}{3}, 12\frac{m_1m_2d_1d_2}{p}\right).$$
\end{proposition}
\begin{proof}[Proof of Proposition \ref{prop:middlegame}]
Suppose without loss of generality that $(x, y, z)\in \mathcal{S}_{A, B, C, D}(\mathbb{F}_p)$ with $x = \chi + \chi^{-1}$ with $\chi$ of order dividing $p-1$ or $p+1$ and at least $p^\epsilon$. Let $H_1 = \langle\chi^2 \rangle.$ As in the endgame, for $H_2$ a subgroup of $\mathbb{F}_{p^2}^\times$ of order dividing either $p-1$ or $p+1$, let $f_{H_1}(H_2)$ denote the number of solutions to
\begin{equation}\label{yCoordEq}\alpha_1t + \alpha_2t^{-1} + \alpha_3 = s + s^{-1}\end{equation}
with $t\in H_1$ and $s\in H_2$, where
\[\alpha_1 = (\chi^{-1} - \chi)^{-1}(\chi^{-1}y - z + (\chi^{-1} - \chi)^{-1} (B + C\chi^{-1}))\]
\[\alpha_2 = (\chi^{-1} - \chi)^{-1}(-\chi y + z + (\chi^{-1} - \chi)^{-1}(B + C\chi))\]
\[\alpha_3 = (\chi^{-1} - \chi)^{-2}(-2B - Cx)\]
If we can show that
$$\sum\limits_{|H_2|\leq |H_1|}f_{H_1}(H_2) \leq |H_1| - 8,$$
then we can show that there is at least one solution to 
$$\alpha_1t + \alpha_2t^{-1} + \alpha_3 = s + s^{-1}$$
with $s$ having order strictly greater than that of $\chi$ such that $f_2(s + s^{-1})\neq 0,$ which will connect our original point to a new point of strictly greater order that is not fixed by both $V_1$ and $V_3$, from which we can iterate to eventually reach a point of large enough order to apply the endgame. Rewriting (\ref{yCoordEq}), as long as
\begin{equation}\label{middlegameEq}
    \alpha_1t^2s + \alpha_3ts + \alpha_2s - ts^2 - t
\end{equation}
is not divisible by a polynomial cutting out a translate of a subtorus, then we apply Proposition \ref{CorvajaZannierBound} to the at most $3$ irreducible components of the curve (\ref{middlegameEq}) for $|H_2|\leq |H_1| \leq p^\frac{3}{4}$. The constraints on the sizes of $H_1$ and $H_2$ allow us to drop the maximum, while the constant comes from using the bound on the number of irreducible components together with the bounds on the genus of each component found in the endgame. We therefore obtain
$$f_{H_1}(H_2) \leq 20(|H_1||H_2|)^\frac{1}{3}.$$

But then
$$\sum\limits_{\substack{|H_2|\leq |H_1|}}f_{H_1}(H_2)\leq 20|H_1|^\frac{1}{3}\sum\limits_{\substack{d|(p-1)\text{ or } (p+1)\\d\leq |H_1|}}d^\frac{1}{3} \leq 20|H_1|^\frac{2}{3}d(p^2-1)\leq C_\epsilon|H_1|^\frac{2}{3}p^\eta,$$
where $\eta$ is chosen sufficiently small relative to $\epsilon$, giving the desired bound. We therefore merely have to show that no irreducible factor of (\ref{middlegameEq}) is a translate of a subtorus. Now, the equation of a translate of a subtorus of $\mathbb{G}_m^2$ takes the shape $t^ds^e - \zeta = 0$ or $t^d - \zeta s^e = 0$ where $\zeta$ is a root of unity. We may safely restrict to eliminating irreducible subtori, at which point constraints on degree imply that the only subtorus translates we have to check are those cut out by $t - \zeta,$ $s - \zeta$, $st-\zeta,$ $t-\zeta s$, $t^2 - \zeta s$, and $t-\zeta s^2.$ Of these, we can eliminate all but $s - \zeta$, $st-\zeta$, and $t-\zeta s$ fairly straightforwardly: for a subtorus translate cut out by $f(s, t)$, just describe the general polynomial $g(s, t)$ such that the product has total degree, $s$-degree, and $t$-degree no larger than (\ref{middlegameEq}), and compare terms. 

For $s - \zeta$, we evaluate (\ref{middlegameEq}) at $s = \zeta$ as
$$\alpha_1\zeta t^2 + (\alpha_3\zeta - \zeta^2 - 1)t + \alpha_2\zeta,$$
which is identically zero only if $\alpha_1 = \alpha_2 = 0$, and in particular $(x, y, z)$ is a double fixed point.

The cases of $st-\zeta$, and $t-\zeta s$ are actually the same: these factors only occur if $\alpha_1\alpha_2 = 1$ and $\alpha_3 = 0,$ in which case we have the factorization
$$\zeta^{-1}t^2s + \zeta s - ts^2 - t = (t - \zeta s)(\zeta^{-1}ts-1).$$
The requirement that $\alpha_3 = 0$ implies that either $B = C = 0$ or  
$$x = \frac{-2B}{C}.$$ 

If $B \neq \pm C,$ then instead of searching through the $y$-coordinates of the orbit to find a point of higher order, we can search through the $z$-coordinates; in doing so we get a curve of a similar form but where $\alpha_3 = 0$ again requires either $B = C = 0$ (forbidden by assumption $B \neq \pm C$) or $x = \frac{-2C}{B}.$ But

$$\frac{-2C}{B} = x = \frac{-2B}{C}$$

contradicts $B \neq \pm C$. 

If $B = \pm C \neq 0$, then $x = \pm 2,$ so that we started from a parabolic point which was connected to the cage already in the endgame.

Finally, if $B = C = 0,$ then using $\alpha_1\alpha_2 = \kappa_1(x)$ the condition $\alpha_1\alpha_2 = 1$ becomes the condition
$$x^2 - 4 = \left(x^2 - Ax - D\right),$$
or
$$Ax = -D+4.$$

If $A = 0,$ then our starting parameters are $(A, B, C, D) = (0, 0, 0, 4),$ which is the Cayley cubic (generally forbidden). However, if $A \neq 0$, then if we have 
$$x = \frac{-D + 4}{A},$$
the above argument fails. We circumvent this in two steps: if the point we started with does not have 
$$x = \frac{-D + 4}{A},$$
we simply demand that at every step, the point of higher order we move to does not end up having
$$x = \frac{-D + 4}{A};$$
this can be done because
$$\sum\limits_{\substack{|H_2|\leq |H_1|}}f_{H_1}(H_2)\leq C_\epsilon|H_1|^\frac{2}{3}p^\epsilon$$
is significantly smaller than $|H_1|$ and we are only forbidding one value of $x$. Second, if we start at a point with
$$x = \frac{-D + 4}{A},$$
we first move to a point of slightly smaller order but where
$$x \neq \frac{-D + 4}{A},$$
similarly to how we handled parabolic and ``forbidden'' points in the endgame (Proposition \ref{endgame}). This finishes the proof.
\end{proof}
\begin{proof}[Proof sketch of Theorem \ref{th:EndgamePlusMiddlegame}]
Fix $\epsilon>0$, let $p$ be sufficiently large, and suppose $(A,B,C,D)$ is nondegenerate modulo $p$. No density-one hypothesis on $p$ is needed here, as the opening is not used. The proof assembles the results of Sections \ref{EndgameSection}, \ref{Comparison}, and \ref{middlegameSection}:
\begin{enumerate}
    \item \textit{The cage is connected.} By Proposition \ref{cageConnectivity}, the cage $\mathcal{K}_{A,B,C,D}(p)$ lies in a single orbit of $\Gamma'$ (the extra symmetries of $\Gamma'$ being used only when $\Gamma'\neq\Gamma$). Let $\mathcal{C}$ denote this $\Gamma'$-orbit.
    \item \textit{The endgame.} By Proposition \ref{endgame}, all but boundedly many exceptional values of the relevant coordinate aside, every point of order at least $p^{\frac{1}{2}+\delta}$ is joined to the cage, and every point of order at least $p^{\frac{3}{4}}$ is joined to a point of order at least $p^{\frac{5}{8}}$ which is not exceptional.
    \item \textit{The middlegame.} By Section \ref{middlegameSection}, every point of order at least $p^\epsilon$ in $\mathcal{S}^*_{A,B,C,D}(p)$ which is not a double fixed point is joined to the cage. Hence $\mathcal{C}$ contains all such points.
    \item \textit{From $\Gamma'$ to $\Gamma$.} By Proposition \ref{prop:ComparisonLargeComponent}, $\mathcal{C}$ is a single $\Gamma$-orbit; this is the large $\Gamma$-orbit identified in the theorem. The point of this step is that Theorem \ref{comparisonOfGroupsThm} cannot be applied directly, since we do not know that $\Gamma'$ acts transitively on the whole of $\mathcal{S}^*_{A,B,C,D}(p)$; but the special points used in its proof are plentiful, and almost all of them lie in the large component.
    \item \textit{Counting.} By steps (3) and (4), $\mathcal{S}_{A,B,C,D}(\mathbb{F}_p)\setminus\mathcal{C}_{A,B,C,D}(p)$ is contained in the union of $E(p)$, the set of points all of whose coordinates have order less than $p^\epsilon$, and the set of double fixed points. The first has bounded cardinality by Theorem \ref{FiniteOrbitClassification}; the second has cardinality $O(p^{4\epsilon})$, as $\mathbb{F}_p$ has $O(p^{2\epsilon})$ elements of order less than $p^\epsilon$ and $z$ is determined by $(x,y)$ up to two choices; the third has at most $12$ elements by the Remark following the statement of the middlegame proposition. Replacing $\epsilon$ by $\epsilon/4$ gives the theorem. \qedhere
\end{enumerate}
\end{proof}

\section{Opening}\label{openingSection}
In order to prove our main theorem, we need to connect an arbitrary point $(x_0, y_0, z_0)\in\mathcal{S}_{A, B, C, D}^*(p)$ possibly of (\textit{a priori}) extremely small order to a point of order at least $p^\epsilon$ for some $\epsilon>0$ which is not a double fixed point. We do this in two steps, following Sections 5 and 6 of \cite{BGS}: first, we shall connect $(x_0, y_0, z_0)$ to a point whose order grows (albeit very slowly) with $p$; then, for density one of $p$, we shall connect this second point to a third of order $\gg p^\epsilon$. 

To accomplish our first step, we lift to characteristic zero. There are in fact two ways to do this. The first approach is to ``reconstruct'' a given mod $p$ orbit for particular parameters $(A, B, C, D)$ within a sufficiently large cyclotomic field, and estimate norms in order to derive a lower bound on the size of the orbit. This approach yields a clean lower bound on the size of orbits in $\mathcal{S}_{A, B, C, D}^*(p)$, which unfortunately depends on the parameters. The second approach uses a ``spreading out'' type argument to show that orbits in $\mathcal{S}_{A, B, C, D}^*(p)$ tend to infinity uniformly in the parameters as $p$ grows, but gives no such explicit lower bound. 

\subsection{An Explicit Lower Bound on Orbit Size}\label{subsect:AlgNTBound}
We begin by stating the result of the first approach.

\begin{proposition}\label{AlgNTBound}
    Fix integers $(A, B, C, D)$ for which $(A, B, C, D)\not\equiv (0, 0, 0, 4)\pmod{p}$, and let $(x_0, y_0, z_0)\in\mathcal{S}_{A, B, C, D}^*(p).$ Then $(x_0, y_0, z_0)$ connects to a point which is not a double fixed point and which has order at least 
    $$(\log_{(20 + 2|A| + 2|B| + 2|C| + |D|)}p)^\frac{1}{3}.$$ In particular, any connected component $F$ of $\mathcal{S}_{A, B, C, D}^*(p)$ satisfies
    $$|F|\geq \frac{1}{2}(\log_{(20 + 2|A| + 2|B| + 2|C| + |D|)}p)^\frac{1}{3}.$$
\end{proposition}
\begin{proof}
Fix a primitive $(p^2-1)^\text{st}$ root of unity $\chi\in\mathbb{F}_{p^2}^*$. Then, for each point $P=(x,y, z)\in\mathcal{S}_{A, B, C, D}^*(p),$ we may write $x = \chi^{a_1(P)} + \chi^{-a_1(P)}$, $y = \chi^{a_2(P)} + \chi^{-a_2(P)}$, and $z = \chi^{a_3(P)} + \chi^{-a_3(P)}$, with each $a_i(P)$ divisible by one of $p\pm 1$ depending on whether the corresponding coordinate is hyperbolic or elliptic. Fix a primitive $(p^2-1)^\text{st}$ root of unity $\zeta\in\mathbb{C},$ and, for each $P=(x,y, z)\in\mathcal{S}_{A, B, C, D}^*(p),$ define $P'\in\mathbb{C}^3$ by $P' = (x', y', z')$ where
\[\begin{split}
&x' = \zeta^{a_1(P)} + \zeta^{-a_1(P)},\\
&y' = \zeta^{a_2(P)} + \zeta^{-a_2(P)},\text{ and}\\
&z' = \zeta^{a_3(P)} + \zeta^{-a_3(P)}.
\end{split}\]
From a given starting point $Q\in\mathcal{S}_{A, B, C, D}^*(p),$ we shall seek to navigate to a point $P$ in its orbit which is not a double fixed point and such that some polynomial in the coordinates $(x', y', z')$ of the corresponding $P'\in\mathbb{C}^3$ does not vanish. The first condition is necessary in order to translate our lower bound on the order of $P$ to a lower bound on the size of its orbit. The second condition will allow us to obtain a lower bound on the order of $P$ as in \cite{BGS}. There are several cases:
\begin{enumerate}
    \item\label{case:OpeningGeneric} Some $P'$ whose corresponding $P$ is not a double fixed point and lies in the orbit of $Q$ does not lie on $\mathcal{S}_{A, B, C, D}(\mathbb{C}).$
    \item\label{case:OpeningSelfLoop} Each $P'$ whose corresponding $P$ is not a double fixed point and lies in the orbit of $Q$ lies on $\mathcal{S}_{A, B, C, D}(\mathbb{C}),$ but either some double fixed point $R$ in the orbit of $Q$ does not lie on $\mathcal{S}_{A, B, C, D}(\mathbb{C}),$ or the identities implied by some edge in the graph corresponding to the orbit of $Q$ (over $\mathbb{F}_p$) do not hold in $\mathbb{C}$. 
    \item \label{case:OpeningFiniteOrbits} Each $P'$ whose corresponding $P$ lies in the orbit of $Q$ lies on $\mathcal{S}_{A, B, C, D}(\mathbb{C}),$ and the identities implied by each edge in the graph corresponding to the orbit of $Q$ (over $\mathbb{F}_p$) hold in $\mathbb{C}$. 
\end{enumerate}

We dispense with case (\ref{case:OpeningFiniteOrbits}) first, as it is the simplest. In this case, we have completely lifted the orbit of $Q$ to an orbit over $\mathbb{C}$. This orbit is necessarily finite, so that the orbit of $Q$ is an element of $\mathcal{E}_{A, B, C, D}(p)$ (see Notation \ref{notation:FiniteOrbits}), and therefore the point $Q$ must have been removed in the passage from $\mathcal{S}_{A, B, C, D}(\mathbb{F}_p)$ to $\mathcal{S}_{A, B, C, D}^*(p).$

Case (\ref{case:OpeningGeneric}) is the case most analogous to the arguments in \cite{BGS}. In this case, we simply choose $P$ to be our point of interest in the orbit of $Q$, and take $\eta$ to be the value of the polynomial  
\[\begin{split}\eta = (\zeta^{a_1(P)} + \zeta^{-a_1(P)})^2 &+ (\zeta^{a_2(P)} + \zeta^{-a_2(P)})^2 + (\zeta^{a_3(P)} + \zeta^{-a_3(P)})^2\\
&-  (\zeta^{a_1(P)} + \zeta^{-a_1(P)})(\zeta^{a_2(P)} + \zeta^{-a_2(P)})(\zeta^{a_3(P)} + \zeta^{-a_3(P)})\\
&- A (\zeta^{a_1(P)} + \zeta^{-a_1(P)}) - B(\zeta^{a_2(P)} + \zeta^{-a_2(P)}) - C(\zeta^{a_3(P)} + \zeta^{-a_3(P)})- D.\end{split}\]

Case (\ref{case:OpeningSelfLoop}) is a little bit trickier. For any edge between \textit{distinct} vertices in $\mathcal{S}_{A, B, C, D}^*(p),$ this cannot happen---the vertices incident with such an edge in $\mathcal{S}_{A, B, C, D}^*(p)$ necessarily share two coordinates by the very definition of the Vieta involutions which give rise to the edges, and thus the lifts of the two vertices will also share two coordinates; if both of the corresponding points lie in $\mathcal{S}_{A, B, C, D}(\mathbb{C})$, they therefore must be connected by an edge by the nature of the Vieta involutions. However, if a vertex $P\in\mathcal{S}_{A, B, C, D}^*(p)$ has a self-loop, the implied relation between the coordinates of $P$ may no longer apply to $P'$. Thus, we must split this case further:

\begin{enumerate}
    \item \label{case:OpeningOneFixedPoint} Some point $P$ is fixed by exactly one of the Vieta involutions, but its lift $P'$ is not fixed by any Vieta involution, or
    \item \label{case:OpeningDoubleFixedPoint} Some point $P$ which is a double fixed point either has a lift $P'\not\in\mathcal{S}_{A, B, C, D}(\mathbb{C})$ or has a lift $P'\in\mathcal{S}_{A, B, C, D}(\mathbb{C})$ which is no longer a double fixed point.
\end{enumerate}

In subcase (\ref{case:OpeningOneFixedPoint}), the exact form of the polynomial we take depends on which Vieta involution fixes $P$. We assume $P$ is fixed by $V_1$, the other cases being similar. In this case we may take the polynomial

$\eta = A +(\zeta^{a_2(P)} + \zeta^{-a_2(P)})(\zeta^{a_3(P)} + \zeta^{-a_3(P)}) - 2(\zeta^{a_1(P)} + \zeta^{-a_1(P)}).$

In subcase (\ref{case:OpeningDoubleFixedPoint}), we have to work slightly harder. If $R$ is a double fixed point whose lift $R'$ does not lie on $\mathcal{S}_{A, B, C, D}(\mathbb{C}),$ or whose lift $R'$ lies on $\mathcal{S}_{A, B, C, D}(\mathbb{C})$ but is no longer a double fixed point, we proceed indirectly. Let $P$ denote the unique neighbor of $R$ in the graph of $\mathcal{S}_{A, B, C, D}(\mathbb{F}_p)$, and let $P'$ be the corresponding lift of $P$. We then search for a polynomial in the coordinates of $P'$, rather than in the coordinates of $R'$, as a lower bound on the order of $R$ does not translate to a usable lower bound on the size of the orbit containing $R$, while a lower bound on the order of $P$ does. Let $V_i$ be the Vieta involution changing $R$ into $P$. In the case where $R'$ does not lie on $\mathcal{S}_{A, B, C, D}$, we may assume that $V_iP'$ is not a double fixed point, as this reduces to case (\ref{case:OpeningFiniteOrbits}) of our original case breakdown. Using these facts, we will find a Vieta involution $V_j$ such that 
$$V_j\circ V_i P' \neq V_iP',$$
from which we can extract a nonvanishing polynomial and corresponding algebraic integer $\eta$. In the case where $R'\in\mathcal{S}_{A, B, C, D}(\mathbb{C})$, we may take $V_j$ to be the Vieta involution which fixes $R$ but does not fix $R'$; in the case where $R'\not\in\mathcal{S}_{A, B, C, D}(\mathbb{C})$ we may take $V_j$ to be any Vieta involution other than $V_i$ which does not fix $V_iP'$, at least one of which exists from our above discussion. Now, from 
$$V_j\circ V_i P' \neq V_iP',$$
we can extract our choice of $\eta.$ We will spell out the case where $i = 1$ and $j=2$; other cases are (essentially) symmetric. Comparing the $y$-coordinates of $V_j\circ V_i P'$ and $V_iP'$, we find that

$$B + AZ + YZ^2 - XZ - 2Y\neq 0,$$

from which we set

\[\begin{split}\eta = B + &A(\zeta^{a_3(P)} + \zeta^{-a_3(P)}) + (\zeta^{a_2(P)} + \zeta^{-a_2(P)})(\zeta^{a_3(P)} + \zeta^{-a_3(P)})^2\\
&- (\zeta^{a_1(P)} + \zeta^{-a_1(P)})(\zeta^{a_3(P)} + \zeta^{-a_3(P)}) - 2(\zeta^{a_2(P)} + \zeta^{-a_2(P)}).\end{split}\]

Thus, for each case, we have produced the desired polynomial. We remark that in all cases, we have that 
$$|\eta|\leq 20 + 2|A| + 2|B| + 2|C| + |D|,$$
and the same is true of any conjugate of $\eta$. We now use these facts to show that $P$ has order at least

$$\frac{1}{2}(\log_{(20 + 2|A| + 2|B| + 2|C| + |D|)}p)^\frac{1}{3},$$

following \cite{BGS}. For the point $P'$ which we have constructed in the course of the case breakdown, let $\eta$ be the corresponding polynomial, $\chi_i = \chi^{a_i(P)},$ and $l_i$ denote the order of $\zeta^{a_i(P)}$ (and thus also $\chi_i$). Let $n = \text{lcm}(l_1, l_2, l_3)$ and $L_n = \mathbb{Q}(\zeta_n)$ with $\zeta_n$ a primitive $n^\text{th}$ root of unity; note that $\eta\in L_n$ with $\eta\neq 0$ by the above. Using the bound 
$$|\eta|\leq 20 + 2|A| + 2|B| + 2|C| + |D|,$$
known also for any conjugate of $\eta$, we have that
$$|\operatorname{Norm}(\eta)| \leq (20 + 2|A| + 2|B| + 2|C| + |D|)^{\phi(n)} \leq (20 + 2|A| + 2|B| + 2|C| + |D|)^n.$$
For any prime $\mathfrak{P}$ of $\mathcal{O}_{L_n},$ if $\mathfrak{P}\mid(\eta),$ then
$$\operatorname{Norm}(\mathfrak{P}) \leq |\operatorname{Norm}(\eta)| \leq (20 + 2|A| + 2|B| + 2|C| + |D|)^n.$$
In particular, if 
$$\log_{(20 + 2|A| + 2|B| + 2|C| + |D|)}\operatorname{Norm}(\mathfrak{P}) > n,$$
then $\eta \not \equiv 0\pmod{\mathfrak{P}}.$
If each $\chi_i\in\mathbb{F}_p$, then $n\mid p-1,$ so that $p$ splits completely in $L_n$. We may thus fix a prime $\mathfrak{P}$ above $p$ such that
$$\mathcal{O}_{L_n}/\mathfrak{P} \cong \mathbb{F}_p, \, \operatorname{Norm}(\mathfrak{P}) = p,$$
and $\pi(\zeta^{a_i(P)}) = \chi_i.$ As we assumed that $P=(x, y, z)$ is a solution to 
$$X^2 + Y^2 + Z^2 = XYZ + AX + BY+CZ+D$$
over $\mathbb{F}_p$, we must have that $\eta \equiv 0\pmod{\mathfrak{P}}$, whence
$$\log_{(20 + 2|A| + 2|B| + 2|C| + |D|)}p \leq \text{lcm}(l_1, l_2, l_3).$$
If on the other hand we have that some $\chi_i\in\mathbb{F}_{p^2}\setminus\mathbb{F}_p,$ then $n\mid p^2-1$ but $n\nmid p-1,$ so that there is a prime $\mathfrak{P}$ of $\mathcal{O}_{L_n}$ lying above $p$ such that
$$\mathcal{O}_{L_n}/\mathfrak{P} \cong \mathbb{F}_{p^2}, \, \operatorname{Norm}(\mathfrak{P}) = p^2,$$
and $\pi(\zeta^{a_i(P)}) = \chi_i.$ Repeating the previous argument, we now find that
$$2\log_{(20 + 2|A| + 2|B| + 2|C| + |D|)}p \leq \text{lcm}(l_1, l_2, l_3).$$
In either case, we may conclude that
$$\max(l_1, l_2, l_3) \geq (\log_{(20 + 2|A| + 2|B| + 2|C| + |D|)}p)^\frac{1}{3}.$$
Since we have constructed $P$ to lie in the orbit of $Q = (x_0, y_0, z_0)$ and to not be a double fixed point, we obtain that the orbit of $P$, and thus the orbit of $Q$, has size at least
$$\frac{1}{2}(\log_{(20 + 2|A| + 2|B| + 2|C| + |D|)}p)^\frac{1}{3},$$
given that the order of $D_i$ is at least half of the order of the associated coordinate, as was to be shown.
\end{proof}

\subsection{An Implicit Uniform Lower Bound on Orbit Sizes}\label{subsect:SpreadingOut}
We now describe the second, more uniform approach. Recall that by $G = G_{A, B, C, D}(p)$ we mean the three regular, three-edge-colored graph whose vertices are elements of $\mathcal{S}_{A, B, C, D}^*(p)$ and in which two vertices are connected by an edge of color $i$ if each is taken to the other by $V_i$; see the beginning of Section \ref{ConicsSection}. Recall also that $\mathcal{S}^*_{A, B, C, D}(p)$ was formed by removing from $\mathcal{S}_{A, B, C, D}(\mathbb{F}_p)$ the projections of the finite orbits on $\mathcal{S}_{A, B, C, D}(\mathbb{C}),$ which finite orbits were classified as either type I, II, III, or IV orbits (having 1, 2, 3, or 4 points respectively) which occur for infinite families of parameters, exceptional orbits occurring for only one choice of parameters (of which there are 45, up to equivalence), or orbits on the Cayley cubic $\mathcal{S}_{0, 0, 0, 4}$; see Section \ref{FiniteOrbitsSection} for details, and especially Theorem \ref{FiniteOrbitClassification}. We need one more definition before we can state the main proposition.
\begin{definition}\label{def:HOccursAtp}
    Let $H$ be a fixed finite, connected, $3$-regular, $3$-edge colored graph (not necessarily loopless). We say that $H$ \textit{occurs at $p$} if there is some choice of parameters $(A, B, C, D)\not\equiv (0, 0, 0, 4)\pmod{p}$ such that $H$ is a connected component of $G_{A, B, C, D}(p).$
\end{definition}
For a $\Gamma$-orbit $\mathcal{O}\subseteq\mathcal{S}_{A,B,C,D}(k)$ ($k$ a field), the \textit{graph of $\mathcal{O}$} has vertex set $\mathcal{O}$ and, for each $i$ and each $P\in\mathcal{O}$, one edge of color $i$ joining $P$ to $V_iP$ (a loop if $V_iP = P$). With this convention, the connected components of $G_{A,B,C,D}(p)$ are exactly the graphs of the $\Gamma$-orbits in $\mathcal{S}^*_{A,B,C,D}(p)$.
\begin{definition}\label{def:QuotientGraph}
Let $H$ be as in Definition \ref{def:HOccursAtp}. A \textit{quotient} of $H$ is a graph $H/{\sim}$, where $\sim$ is an equivalence relation on the vertex set of $H$ with the property that if $v\sim w$, then for each color $i$ the color-$i$ neighbors of $v$ and $w$ are also equivalent. The vertices of $H/{\sim}$ are the equivalence classes, and the class of $v$ is joined by an edge of color $i$ to the class of $v'$ whenever $v$ and $v'$ are joined by an edge of color $i$ in $H$. This is again a connected, $3$-regular, $3$-edge-colored graph, possibly with loops. We call a graph \textit{admissible} if it is isomorphic to the graph of a finite $\Gamma$-orbit on $\mathcal{S}_{A,B,C,D}(\mathbb{C})$ for some $(A,B,C,D)\in\mathbb{C}^4$, including the Cayley parameters; i.e., if it is isomorphic to a Type I, II, III, or IV orbit, one of the 45 exceptional orbits identified by Lisovyy and Tykhyy, or one of the infinitely many finite orbits on the Cayley cubic (Theorem \ref{FiniteOrbitClassification}).
\end{definition}

\begin{definition}\label{def:Realization}
Let $H$ be as in Definition \ref{def:HOccursAtp}, let $k$ be a field, and let $A,B,C,D\in k$. A \textit{realization} of $H$ over $k$ with parameters $(A,B,C,D)$ is a map $\varphi:V(H)\to\mathcal{S}_{A,B,C,D}(k)$ such that $\varphi(v')=V_i\varphi(v)$ whenever $v$ and $v'$ are joined by an edge of color $i$ (when the edge is a loop, this says that $\varphi(v)$ is fixed by $V_i$).
\end{definition}

\begin{lemma}\label{lemma:SolutionsAreQuotients}
Let $H$ be a finite, connected, $3$-regular, $3$-edge-colored graph, and let $\varphi$ be a realization of $H$ over a field $k$. Then:
\begin{enumerate}
    \item the image $\varphi(V(H))$ is a $\Gamma$-orbit $\mathcal{O}_\varphi$, the relation $v\sim_\varphi w\iff\varphi(v)=\varphi(w)$ is compatible with the edges in the sense of Definition \ref{def:QuotientGraph}, and $\varphi$ induces an isomorphism from $H/{\sim_\varphi}$ onto the graph of $\mathcal{O}_\varphi$; in particular, $H$ is the graph of $\mathcal{O}_\varphi$ if and only if $\varphi$ is injective;
    \item conversely, if $H'=H/{\sim}$ is a quotient of $H$ and $\varphi'$ is a realization of $H'$, then $\varphi'\circ q$ is a realization of $H$, where $q:V(H)\to V(H')$ is the quotient map.
\end{enumerate}
\end{lemma}
\begin{proof}
Since $H$ is connected and each vertex has a neighbor of each color, the image of $\varphi$ is nonempty, closed under $V_1,V_2,V_3$, and connected under them, so it is a single $\Gamma$-orbit. If $\varphi(v)=\varphi(w)$, then for the color-$i$ neighbors $v',w'$ of $v,w$ we have $\varphi(v')=V_i\varphi(v)=V_i\varphi(w)=\varphi(w')$, so $\sim_\varphi$ is compatible with the edges. The induced map $H/{\sim_\varphi}\to$ (graph of $\mathcal{O}_\varphi$) is a bijection on vertices which preserves colored edges, and both graphs have exactly one edge of each color at each vertex, so it is an isomorphism. Part (2) is immediate from the definitions.
\end{proof}

\begin{proposition}\label{prop:SpreadingOut}
    Let $H$ be a fixed finite, connected, $3$-regular, $3$-edge colored graph (potentially with loops). Then $H$ occurs at $p$ for only finitely many primes $p$. 
\end{proposition}
\begin{proof}
    The proof is by a ``spreading out'' argument. Suppose $H$ has $n$ vertices and $m$ edges. For any commutative ring $R$, the condition for some $\mathcal{S}_{A, B, C, D}(R)$ to contain an orbit whose graph is $H$ may be described as a large system of polynomials, as follows. For each vertex $v_i$, we have an equation of the form
    $$F_i(X_i, Y_i, Z_i) \coloneq X_i^2 + Y_i^2 + Z_i^2 - X_iY_iZ_i - AX_i - BY_i - CZ_i - D = 0,$$
    for each edge with color 1 (say from vertex $v_i$ to $v_j$) we have three equations, of the form
    \begin{align*}&G_{1, 1, i, j} \coloneq A + Y_iZ_i - X_i - X_j= 0,\\
    &G_{1, 2, i, j} \coloneq Y_i - Y_j = 0,\text { and}\\
    &G_{1, 3, i, j} \coloneq Z_i - Z_j = 0
    \end{align*}
    for each edge with color 2 (say from vertex $v_i$ to $v_j$) we have three equations, of the form
    \begin{align*}&G_{2, 1, i, j} \coloneq X_i - X_j= 0,\\
    &G_{2, 2, i, j}\coloneq B + X_iZ_i - Y_i - Y_j = 0,\text { and}\\
    &G_{2, 3, i, j} \coloneq Z_i - Z_j = 0
    \end{align*}
    
    and for each edge with color 3 (say from vertex $v_i$ to $v_j$) we have three equations, of the form
    \begin{align*}&G_{3, 1, i, j} \coloneq X_i - X_j= 0,\\
    &G_{3, 2, i, j}\coloneq Y_i - Y_j = 0,\text { and}\\
    &G_{3, 3, i, j} \coloneq C + X_iY_i - Z_i - Z_j = 0.
    \end{align*}

    Let $I$ be the ideal of $\mathcal{P} \coloneq \mathbb{Z}[A, B, C, D, X_1, X_2, \dots, X_n, Y_1, Y_2, \dots, Y_n, Z_1, Z_2, \dots, Z_n]$ generated by the polynomials $F_i$ associated to each vertex together with the polynomials $G_{1, i, j},$ $G_{2, i, j}$ and $G_{3, i, j}$ associated to each edge, and let $\mathcal{O} = \mathcal{P}/I$. Unwinding the equations, a solution of the system over a field $k$ (a $k$-point of $\operatorname{Spec}(\mathcal{O})$) is exactly a realization of $H$ over $k$ in the sense of Definition \ref{def:Realization}: the equation $F_i=0$ says that $(X_i,Y_i,Z_i)\in\mathcal{S}_{A,B,C,D}(k)$; for an edge of color $1$ from $v_i$ to $v_j$, the equations $G_{1,2,i,j}=G_{1,3,i,j}=0$ say that the two points have the same $y$- and $z$-coordinates, and $G_{1,1,i,j}=0$ says that $X_j=A+Y_iZ_i-X_i$, i.e.\ that $(X_j,Y_j,Z_j)=V_1(X_i,Y_i,Z_i)$; colors $2$ and $3$ are identical, and for a loop ($i=j$) the equations say that the point is fixed by $V_i$. For a solution $s$ we write $\varphi_s$ for the corresponding realization, $\varphi_s(v_i)=(X_i,Y_i,Z_i)$.

    We find that if $H$ appears as the graph of some orbit on some $\mathcal{S}_{A, B, C, D}$ over a commutative ring $R$ (for some choice of $A, B, C, D\in R$), then there is a solution to the system in the ring $R$ (equivalently, for readers who prefer the language of schemes, the affine scheme $\operatorname{Spec}(\mathcal{O})$ has an $R$-point). The converse is false, for a reason which we must keep in mind throughout the proof.
    
    \begin{remark}[Solutions of the system need not be injective]\label{rmk:SolutionsNeedNotBeInjective}
    The system only records that the vertices $v_1,\dots,v_n$ of $H$ are mapped to points of $\mathcal{S}_{A,B,C,D}$ in a way compatible with the Vieta involutions; it does not require the points $(X_i,Y_i,Z_i)$ to be distinct. If two of them coincide, then by Lemma \ref{lemma:SolutionsAreQuotients} the solution does not describe an orbit with graph $H$, but rather an orbit whose graph is a proper quotient of $H$. For example, let $H$ be the graph with two vertices joined by three edges, of colors $1$, $2$, and $3$. Any solution has $V_1P=V_2P=V_3P$, which forces $P$ to be fixed by all three Vieta involutions (cf.\ footnote \ref{NoBigons}), so every solution is a triple fixed point counted twice; $H$ is not admissible, but its quotient with one vertex and three loops is, and the system for $H$ has solutions over every field containing a Type I orbit. Two consequences are worth stating at the outset.
    \begin{enumerate}
        \item The existence of a solution to the system for $H$, over $\mathbb{C}$ or over $\mathbb{F}_p$, does \textit{not} imply that $H$ is the graph of an orbit; so the dichotomy in the proof below must be phrased in terms of all quotients of $H$, and not just in terms of $H$ itself. In particular, the system for $H$ can have solutions over $\mathbb{C}$ even though $H$ is not admissible, as long as some quotient of $H$ is.
        \item The spreading-out argument cannot single out the injective solutions, since injectivity is an open condition (some coordinate of some pair of vertices must differ) rather than one cut out by the equations of the system. What it controls is all solutions at once, injective or not. We will therefore prove the following statement, which is formally stronger than the proposition: for all but finitely many $p$, \textit{every} solution of the system over $\mathbb{F}_p$ has image contained in $E(p)$, unless $(A,B,C,D)\equiv(0,0,0,4)\pmod p$. This implies the proposition. Indeed, if $H$ occurs at $p$, then $H$ is a connected component of $G_{A,B,C,D}(p)$ for some $(A,B,C,D)\not\equiv(0,0,0,4)$, so by Lemma \ref{lemma:SolutionsAreQuotients} there is an injective solution over $\mathbb{F}_p$ whose image is an orbit in $\mathcal{S}^*_{A,B,C,D}(p)$, and this orbit is disjoint from $E(p)$.
    \end{enumerate}
    \end{remark}

    The proof is simplest in the case that no quotient of $H$ (including $H$ itself) is admissible. In this case the system has no solution over $\mathbb{C}$ (i.e.\ $\operatorname{Spec}(\mathcal{O})$ has no $\mathbb{C}$-point), and hence no $\overline{\mathbb{Q}}$-points: by Lemma \ref{lemma:SolutionsAreQuotients}, a solution over $\mathbb{C}$ would produce a $\Gamma$-orbit on some $\mathcal{S}_{A,B,C,D}(\mathbb{C})$ whose graph is a quotient of $H$, and this orbit would be finite, hence admissible. It follows from the Nullstellensatz that, for some polynomials $g_k\in\mathcal{P}\otimes\overline{\mathbb{Q}}$ and polynomials $f_k$ of either the form $F_i$, $G_{1, i, j}$, $G_{2, i, j},$ or $G_{3, i, j},$

    \begin{equation}\label{eq:Nullstellensatz1}
        f_1g_1 + \dots + f_Ng_N = 1.
    \end{equation}

    The polynomials $f_k$ have $\mathbb{Z}$-coefficients. There are only finitely many $g_k$, each of which has finitely many coefficients, and each coefficient lives in a finite extension of $\mathbb{Q}$, so there is some number field $K$ containing all coefficients of all $g_k.$ By possibly extending to a larger number field, $K$ may be assumed to be Galois over $\mathbb{Q}$. Summing over Galois-conjugates of (\ref{eq:Nullstellensatz1}), we get an equation of the form

    $$f_1g_1 + \dots + f_Ng_N = [K: \mathbb{Q}]$$
    with new polynomials $g_k$ having coefficients in $\mathbb{Q}$. Multiplying by the denominators of the $g_k$, we get an equation of the form
    $$f_1g_1 + \dots + f_Ng_N = M,$$
    for some integer $M$ where the new $g_k$ may be assumed to have integer coefficients. Reducing this mod $p$, we immediately see that the system has no solution over $\mathbb{F}_p$  for $p$ not dividing $M$, whence $H$ does not occur at $p$ for any $p$ not dividing $M.$ 

    We now turn to the remaining case, in which some quotient of $H$ is admissible. By Remark \ref{rmk:SolutionsNeedNotBeInjective}, the system now does have solutions over $\mathbb{C}$ (pull back the finite orbit), and we must show that nonetheless, away from finitely many primes, all solutions over $\mathbb{F}_p$ come from the orbits removed in passing to $\mathcal{S}^*_{A,B,C,D}(p)$, or from the Cayley parameters. Let $v_1$ be a fixed vertex of $H$, and let $\pi:\operatorname{Spec}(\mathcal{O})\to\mathbb{A}^7$ send a solution to $(A,B,C,D,X_1,Y_1,Z_1)$.
 
\textit{Step 1: a closed set containing the possible values of $\pi$ over $\mathbb{C}$.} Let $\mathcal{T}\subseteq\mathbb{A}^7$ be the union of the following closed sets:
\begin{itemize}
    \item $\mathcal{T}_{\mathrm{I}}$: the set where $A=2x-yz$, $B=2y-xz$, $C=2z-xy$, $D=2xyz-x^2-y^2-z^2$ (Proposition \ref{singletonOrbits}), i.e.\ $(x,y,z)$ is a point of a Type I orbit;
    \item $\mathcal{T}_{\mathrm{II}},\mathcal{T}_{\mathrm{III}},\mathcal{T}_{\mathrm{IV}}$: the sets where $(A,B,C,D)$ satisfy the relations defining Type II, III, IV orbits and $(x,y,z)$ is one of the explicit points of the corresponding orbit listed in Section \ref{FiniteOrbitsSection}, together with the images of these sets under the equivalences of Definition \ref{def:EquivalentParams};
    \item $\mathcal{T}_{\mathrm{ex}}$: the finite set of points $(A,B,C,D,x,y,z)\in\overline{\mathbb{Q}}^7$ for which $(x,y,z)$ lies on one of the $45$ exceptional orbits of Theorem \ref{FiniteOrbitClassification} (or an equivalent orbit); and
    \item $\mathcal{T}_{\mathrm{Cay}}$: the set where $(A,B,C,D)=(0,0,0,4)$.
\end{itemize}
Each $\mathcal{T}_\ast$ is a Zariski closed subset of $\mathbb{A}^7_{\overline{\mathbb{Q}}}$ which is stable under $\operatorname{Gal}(\overline{\mathbb{Q}}/\mathbb{Q})$ (for the exceptional orbits, this holds because the Vieta involutions are defined over $\mathbb{Z}$, so that Galois conjugation carries the $\Gamma$-orbit of a point $P$ bijectively onto the $\Gamma$-orbit of its conjugate, and hence preserves the property of having a finite orbit, as well as the property of lying on a Type I--IV orbit). As in the first case of our argument, a Galois-stable closed subset of affine space over $\overline{\mathbb{Q}}$ is defined over $\mathbb{Q}$, i.e.\ its ideal is generated by polynomials with rational coefficients, and clearing denominators we may take these generators to lie in $\mathbb{Z}[A,B,C,D,x,y,z]$. By Lemma \ref{lemma:SolutionsAreQuotients} and Theorem \ref{FiniteOrbitClassification}, if $s$ is a solution over $\mathbb{C}$ then $\varphi_s(V(H))$ is a finite orbit, or $(A,B,C,D)=(0,0,0,4)$, so that $\pi(s)\in\mathcal{T}$. Note that this uses nothing about whether $\varphi_s$ is injective.
 
\textit{Step 2: spreading out.} Let $t_1,\dots,t_r\in\mathbb{Z}[A,B,C,D,x,y,z]$ be the products $t_{\ast,j_\ast}$ obtained by choosing one integer-coefficient generator from the ideal of each piece $\mathcal{T}_\ast$ (the vanishing of all such products at a point is equivalent to lying in the union of the pieces, even after reduction modulo $p$). Each $t_j\circ\pi$ vanishes on every $\mathbb{C}$-point of $\operatorname{Spec}(\mathcal{O})$, so by the Nullstellensatz, some power $(t_j\circ\pi)^{m}$ lies in $I\otimes\overline{\mathbb{Q}}$. Summing over Galois conjugates and clearing denominators exactly as in the previous case, we obtain an integer $M_j$ and a relation $M_j(t_j\circ\pi)^m\in I$. It follows that for every prime $p\nmid M_j$ for each $j$ and every solution $s$ of the system over $\overline{\mathbb{F}}_p$, we have $t_j(\pi(s))=0$ for all $j$, i.e.\ $\pi(s)$ lies in the reduction modulo $p$ of one of the pieces $\mathcal{T}_\ast$. Only finitely many primes are excluded, and they depend only on $H$.
 
\textit{Step 3: identifying the reductions.} Let $p\nmid M$ be large, let $s$ be a solution over $\mathbb{F}_p$, and let $(A,B,C,D,x,y,z)=\pi(s)$. If $\pi(s)\in\mathcal{T}_{\mathrm{Cay}}$ mod $p$, then $(A,B,C,D)\equiv(0,0,0,4)\pmod p$, which is excluded in Definition \ref{def:HOccursAtp}. If $\pi(s)\in\mathcal{T}_{\mathrm{I}},\dots,\mathcal{T}_{\mathrm{IV}}$ mod $p$, then $(x,y,z)$ is a point of a Type I, II, III, or IV orbit over $\mathbb{F}_p$ in the sense of Section \ref{FiniteOrbitsSection}: these orbits are defined by explicit equations on $(A,B,C,D,x,y,z)$, so no further argument is needed (when a Type II or III orbit collapses modulo $p$ it becomes an orbit of a smaller type, still in $\mathcal{E}(p)$). Finally, $\mathcal{T}_{\mathrm{ex}}$ is a finite set of $\overline{\mathbb{Q}}$-points, so for $p$ large its $\overline{\mathbb{F}}_p$-points are exactly the reductions of those points, and by the Remark following Theorem \ref{FiniteOrbitClassification} the point $(x,y,z)$ then lies in an orbit in $\mathcal{E}(p)$ (possibly the reduction of an exceptional orbit with some points identified). In every case $(x,y,z)\in E(p)$ or $(A,B,C,D)\equiv(0,0,0,4)$, and as $E(p)$ is $\Gamma$-invariant, $\varphi_s(V(H))\subseteq E(p)$.
 
\textit{Step 4: conclusion.} Suppose $p\nmid M$ is large and $H$ occurs at $p$. Then $H$ is a connected component of $G_{A,B,C,D}(p)$ for some $(A,B,C,D)\not\equiv(0,0,0,4)$, so by Lemma \ref{lemma:SolutionsAreQuotients} there is a solution $s$ over $\mathbb{F}_p$ with $\varphi_s$ injective and $\varphi_s(V(H))$ an orbit in $\mathcal{S}^*_{A,B,C,D}(p)$, which is disjoint from $E(p)$. This contradicts Step 3. Hence $H$ occurs at only finitely many $p$.
 
\begin{remark}[What happens when $H$ has an admissible quotient]\label{rmk:AdmissibleQuotient}
If $H$ has a quotient $H'$ which is the graph of one of the permitted finite orbits, then the system for $H$ certainly has solutions, over $\mathbb{C}$ and over $\mathbb{F}_p$ for infinitely many $p$ (pull back the finite orbit or its reduction along $V(H)\to V(H')$). The argument above shows that, away from the finitely many primes dividing $M$, these are the only solutions: every solution over $\mathbb{F}_p$ factors through a quotient of $H$ whose image is an orbit in $\mathcal{E}(p)$. If $H$ is itself admissible, the solutions may be injective, and then they exhibit $H$ as the graph of an orbit in $\mathcal{E}(p)$; but such orbits were removed in the passage to $\mathcal{S}^*_{A,B,C,D}(p)$, so they do not witness that $H$ occurs at $p$. If $H$ is not admissible but has an admissible proper quotient, no solution over $\mathbb{F}_p$ (for $p\nmid M$) is injective, and again $H$ does not occur at $p$.
\end{remark}

\end{proof}
\subsection{Completing the Proof of the Main Theorem}\label{subsect:ProofCompletion}
We have now completed the first step of the opening, lifting to characteristic zero to achieve weak growth of orbit sizes, in two distinct ways. In order to complete the proof of our main Theorem \ref{th:mainFirstForm}, we need the following result (\cite{BGS}, Theorem 18), which is a variant of a result from \cite{CKSZ14} (see also \cite{BeukersSmyth} for the simpler setting of characteristic 0).
\begin{theorem}[\cite{BGS}, Theorem 18]\label{LowerBoundOnOrder}
    Fix $d\in\mathbb{Z}_+$ and $\delta>0$. There is an $\epsilon>0,\,\epsilon = \epsilon(d, \delta)$ such that for all primes $p\leq z$ ($z$ sufficiently large) with the exception of at most $z^\delta$ of them, the following property holds: Let $f(x, y)\in\mathbb{F}_p[x, y]$ be of degree at most $d$ and not divisible by any non-constant polynomial of the form $\rho x^\alpha y^\beta - 1$ or $\rho y^\beta - x^\alpha$ for any $\rho\in\overline{\mathbb{F}}_p$ and nonnegative integers $\alpha$ and $\beta$. Then all solutions $(x, y)\in\overline{\mathbb{F}}_p^\times\times\overline{\mathbb{F}}_p^\times$ of $f(x, y) = 0$ satisfy
    $\text{ord}(x) + \text{ord}(y) \geq p^\epsilon$
    except for at most $11d^3 + d$ of them.
\end{theorem}
\begin{notation}
    Let $\mathcal{N} = \mathcal{N}(\delta)$ denote the set of primes excluded from the conclusion of Theorem \ref{LowerBoundOnOrder} and the set of small primes excluded in Theorem \ref{th:EndgamePlusMiddlegame}.
\end{notation}
With this theorem and notation in place, we can restate our main theorem in a way which is slightly more explicit about the zero-density set of exceptional primes. 
\begin{theorem}\label{th:mainBadPrimesExplicit}
    There is a finite set of primes $\mathcal{N}'$ such that for all primes $p\not\in \mathcal{N}\cup\mathcal{N}'$, $\Gamma$ acts transitively on $\mathcal{S}^*_{A, B, C, D}(p)$ whenever $(A, B, C, D)$ is nondegenerate modulo $p$.

    For a specific choice of nondegenerate parameters $A, B, C, D$, we may instead take the exceptional set of primes to be 
    $$\mathcal{N}\cup\{p:p \leq (20 + 2|A| + 2|B| + 2|C| + |D|)^{272,097,792}\}.$$
\end{theorem}
\begin{proof} We apply Theorem \ref{LowerBoundOnOrder} to the curves
$$\alpha_1 t + \alpha_2 t^{-1} + \alpha_3 = s + s^{-1},$$
which, by avoiding double fixed points and carefully choosing the coordinate of the orbit under $D_i$ to study as in the middlegame, may be chosen so as to not be divisible by any of the polynomials forbidden in Theorem \ref{LowerBoundOnOrder}. Starting from an arbitrary point $(x_0, y_0, z_0)\in\mathcal{S}_{A, B, C, D}^*(p),$ we seek to navigate to a point $(x, y, z)$ which is not a double fixed-point and which has order greater than $324$. Once we do so, for $p$ sufficiently large, we may connect it to points of greater and greater order, and eventually to the cage. Indeed, if the order of this point $(x, y, z)$ is already $\gg p^\epsilon$, it is connected to the cage in the middlegame. Otherwise, for primes $p$ in a density one set, we connect it to such a point via Theorem \ref{LowerBoundOnOrder}. Indeed, for $p$ outside the exceptional set in Theorem \ref{LowerBoundOnOrder}, there is a point $(s, t)$ on the curve
$$\alpha_1 t + \alpha_2 t^{-1} + \alpha_3 = s + s^{-1},$$
lying in the chosen Dehn-twist orbit of $(x, y, z)$, which is neither one of the at most $11d^3 + d$ exceptional points from Theorem \ref{LowerBoundOnOrder} (here $d = 3$, so $11d^3 + d = 300$), nor one of the at most $24$ parameter pairs corresponding to double fixed points. For such $(s, t)$, the corresponding point on $\mathcal{S}_{A, B, C, D}^*(p)$ has order $\gg p^\epsilon$, and thus connects to the cage via the middlegame, endgame, and comparison theorem of Section \ref{Comparison}. (Indeed, we have that $s$ must have order $\gg p^\epsilon$, as if instead $t$ has order $\gg p^\epsilon$, then so did the original point $(x, y, z)$, which we assumed was not the case.). Thus, the proof of our main theorem reduces to showing that $(x_0, y_0, z_0)$ connects to a point $(x,y,z)$ of order greater than $324$. 

For the parameter-dependent form of the result, it suffices to navigate as in Proposition \ref{AlgNTBound} to a point $(x, y, z)\in\mathcal{S}_{A, B, C, D}^*(p)$ with the property that one of the Dehn twists starting from this point has order at least
$$\frac{1}{2}(\log_{(20 + 2|A| + 2|B| + 2|C| + |D|)}p)^\frac{1}{3}.$$
Solving
$$\frac{1}{2}(\log_{(20 + 2|A| + 2|B| + 2|C| + |D|)}p)^\frac{1}{3} > 324$$
for $p$ gives the set of additional excluded primes in this form of the theorem.

For the parameter-independent form of the result, we instead appeal to Proposition \ref{prop:SpreadingOut}. As there are at most $d$ points in a cyclic group having order $d$, and the order of a point is defined to be the supremum of the orders of its coordinates, any orbit graph on any $\mathcal{S}_{A, B, C, D}(\mathbb{F}_p)$ corresponding to an orbit containing no points of order greater than $324$ other than double fixed points has a bounded number of vertices, so that there are only finitely many such graphs. Each such graph occurs at $p$ for only finitely many $p$ by Proposition \ref{prop:SpreadingOut}. Letting $\mathcal{N}'$ be the union of the finitely many finite sets of primes excluded by applying Proposition \ref{prop:SpreadingOut} to each candidate graph, we find that away from primes in $\mathcal{N}'$, every orbit on every $\mathcal{S}^*_{A, B, C, D}(p)$ with $(A, B, C, D) \not\equiv (0, 0, 0, 4)\pmod{p}$ contains a point which is not a double fixed-point and which has order greater than $324$, as desired. 
\end{proof}
\section{Obstruction to Transitivity for Degenerate Parameters}\label{ObstructionsSection}
In this section, we prove the following theorem:
\begin{theorem}\label{QuadraticObstructions}
Suppose that $p$ is large enough, and that $(A, B, C, D)$ is a degenerate quadruple modulo $p$. Then $\mathcal{S}_{A, B, C, D}^*(p)$ contains at least two orbits. If $\pm A =\pm B = \pm C$ and $(A, B, C, D)$ is degenerate, then $\mathcal{S}_{A, B, C, D}^*(p)$ contains at least four large orbits. When we include the action of the larger group $\Gamma'$, three of the four $\Gamma$-invariant subsets of $\mathcal{S}_{A, B, C, D}^*(p)$ are joined to each other.
\end{theorem}
\begin{remark}[The Cayley Cubic]
    This theorem and its proof do hold for the Cayley parameters $(0, 0, 0, 4),$ but the resulting lower bound on the number of orbits is a radical undercount, as mentioned in Remark \ref{rmk:CayleyParams}. For other degenerate parameters, we will show in Theorem \ref{DegenerateParameterOrbits} that the lower bound in the above theorem is in fact the true count, at least for density one of $p$.
\end{remark}
\begin{proof}
We write the proof in the degenerate case where $A = B$ and $D = -\frac{A^2}{4} + 2C + 4;$ the other cases follow similarly, and in fact may be reduced to this case via equivalence of parameters as in Subsection \ref{EquivalenceParamsSubect}. More precisely, we will prove this result by constructing certain invariants, and applying these equivalence maps to the invariants in the case $A = B$ gives the invariants in the other cases, where $A = -B$, $A = \pm C$, or $B = \pm C.$ This will be done implicitly to produce the larger set of invariants when $A = B = \pm C$.

Suppose that $(x, y, z)$ and $(x, y, z')$ are a pair of solutions, connected by $V_3$. Then we have
$$zz' = x^2 + y^2 - A(x+y) + \frac{A^2}{4} - 2C - 4 = \left(x+y-\frac{A}{2}\right)^2 -2C - 2xy - 4,$$
and
$$z + z' = C + xy,$$
so that
\[\begin{split}
(z+2)(z'+2) &= zz' + 2(z+z') + 4 \\&= \left[\left(x+y-\frac{A}{2}\right)^2 -2C - 2xy - 4\right] + 2(C + xy) + 4\\&=  \left(x+y-\frac{A}{2}\right)^2,
\end{split}\]
and thus
$$\chi_p(z+2)\chi_p(z'+2) \geq 0,$$
where $\chi_p(\cdot)\coloneq \left(\frac{\cdot}{p}\right)$ denotes the quadratic character modulo $p$, extended by $\chi_p(0) = 0.$
This forms the basis for a quadratic obstruction---points $(x, y, z)$ with $\chi_p(z+2) = 1$ will not connect to points with $\chi_p(z+2) = -1$. To prove our result, we must properly define this obstruction when $z = -2$. The simplest way to do so is by way of Lemma \ref{Parabolics}: each point of $C_3(-2)\cap\mathcal S^*_{A,B,C,D}(p)$ is fixed by \(V_1\)
and \(V_2\), and hence has a unique nontrivial neighbor in
\(\mathcal S^*_{A,B,C,D}(p)\), namely its image under \(V_3\). We may assign $\pm 1$ to each point of $C_3(-2)$ according to the value of $\chi_p(z+2)$ in its unique neighbor, namely, its image under $V_3$. As each point of $C_3(-2)$ is fixed by $V_1$ and $V_2$, there is no path between points with $\chi_p(z+2) = 1$ and $\chi_p(z+2) = -1$ by passing through points of $C_3(-2)$, and this invariant indeed separates $\mathcal{S}_{A, B, C, D}^*(p)$ as desired. 

Alternatively, we can mimic the method of \cite{pDivisibilityClusterAlg}. This method is more intricate, requiring a number of complicated polynomial identities, some of which we verify by hand and others we verify with Macaulay2 \cite{Macaulay2}, but gives a slightly more uniform description of the $\Gamma$-invariant sets we identify. We begin by noting that, for 
$$z' = C + xy - z,$$
\[\begin{split}
\phi(z) &\vcentcolon= (z+2)(z+z' + 2x + 2y - A + 4)\\
&=(z+2)(xy + C + 2x + 2y - A + 4)\\
&= (xyz + Cz) - Az + 4z +  2xy + 2xz + 2yz + 4x + 4y + 2C - 2A + 8\\
&= \left(x^2 + y^2 + z^2 - Ax - Ay + \frac{A^2}{4} - 2C - 4\right)\\ &\,\,\,\,\,\,\,\,\,\,\,\,\,\,\,\,\,\,- Az + 4z + 2xy + 2xz + 2yz + 4x +4y + 2C - 2A + 8\\
&= x^2 + y^2 + z^2 + 2xy + 2xz + 2yz - Ax - Ay - Az + 4x + 4y + 4z - 2A + \frac{A^2}{4} + 4\\
&= (x + y + z - \frac{A}{2} + 2)^2;
\end{split}\]
it follows that we thus also have
$$\chi_p(z+2)\chi_p(z+z'+2x+2y-A+4) \geq 0.$$
Further, these cannot both be zero. Indeed, if $z = -2$, then
$$x^2 + y^2 + 4 = -2xy + Ax + Ay - 2C - \frac{A^2}{4} + 2C + 4,$$
so that
$$x^2 + 2xy + y^2 - Ax - Ay + \frac{A^2}{4} = \left(x+y-\frac{A}{2}\right)^2 = 0,$$
or, in other words,
$$A = 2x + 2y = 2x - yz = 2y - xz$$
as we have assumed $z = -2$. We may also rewrite the equation
$$z+z' + 2x + 2y - A + 4 = 0,$$
as
$$C + xy + 2x + 2y - A + 4 = 0.$$
As we know that
$$x + y -\frac{A}{2} = 0,$$
this is equivalent to
$$C = -4 - xy = 2z - xy.$$
But then Proposition \ref{singletonOrbits} implies that $(x, y, z)$ is a fixed point of each of $V_1$, $V_2$, and $V_3$, and so is not a point of $\mathcal{S}_{A, B, C, D}^*(p).$ We may thus write $\mathcal{S}_{A, B, C, D}^*(p) = \mathcal{S}_1 \sqcup \mathcal{S}_2$, where 
$$\mathcal{S}_1 = \{(x, y, z)\in\mathcal{S}_{A, B, C, D}^*(p)\mid\chi_p(z+2)\geq 0,\,\chi_p(z+z'+2x+2y-A+4)\geq 0\}$$
and
$$\mathcal{S}_2 = \{(x, y, z)\in\mathcal{S}_{A, B, C, D}^*(p)\mid\chi_p(z+2)\leq 0,\,\chi_p(z+z'+2x+2y-A+4)\leq 0\}.$$
To prove our theorem we show these two sets are invariant under $V_1$, $V_2$, and $V_3$, giving our quadratic obstruction. We already know this for $V_3$: we have that $\chi_p(z+2)\chi_p(z'+2) \geq 0$, and by symmetry in $z$ and $z'$ applying $V_3$ will not change the value of $\chi_p(z+z'+2x+2y-A+4)$ one bit. We also have that $V_1$ and $V_2$ do not change the value of $z$, so that $\chi_p(z+2)$ is invariant under $V_1$ and $V_2$. The hard part is to show that they cannot change the value of $\chi_p(z+z'+2x+2y-A+4) = \chi_p(xy + 2x + 2y + C - A + 4).$ We do this for $V_2$; the proof for $V_1$ is identical as the quantity and equation are symmetric in $x$ and $y$. The key identity is
\[\begin{split}
&(xy + 2x + 2y + C - A + 4)(x(xz + A - y) + 2x + 2(xz+A-y) + C - A + 4) \\&= (xy + 2x + 2y + C - A + 4)(x^2z + 2xz - xy + 2x + Ax - 2y + C + A + 4)\\
&\equiv \frac{1}{4}(Ax - 2x^2 - 2xz + 2C - 4x - 4z)^2\text{ mod } (x^2 + y^2 + z^2 - xyz - Ax - Ay - Cz +\frac{A^2}{4} - 2C - 4),
\end{split}\]
which may be verified with Macaulay2 \cite{Macaulay2}.

It is straightforward to show that $|\mathcal{S}_{A, B, C, D}^*(p)| = p^2 + O(p)$ where the implied constant doesn't depend on $A, B, C, D$, or $p$, and that each set $\mathcal{S}_i$ can have at most $\frac{1}{2}p^2 + O(p)$ elements. As $\mathcal{S}_1\cup\mathcal{S}_2 = \mathcal{S}_{A, B, C, D}^*(p),$ the corresponding upper bound must also hold, i.e., each $\mathcal{S}_i$ must contain approximately $\frac{p^2}{2}$. In particular, for $p$ sufficiently large, neither $\mathcal{S}_i$ can be empty, and in fact each $\mathcal{S}_i$ is quite large, giving a dramatic failure of strong approximation.

If we have $A = B = \pm C$, then we get decompositions of this form depending on each of $\chi_p(x\pm2)$, $\chi_p(y\pm2),$ and $\chi_p(z+2)$; \textit{a priori} intersecting these decompositions would seem to give eight invariant sets. However, exactly four of these are empty, as we have
\[\begin{split}(x\pm2)(y\pm2)(z+2) &= xyz + 2xy \pm 2yz \pm 2xz \pm 4x \pm 4y + 4z + 8 \\ 
&= x^2 + y^2 + z^2 + 2xy \pm 2yz \pm 2xz \\
&\,\,\,\,\,\,\,\,\,\,\,\,\,\,\,+ (\pm 4-A)x + (\pm 4-A)y + (4\mp A)z + \frac{A^2}{4} \mp 2A + 4\\
&= \left(\pm x \pm y + z + 2 \mp \frac{A}{2}\right)^2.\end{split}\]
Away from those problematic points where either $x\pm2=0, y\pm2=0$, or $z+2=0$, we must have that either zero or two of $x\pm2, y\pm2,$ or $z+2$ are quadratic nonresidues, proving that four of the potential invariant sets contain only these problematic points; however, each problematic point has a nonproblematic neighbor and must be in the same invariant set as that neighbor, so that those four potential invariant sets are truly empty. Finally, adding in transpositions and negated transpositions collapses together the sets in which exactly two of $(x\pm2), (y\pm2),$ or $(z+2)$ are quadratic nonresidues. 
\end{proof}
For future reference, we record these $\Gamma$-invariant subsets.
\begin{notation}
    Suppose that $(A, B, C, D)$ is a degenerate quadruple of parameters modulo $p$, normalized via equivalence so that $A \equiv B\pmod{p}.$ If $C \neq \pm A,$ then we set
$$\mathcal{S}_1(p) = \{(x, y, z)\in\mathcal{S}_{A, B, C, D}^*(p)\mid\chi_p(z+2)\geq 0,\,\chi_p(z+z'+2x+2y-A+4)\geq 0\}$$
and
$$\mathcal{S}_2(p) = \{(x, y, z)\in\mathcal{S}_{A, B, C, D}^*(p)\mid\chi_p(z+2)\leq 0,\,\chi_p(z+z'+2x+2y-A+4)\leq 0\};$$
these are the $\Gamma$-invariant subsets identified in Theorem \ref{QuadraticObstructions}. Here, as there, $z' = C + xy - z$. If $C = \pm A,$ then we set
\[\begin{split}\mathcal{S}_1(p) = \{(x, y, z)\in\mathcal{S}_{A, B, C, D}^*(p)\mid\,&\chi_p(z+2)\geq 0,\,\chi_p(z+z'+2x+2y-A+4)\geq 0,\\ &\chi_p(\pm x+2)\geq 0,\,\chi_p(\pm x\pm x'+2y\pm 2z-A+4)\geq 0\},\end{split}\]
\[\begin{split}\mathcal{S}_2(p) = \{(x, y, z)\in\mathcal{S}_{A, B, C, D}^*(p)\mid\,&\chi_p(z+2)\geq 0,\,\chi_p(z+z'+2x+2y-A+4)\geq 0,\\ &\chi_p(\pm x+2)\leq 0,\,\chi_p(\pm x\pm x'+2y\pm 2z-A+4)\leq 0\},\end{split}\]
\[\begin{split}\mathcal{S}_3(p) = \{(x, y, z)\in\mathcal{S}_{A, B, C, D}^*(p)\mid\,&\chi_p(z+2)\leq 0,\,\chi_p(z+z'+2x+2y-A+4)\leq 0,\\ &\chi_p(\pm x+2)\geq 0,\,\chi_p(\pm x\pm x'+2y\pm 2z-A+4)\geq 0\},\end{split}\]
and
\[\begin{split}\mathcal{S}_4(p) = \{(x, y, z)\in\mathcal{S}_{A, B, C, D}^*(p)\mid\,&\chi_p(z+2)\leq 0,\,\chi_p(z+z'+2x+2y-A+4)\leq 0,\\ &\chi_p(\pm x+2)\leq 0,\,\chi_p(\pm x\pm x'+2y\pm 2z-A+4)\leq 0\};\end{split}\] 
these are the $\Gamma$-invariant sets identified in Theorem \ref{QuadraticObstructions}. Once again, $z' = C + xy - z$; additionally we define $x' = A + yz - x$. We note that under the action of the larger group $\Gamma',$ $\mathcal{S}_2$, $\mathcal{S}_3,$ and $\mathcal{S}_4$ are mapped to each other by permutations and negated permutations of coordinates.

Throughout the paper, when $p$ is fixed, we abbreviate $\mathcal{S}_i(p)$ to $\mathcal{S}_i$ in both the two-orbit and four-orbit case.
\end{notation}
\section{Modifications for Degenerate Parameters}\label{ModificationsDegenerate}
For degenerate parameters $(A, B, C, D),$ the previous section showed that there must be at least two orbits on $\mathcal{S}_{A, B, C, D}^*(p)$ for $p$ large enough, and in some cases at least four orbits. For the parameters $(A, B, C, D) = (0, 0, 0, 4),$ corresponding to the Cayley cubic, there are many more orbits (see Remark \ref{rmk:CayleyParams}), but for other degenerate parameters, we can show that there are in fact exactly two or four orbits corresponding to the $\Gamma$-invariant sets identified in the previous section by adapting our methods from Sections \ref{ConicsSection}, \ref{Comparison}, \ref{EndgameSection}, \ref{middlegameSection}, and \ref{openingSection}. In fact, as long as the parameters are not $(0, 0, 0, 4),$ the proofs of the results from these sections go through essentially unchanged, save for the endgame and cage connectivity arguments of Section \ref{EndgameSection}. In this section, we prove the following theorem, which we state in the most precise form possible using the notations of Subsection \ref{subsect:ProofCompletion}.\begin{theorem}\label{DegenerateParameterOrbits}
    Suppose that $(A, B, C, D)$ is degenerate, has been normalized so that \[
A=B,\qquad 4D+A^2=8C+16,
\] and is not equal to $(0, 0, 0, 4)$. If $C \neq \pm A$, then $\Gamma$ acts transitively on each of $\mathcal{S}_1$ and $\mathcal{S}_2$ for all primes outside of $\mathcal{N}\cup\mathcal{N}'$. If $C = \pm A$, then $\Gamma$ acts transitively on each of $\mathcal{S}_1$, $\mathcal{S}_2$, $\mathcal{S}_3,$ and $\mathcal{S}_4$ for all primes outside of $\mathcal{N}\cup\mathcal{N}'$. 
As with Theorem \ref{th:mainBadPrimesExplicit}, we may also take the exceptional set of primes to be 
    $$\mathcal{N}\cup\{p:p \leq (20 + 2|A| + 2|B| + 2|C| + |D|)^{272,097,792}\}.$$
\end{theorem}
The proof of this theorem is long, relying ultimately on a breakdown into eleven distinct cases, and will occupy the remainder of this section, broken into a series of lemmas. We begin by recording the reductions inherited from the nondegenerate case, as mentioned above:
\begin{lemma} \label{lemma:DegenerateInheritedReductions}
    Suppose that $(A,B,C,D)$ is degenerate, is not equivalent to the Cayley parameters $(0,0,0,4)$, and has been normalized so that
\[
A=B,\qquad 4D+A^2=8C+16.
\]
Then the opening and middlegame arguments of Sections \ref{openingSection} and \ref{middlegameSection} go through for $\mathcal{S}_{A,B,C,D}^*(p)$, after restricting throughout to a single one of the sets identified in Theorem \ref{QuadraticObstructions}.
\end{lemma}
\begin{proof}
    For the opening, the only modification needed in this setting is our treatment of double fixed points. We must treat this phenomenon a little bit more carefully, as there are approximately $p$ double fixed points in the degenerate case. However, the opening works entirely within a single Dehn-twist orbit, and all but at most 12 double fixed points arise from points with some coordinate equal to \(\pm2\). In a conic section not obtained by fixing that particular coordinate, there are at most two such points, so the opening argument is unaffected.

    For the middlegame, no modifications are needed, as the only exceptional case is the Cayley cubic.    
\end{proof}
We will separate the remainder of the proof of Theorem \ref{DegenerateParameterOrbits} into the two orbits and four-orbit cases.
\subsection{The Two-Orbit Case}
In the two-orbit case, we are able to make use of the larger group $\Gamma_{xy}\coloneq \langle\Gamma, \tau_{xy}\rangle$; as in this case we have $C \neq \pm A,$ this only requires the easy parts of Theorem \ref{comparisonOfGroupsThm}. 
\begin{remark}[$\Gamma_{xy}$ vs. $\Gamma'$.]
    In the case where $A = B = 0,$ $\Gamma_{xy}$ is not equal to $\Gamma',$ as $\Gamma_{xy}$ does not include the map $\operatorname{neg}_{xy}.$ As it turns out, in the course of Lemma \ref{lemma:TwoPieceS1CageConnectivityAZero} where $A = B = 0$ we will have to use the full group $\Gamma',$ but we will transfer transitivity from $\Gamma'$ to $\Gamma$ by way of $\Gamma_{xy}$ rather than by way of $\Gamma\times H.$ 
\end{remark}
\begin{lemma}\label{lemma:Degenerate2OrbitsComparisonOfGroups}
    Suppose that $(A, B, C, D)$ is degenerate, with $A = B \neq \pm C,$ normalized so
    $$4D + A^2 = 8C + 16,$$
    and that $p$ is sufficiently large (independent of the parameters). Suppose that $\Gamma_{xy}$ acts transitively on $\mathcal{S}_1$ (respectively, $\mathcal{S}_2$). Then $\Gamma$ also acts transitively on $\mathcal{S}_1$ (respectively, $\mathcal{S}_2$). 
\end{lemma}
\begin{proof}
    As in the proof of Theorem \ref{comparisonOfGroupsThm}, since there is only one transposition, it suffices to show that each of $\mathcal{S}_1$ and $\mathcal{S}_2$ contains a point fixed by $\tau_{xy}.$ Such points in $\mathcal{S}_1$ correspond to solutions of the system 
    \begin{gather*}
    2X^2 + Z^2 - X^2Z - 2AX - CZ - D = 0\\
    Z = u^2 - 2\\
    Z\neq -2,\\
    \end{gather*}  
    while those in $\mathcal{S}_2$ correspond to solutions of the system 
    \begin{gather*}
        2X^2 + Z^2 - X^2Z - 2AX - CZ - D = 0\\
    Z = \omega u^2 - 2\\
    Z\neq -2,\\
    \end{gather*}
    where $\omega$ is a fixed quadratic nonresidue. We will show that these define (Zariski) open subsets of geometrically irreducible curves, from which the Weil bound will produce the points we need. These two systems are equivalent over $\overline{\mathbb{F}}_p$, so it
suffices to treat the first. We write the first equation as
\[
(2-Z)X^2-2AX+(Z^2-CZ-D)=0
\]
and set
\[
W=(2-Z)X-A.
\]
Using
\[
D=2C+4-\frac{A^2}{4},
\]
we then obtain
\[
\begin{split}
W^2
&=A^2-(2-Z)(Z^2-CZ-D)\\
&=(Z+2)\left(
Z^2-(C+4)Z+\frac{A^2}{4}+2C+4
\right).
\end{split}
\]
Let
\[
q(Z)=Z^2-(C+4)Z+\frac{A^2}{4}+2C+4.
\]
The discriminant of $q$ is
\[
(C+4)^2-4\left(\frac{A^2}{4}+2C+4\right)
=C^2-A^2,
\]
which is nonzero because $C\neq\pm A$. Hence $q(Z)$ is not a square in
$\overline{\mathbb{F}}_p(Z)$. Therefore, the first curve ramifies over the roots of $q(Z);$ as
$$Z = u^2 - 2$$
only ramifies above $-2$ and at infinity, the monodromy group acts transitively on the fiber product, establishing geometric irreiducibility.
\end{proof}

We are now in a position to adapt the endgame to our setting. We remark that $f_i$ is a perfect square polynomial for $i = 1, 2.$ Indeed, inputting our degeneracy condition into the formula for $f_1(x)$ yields
    \[\begin{split}
        f_1(x) &= x^4 - Ax^3 - \left(2C + 8 - \frac{A^2}{4}\right)x^2 + A(C+4)x + (C^2 + 8C + 16)\\
        &= \left(x^2 - \frac{A}{2}x - (C+4)\right)^2;
    \end{split}\]
a similar identity also holds for $f_2(y)$. Thus, we must remove equations involving $f_1$ and $f_2$ from systems in the endgame in order to apply the Weil bound. As $A = B$, we also have to remove the equation $f_3$ from consideration, but this is less of an issue as when moving through $C_3(z)$ we can use the map $\tau_{xy}\circ V_1$ due to Lemma \ref{lemma:Degenerate2OrbitsComparisonOfGroups}, as we have already done in the nondegenerate case of the endgame under $A = B$; this map acts transitively on $C_3(z)$ if $z$ is of maximal order. With these modifications made, we recover irreducibility of the relevant systems in the endgame, and can argue as there to establish:
\begin{lemma}\label{endgameDegenerate}
Suppose that $(A,B,C,D)$ is degenerate, normalized so
\[
A=B,\qquad 4D+A^2=8C+16,
\]
and that $C \neq \pm A.$
Then for $p$ sufficiently large (independent of the parameters) any point of order $\gg p^{1/2 + \delta}$ in either $\mathcal{S}_1$ or $\mathcal{S}_2$ connects to a point in the same set whose $x$- or $y$-coordinate has maximal order.
\end{lemma} 
\begin{remark}
    In contrast, we cannot connect an arbitrary such point to one whose $z$-coordinate is of maximal order, and indeed we shall see below that any point in $\mathcal{S}_1$ whose $z$-coordinate is not $2$ will have $z$ of order at most $\frac{p\pm 1}{2}$. In order to prove Lemma \ref{endgameIrreducibility}, we needed to know that $\alpha_1\alpha_2\neq (\pm 2 - \alpha_3)^2$ as rational functions of $x$. For nondegenerate parameters, this was immediate, as $\kappa_1 = \alpha_1\alpha_2$ was not a perfect square rational function; in the degenerate case this may be shown manually for the functions $\alpha_i(x)$ which arise when taking the $y$-coordinates of the orbit under $D_1$ (equivalently, when taking the $x$-coordinates of the orbit under $D_2$), but fails for the functions which arise when taking the $z$-coordinates of said orbits.
\end{remark}

On conics $C_3(z)$ for $z$ of maximal order, the rotation $\tau_{xy}\circ V_1$ acts transitively, connecting the entire conic section, but on conics $C_1(x)$ or $C_2(y)$ with $x$ or $y$ hyperbolic or elliptic and of maximal order, no transposition intervenes to produce transitivity, and $\langle V_2, V_3\rangle$ (respectively, $\langle V_1, V_3\rangle$) splits $C_1(x)$ (respectively, $C_2(y)$) into two orbits. These orbits do not join up even after including the action by $V_1$ (respectively, $V_2$), and in fact one lies in $\mathcal{S}_1$ while the other lies in $\mathcal{S}_2.$ 

However, when $x$ is of maximal order (hyperbolic or elliptic), the two orbits can be identified in the coordinate systems of Lemmas \ref{hyperbolics} and \ref{elliptics} relatively simply, as we shall discuss below.

\begin{lemma}[Conic components and the sign of \(z+2\)]
\label{lemma:TwoPieceConicSignComponents}
Suppose that
\[
A=B\neq \pm C,\qquad 4D+A^2=8C+16.
\]
Let \(x\neq \pm2\) have maximal order, with $\kappa_1(x) \neq 0,$ and parameterize \(C_1(x)\) as in
Lemmas \ref{hyperbolics} and \ref{elliptics}, so that
\[
x=\xi+\xi^{-1}
\]
and \(D_1=V_3\circ V_2\) acts on the parameter \(t\) by
\[
t\mapsto \xi^2t.
\]
Then \(C_1(x)\cap\mathcal{S}_{A,B,C,D}^*(p)\) is the union of two
\(\langle V_2,V_3\rangle\)-orbits, and these two orbits are distinguished by
the extended quadratic character invariant defining \(\mathcal S_1\) and
\(\mathcal S_2\).

More explicitly, in the hyperbolic case \(t\in\mathbb F_p^*\), the two
\(\langle V_2,V_3\rangle\)-orbits on \(C_1(x)\) are the two cosets of
\((\mathbb F_p^*)^2\) in the \(t\)-coordinate. In this case,
\[
\chi_p(z+2)=1
\]
if and only if \(t\) lies in the nonsquare coset of \(\mathbb F_p^*\), while
\[
\chi_p(z+2)=-1
\]
if and only if \(t\) lies in the square coset.

In the elliptic case, write
\[
\lambda_x=\frac{x^2-\frac A2x-(C+4)}{x^2-4},
\]
so that, in the degenerate setting,
\[
\kappa_1(x)=\lambda_x^2.
\]
In the parameterization by \(\widetilde{C}_1(x)\), set
\[
u=\frac{\xi t}{\lambda_x}.
\]
Then \(u^{p+1}=1\), and the two \(\langle V_2,V_3\rangle\)-orbits on
\(C_1(x)\) are distinguished by the value of \(u^\frac{p+1}{2}\), and hence by the extended sign invariant defining the $\mathcal{S}_i$.
\[
\chi_p(z+2)=1
\]
if and only if \(u^{\frac{p+1}{2}}=\chi_p(\lambda_x)\), and
\[
\chi_p(z+2)=-1
\]
if and only if \(u^{\frac{p+1}{2}}=-\chi_p(\lambda_x)\).

The analogous statements hold for \(C_2(y)\), with the roles of \(x\) and
\(y\), and of \(D_1=V_3\circ V_2\) and \(D_2=V_3\circ V_1\), interchanged.
Thus one component of each such conic lies in \(\mathcal S_1\), and the other
lies in \(\mathcal S_2\).
\end{lemma}
\begin{proof}
    For points on $C_1(x)$ with $x$ hyperbolic or elliptic of maximal order, applying the parameterization from Lemmas \ref{hyperbolics} and \ref{elliptics} gives
    \[\begin{split}
        z + 2 &= (x^2 - 4)^{-2}\left(\xi t(x^2-4)^2 + (2(x^2-4) - Ax - 2C)(x^2 - 4) + \frac{f_1(x)}{\xi t}\right)\\
        &=(x^2 - 4)^{-2}\frac{1}{\xi t}\left(\xi^2 t^2 (x^2-4)^2 + 2\left(x^2 - \frac{A}{2}x - (C+4)\right)\xi t + \left(x^2 - \frac{A}{2}x - (C+4)\right)^2\right)\\
        &= (x^2 - 4)^{-2}\frac{1}{\xi t}\left(\xi t (x^2 - 4) + \left(x^2 - \frac{A}{2}x - (C+4)\right)\right)^2.
    \end{split}\]
    In the hyperbolic case, we are done: $z + 2$ is a quadratic residue essentially if and only if $\xi t$ is, if and only if $t$ is a quadratic nonresidue. As the action in these coordinates by $V_3\circ V_2$ takes $t\mapsto \xi^2 t,$ this perfectly distinguishes the two components of $C_1(x).$ 

    In the elliptic case, we now note that
    \[\begin{split}
        (z+2)^\frac{p-1}{2} &= (x^2 - 4)^{-(p-1)}\frac{\left(\xi t (x^2 - 4) + \left(x^2 - \frac{A}{2}x - (C+4)\right)\right)^{p-1}}{(\xi t)^{\frac{p-1}{2}}}\\
        &=\frac{\left(\xi t (x^2 - 4) + \left(x^2 - \frac{A}{2}x - (C+4)\right)\right)^{p}}{(\xi t)^{\frac{p-1}{2}}\left(\xi t (x^2 - 4) + \left(x^2 - \frac{A}{2}x - (C+4)\right)\right)}\\
        &= \frac{(\xi t)^p (x^2 - 4) + \left(x^2 - \frac{A}{2}x - (C+4)\right)}{(\xi t)^{\frac{p-1}{2}}\left(\xi t (x^2 - 4) + \left(x^2 - \frac{A}{2}x - (C+4)\right)\right)}\\
        &= \frac{\frac{\kappa_1(x)}{\xi t} (x^2 - 4) + \left(x^2 - \frac{A}{2}x - (C+4)\right)}{(\xi t)^{\frac{p-1}{2}}\left(\xi t (x^2 - 4) + \left(x^2 - \frac{A}{2}x - (C+4)\right)\right)}\\
        &= \frac{ \left(x^2 - \frac{A}{2}x - (C+4)\right)^2(x^2 - 4)^{-1} + \xi t\left(x^2 - \frac{A}{2}x - (C+4)\right)}{(\xi t)^{\frac{p+1}{2}}\left(\xi t (x^2 - 4) + \left(x^2 - \frac{A}{2}x - (C+4)\right)\right)}.\\
    \end{split}\]
    Now, as
    $$(\xi t)^{p+1} = \kappa_1(x) = \frac{\left(x^2 - \frac{A}{2}x - (C+4)\right)^2}{(x^2-4)^2},$$
    we have that
    $$(\xi t)^\frac{p+1}{2} = \sigma\frac{x^2 - \frac{A}{2}x - (C+4)}{x^2-4}$$
    for $\sigma = \pm 1,$ where $\sigma = 1$ if and only if $u \coloneq \frac{\xi t}{\lambda_x} = a^2$ for some $a\in\mu_{p+1}.$ In this notation, we get
    \[\begin{split}
        (z+2)^\frac{p-1}{2} &= \frac{ \left(x^2 - \frac{A}{2}x - (C+4)\right)^2(x^2 - 4)^{-1} + \xi t\left(x^2 - \frac{A}{2}x - (C+4)\right)}{\sigma\left(x^2 - \frac{A}{2}x - (C+4)\right)(x^2-4)^{-1}\left(\xi t (x^2 - 4) + \left(x^2 - \frac{A}{2}x - (C+4)\right)\right)}\\
        &= \frac{\left(x^2 - \frac{A}{2}x - (C+4)\right) + \xi t(x^2 - 4)}{\sigma \left(\xi t (x^2 - 4) + \left(x^2 - \frac{A}{2}x - (C+4)\right)\right)}\\
        &= \frac{1}{\sigma}\\
        &= \sigma.
    \end{split}\]
    We find that $z + 2$ is a quadratic residue if and only if $\sigma = 1$, which occurs if and only if $u$ is a square in $\mu_{p+1}$. From this we can obtain the classification stated in the lemma. As the action in these coordinates by $V_3\circ V_2$ takes $t\mapsto \xi^2 t,$ this again perfectly distinguishes the two components of $C_1(x).$ 
\end{proof} 
    
    To prove our result, it suffices now to prove a cage-connectivity-type result for each of $\mathcal{S}_1$ and $\mathcal{S}_2$, say allowing us to connect the component of $C_1(x)$ in $\mathcal{S}_i$ to $C_2(y)$ by way of some $C_3(z)$ lying in $\mathcal{S}_i$ for $i = 1, 2$ assuming $x$ and $y$ have maximal order. A complication is that, writing 
    $$z = \zeta + \zeta^{-1},$$
    we have
    $$z + 2 = \zeta + 2 + \zeta^{-1} = \zeta^{-1}(\zeta + 1)^2,$$
    so that for hyperbolic $z$, $z + 2$ is a quadratic residue if and only if $\zeta$ is. For elliptic $z$, we can compute that
    $$(z+2)^\frac{p-1}{2} = \frac{(\zeta + 1)^p}{\zeta^\frac{p-1}{2}(\zeta + 1)} = \frac{1 + \zeta^{-1}}{\zeta^\frac{p-1}{2}(\zeta + 1)} = \frac{\zeta + 1}{\zeta^\frac{p+1}{2}(\zeta + 1)} = \zeta^{-\frac{p+1}{2}},$$
    so that once again $z+2$ is a quadratic residue if and only if $\zeta$ is a square in $\mu_{p+1}.$ In either case, we lack transitive action on $C_3(z)$ whenever $(x, y, z)\in\mathcal{S}_1.$ Thus, we shall have to treat the two components separately.
\begin{lemma}[Cage connectivity for \(\mathcal{S}_2\)]
\label{lemma:TwoPieceS2CageConnectivity}
Suppose that
\[
A=B\neq \pm C,\qquad 4D+A^2=8C+16.
\]
For all sufficiently large primes $p$ (independent of the parameters),
the cage inside $\mathcal{S}_2$ is connected. Explicitly, if $x$ and $y$ are coordinates of maximal order for which
\[
C_1(x)\cap \mathcal{S}_2\neq \emptyset,\qquad
C_2(y)\cap \mathcal{S}_2\neq \emptyset,
\]
then, except possibly for a bounded number of exceptional pairs $(x,y)$, there is a value $z$ of maximal order such that
\[
C_1(x)\cap C_3(z)\cap \mathcal{S}_2\neq \emptyset
\]
and
\[
C_2(y)\cap C_3(z)\cap \mathcal{S}_2\neq \emptyset.
\]
Consequently the corresponding maximal-order conic components in
$\mathcal{S}_2$ lie in a single $\Gamma_{xy}$ orbit.
\end{lemma}
\begin{proof}
    As in the nondegenerate case, we will achieve this by applying the Weil bound and sieving. As in Subsection \ref{cageConnectivitySubsect}, those choices of (say, hyperbolic) $z$ of order dividing $(p-1)/\ell$ for which $z+2$ is a quadratic (non)residue are given by solutions to the system
    \[\begin{split}
    &z = t^\ell + t^{-\ell}\\
    &z = \omega u^2 - 2\\
    &z^2 + A_1 z - B_1 v^2 = C_1\\
    &z^2 + A_2 z - B_2w^2 = C_2\\
\end{split}\]

where $\omega$ is a fixed quadratic nonresidue and $A_1, A_2, B_1, B_2, C_1,$ and $C_2$ are polynomials in $x$ and $y$. As we are working in $\mathcal{S}_2,$ we already know from Lemma \ref{lemma:TwoPieceConicSignComponents} that if
$$z = \zeta + \zeta^{-1},$$
then the parameter \(\zeta\) must be a quadratic nonresidue in $\mathbb{F}_p^\times$, whence only odd values of \(\ell\) contribute to the sieve for maximal-order \(z\). Thus, we may restrict our attention to those equations where $\ell$ is odd. We consider each equation as defining a ramified cover of $\mathbb{P}_z^1,$ and the system as defining a projective curve which is the fiber product of these ramified covers. As in Subsection \ref{cageConnectivitySubsect}, by removing a bounded number of values of $x$ and $y$, and an additional bounded number of pairs $(x, y)$ from consideration, we can force the last two equations not to ramify at $z=\pm2$ and not to have the same ramification points; the values of $x$ and $y$ and pairs $(x, y)$ we exclude may then be connected to the cage via the endgame as in Subsection \ref{cageConnectivitySubsect}.

We must treat the first two equations with a bit more care, however: the cover associated to the first equation ramifies at $z = 2,-2, \infty$ and the cover associated to the second equation ramifies at $z = -2, \infty.$ As in \cite{BGS}, Lemma 12 or Subsection \ref{cageConnectivitySubsect}, the branch cycles for the first equation are given by

\begin{equation}
\begin{cases} \sigma_{-2} = (12)(34)\dots(2\ell-1\, 2\ell)\\
\sigma_2 = (1\,2\ell)(23)\dots(2\ell-2\, 2\ell-1)\\
\sigma_\infty = (135\dots2\ell-1)(246\dots2\ell),\end{cases}
\end{equation}

while the second equation defines a two sheeted cover, whose branch cycles simply interchange the two sheets. For ease of reference, we will label these sheets $a$ and $b$. As we have assumed $\ell$ to be odd, the action of $\sigma_\infty$ on sheets $1$ through $2\ell$ may be written as the product of two disjoint cycles of odd order, so that the same holds true of the action of $\sigma_\infty^2$ on these sheets. In contrast, $\sigma_\infty^2$ acts trivially on sheets $a$ and $b$ of the cover associated to the second equation. Thus, the product of monodromy groups will act transitively on the sheets of the fiber product: from a starting sheet $(i, c)$ with $i\in\{1, 2, \dots, 2\ell\}$ and $c = a, b,$ we can get to the sheet $(j, d)$ by first applying either $\sigma_2$ or $\sigma_{-2}$ to ensure that the parities of $i$ and $j$ and the values of $c$ and $d$ match, then acting by powers of $\sigma_\infty^2.$ In this case, the Weil bound implies that the system has $p + O(\ell\sqrt{p})$ solutions $(z, t, u, v, w);$ as the projection down to $z$ is generically $16\ell$-to-1, we get 
$$\frac{p}{16\ell} + O(\sqrt{p})$$
possible values of $z$.

Sieving now over odd divisors of $p-1$, and noting that 
$$z = t^\ell + t^{-\ell}$$
for even values of $\ell$ cannot yield the desired sign invariant, we find that the number of values of $z$ of maximal order connecting $C_1(x)$ to $C_2(y)$ is

$$\frac{p-1}{16}\sum\limits_{\substack{d|(p-1)\\ d\text{ odd}}}\frac{\mu(d)}{d} + O_\epsilon\left(p^{\frac{1}{2} + \epsilon}\right).$$
Factoring $p-1$ as $2^km$ for $m$ odd, we may simplify the sum to
$$\frac{2^k\varphi(m)}{16} + O_\epsilon\left(p^{\frac{1}{2} + \epsilon}\right) =\frac{\varphi(p-1)}{8} + O_\epsilon\left(p^{\frac{1}{2} + \epsilon}\right),$$
which is positive for $p$ sufficiently large as before. 
\end{proof} 

At this point, we have shown that for a density-one set of primes, $\mathcal{S}_2$ is connected, and that in any case, $\mathcal{S}_2$ has a large connected component. Let us now turn to $\mathcal{S}_1$. We first assume that $A = B \neq 0;$ in this setting the presence of parabolics makes connectivity relatively more straightforward:

\begin{lemma}[Cage connectivity for \(\mathcal{S}_1\) when \(A\neq0\)]
\label{lemma:TwoPieceS1CageConnectivityNonzeroA}
Suppose that
\[
A=B\neq 0,\qquad C\neq \pm A,\qquad 4D+A^2=8C+16.
\]
Then, for all sufficiently large primes $p$, every sufficiently high-order point of $\mathcal{S}_1$ is connected, within
$\mathcal{S}_1$, to the parabolic conic section $C_3(2),$ which itself forms a connected subgraph of $\mathcal{S}_1.$
\end{lemma}
\begin{proof}
We first note that $C_3(2)\cap\mathcal{S}^*\subset\mathcal{S}_1$ as for points in $C_3(2)$ we have $z + 2 = 4$, a quadratic residue. As we have assumed $A = B \neq 0,$ so that $A \neq -B$, Lemma \ref{Parabolics} implies that $C_3(2)$ is a single orbit, consisting of exactly $p$ points. For a fixed value of $x$ (respectively, $y$), there will be a point $(x, y, 2)$ on $\mathcal{S}_{A, B, C, D}^*(p)$ if and only if $8Ax + 16C$ (respectively, $8Ay + 16C$) is zero or a quadratic residue, by using $A = B$ and $4D + A^2 = 8C + 16$ in (\ref{parabolicCondition}), keeping in mind that (\ref{parabolicCondition}) is written with $x = \pm 2$ rather than $z = 2$ so that the coefficients $A, B,$ and $C$ must be appropriately permuted. Now, if $x$ and $y$ are of maximal order and for some $u, v\in\mathbb{F}_p$ we have
$$8Ax + 16C = u^2;$$
$$8Ay + 16C = v^2$$
then $C_1(x)\cap\mathcal{S}_1$ and $C_2(y)\cap\mathcal{S}_1$ are connected subgraphs, each of which contains a point with $z = 2$, and thus connect to each other by way of $C_3(2)$. To connect up all of $\mathcal{S}_1,$ we will need to modify the endgame to connect points of large order to points whose $x$ or $y$ order is maximal and for which $8Ax + 16C$ (respectively, $8Ay + 16C$) is a quadratic residue. This is rather straightforward: Starting from a point $(x, y, z)$ with $x$ of order at least $p^{\frac{1}{2} + \delta}$ the relevant system of equations is
$$y = t^d + t^{-d}$$
$$y = \alpha_1 s^e + \alpha_2 s^{-e} + \alpha_3$$
$$8Ay + 16C = u^2.$$
As the cover associated to the last equation ramifies over $\infty$ and 
$$y = -\frac{2C}{A}$$ 
and since we have assumed $C \neq \pm A$ and $A \neq 0$ we have that
$$-\frac{2C}{A} \neq \pm 2$$
and irreducibility follows, possibly excepting a small number of values of $x$, $y$, and pairs $(x, y)$ as in the endgame. Thus, we can sieve as before, and connectivity of $\mathcal{S}_1$ follows.
\end{proof}

If, however, we have $A = 0$, this approach can fail, as now we also have $A = -B$, so that $C_3(2)$ is nonempty if and only if $C$ is a quadratic residue mod $p$, and in this case $C_3(2)$ breaks into two orbits. Fortunately, we have access to another map, $\operatorname{neg}_{xy} :(x,y, z)\mapsto (-x, -y, z)$. We will first enlarge the group under consideration to $\Gamma'$, before transferring transitivity to $\Gamma_{xy}$, and from there to $\Gamma$ via Lemma \ref{lemma:Degenerate2OrbitsComparisonOfGroups}.
\begin{lemma}[Cage-Connectivity in $\mathcal{S}_1$ when \(A=B=0\), \(C\neq0\)]
\label{lemma:TwoPieceS1CageConnectivityAZero}
Suppose that
\[
A=B=0,\qquad C\neq0,\qquad D=2C+4.
\]
Then, for all sufficiently large primes $p$, if $x$ and $y$ have maximal order and
$$C_1(x)\cap \mathcal{S}_1\neq \emptyset\,\text{and}$$
$$C_2(y)\cap \mathcal{S}_1\neq\emptyset,$$
then, except for boundedly many pairs $(x,y)$, there is a path from some (equivalently, any) point of $C_1(x)$ to some (equivalently, every) point of $C_2(y)$ using edges from $\Gamma.$ 
\end{lemma}
\begin{proof}
We first seek to show that $\Gamma'$ connects up many choices of $C_3(z)$ in $\mathcal{S}_1$. Choose $\epsilon_0 = \pm1$ so that $p+\epsilon_0 \equiv 2$ (mod 4) (and thus $p - \epsilon_0 \equiv 0$ (mod 4)). If $z$ is chosen so that 
$$z = \zeta + \zeta^{-1}$$
with $\zeta$ of order $\frac{p+\epsilon_0}{2}$, so that $z$ is hyperbolic or elliptic depending on the value of $\epsilon_0$, and of the maximal order possible in $\mathcal{S}_1$, then the group generated by $\tau_{xy}\circ V_1$ and $\operatorname{neg}_{xy}$ acts transitively on $C_3(z)$, as in the coordinates given in Lemmas \ref{hyperbolics} and \ref{elliptics} $\tau_{xy}\circ V_1$ acts by multiplication by a value of (odd) order $\frac{p+\epsilon_0}{2}$ while $\operatorname{neg}_{xy}$ acts by multiplication by $-1$, which has order $2$. 

Before proceeding further, we must show, as in Section \ref{Comparison}, that a transitive $\Gamma'$ action on $\mathcal{S}_1$, or more generally on the large connected component containing all points of maximal order within $\mathcal{S}_1$ implies a transitive $\Gamma$ action on the same. By the proof of Lemma \ref{lemma:Degenerate2OrbitsComparisonOfGroups} modified as in the proof of Proposition \ref{prop:ComparisonLargeComponent} to apply to the large connected components, it suffices to show that a transitive $\Gamma'$ action implies a transitive $\Gamma_{xy}$ action. Now, the map $\operatorname{neg}_{xy}$ commutes with $V_1,$ $V_2,$ $V_3,$ and $\tau_{xy}$, so that $\Gamma_{xy}\unlhd \Gamma'$, so that as in the proof of Lemma \ref{lemma:RemovalOfH} it suffices to exhibit a point $P\in\mathcal{S}_1(p)$ which connects to $\operatorname{neg}_{xy}P$ via iterates of $\tau_{xy}\circ V_1$. But this is guaranteed to happen for any point $P$ on $C_3(z)$ where $z = \zeta + \zeta^{-1}$ is chosen so that $\zeta$ has multiplicative order $\frac{p-\epsilon_0}{2}.$ As there exist many points of $\mathcal{S}_1$ whose $z$-coordinates have this form, the result follows.

To complete our connection of $\mathcal{S}_1$ in the case $A = B = 0$, we need only show that there is a $z$ of order $\frac{p+\epsilon_0}{2}$ such that
$$C_1(x)\cap C_3(z)\neq \emptyset\text{ and}$$
$$C_2(y)\cap C_3(z) \neq \emptyset,$$
for all but a small number of pairs of values $(x,y)$ each of maximal order. If $\epsilon_0 = -1,$ this may be achieved with essentially the same system used to connect $\mathcal{S}_2,$ but with $\omega = 1$ instead of a fixed quadratic nonresidue. If $\epsilon_0 = 1,$ we need $z$ to be elliptic instead of hyperbolic. This may be handled just as the elliptic case of the endgame, replacing the equation
$$z = t^\ell + t^{-\ell}$$
with the pair of equations
\begin{equation}\label{EqForConnectingToElliptic}
    \begin{split}
        &z = 2 g_{\ell}(\xi, \eta)\\
    &\xi^2 - \omega\eta^2 = 1
    \end{split}
\end{equation}
where $\omega$ is a fixed quadratic nonresidue and $g_\ell$ is as in (\ref{ellipticsModificationPolys}). In either case, we have established connectivity for $\mathcal{S}_1$ in the case $A = B = 0$ as well.
\end{proof}

This completes our proof of transitivity on $\mathcal{S}_1$ for sufficiently large primes in a density-one set, and thus completes our work in the two-orbit case.

\subsection{The Four-Orbit Case}

We now study the four-orbit case, wherein $A = B = \epsilon C$ and $D = 2C + 4 - \frac{A^2}{4} = -\frac{A^2}{4} + 2\epsilon A + 4.$ We remark that as we exclude the Cayley cubic parameters $(A, B, C, D) = (0, 0, 0, 4)$ from consideration, in this case we have $A, B, C \neq 0$. We record some basic facts about this setting in the following lemma:
\begin{lemma}[Conic components in the four-piece case]
\label{lemma:FourPieceConicSignComponents}
Suppose that
\[
A=B=\epsilon C,\qquad \epsilon\in\{\pm1\},
\]
and that $(A,B,C,D)$ is not equivalent to the Cayley parameters. Then the relevant polynomials $f_1$, $f_2$, and $f_3$ are perfect squares. Moreover, for each non-parabolic conic section $C_i$, only two of the $\Gamma$-invariant sets $\mathcal{S}_1$, $\mathcal{S}_2$, $\mathcal{S}_3$, and $\mathcal{S}_4$ can possibly contain points in $C_i$ due to the sign relations defining the sets $\mathcal{S}_i$, and if said conic splits into exactly two orbits under the action of the subgroup generated by the two Vieta involutions fixing the conic setwise, those two orbits lie in different $\Gamma$-invariant sets.
\end{lemma}
\begin{proof}
    For the statements about the polynomials $f_1$, $f_2,$ and $f_3,$ note that 
    $$f_1(x) = x^4 - Ax^3 + \left(\frac{A^2}{4} \mp 2A - 8 \right)x^2 + A(\pm A + 4)x + (A^2 \pm 8A + 16) = \left(x^2 - \frac{A}{2}x - (\pm A+4)\right)^2,$$
the same polynomial formula holds for $f_2(y)$, and
$$f_3(z) = z^4 \mp Az^3 +\left(\frac{A^2}{4} \mp 2A - 8\right)z^2 + A(A\pm 4)z +\left(A^2 \pm 8A + 16\right) = \left(z^2 \mp \frac{A}{2}z \mp (A\pm 4)\right)^2.$$
The statement about conic components lying in different $\Gamma$-invariant sets follows exactly from the proof of \ref{lemma:TwoPieceConicSignComponents}.
\end{proof}

As in the two-orbit case, the opening and middlegame may be extended with minimal modifications, as has been recorded in Lemma \ref{lemma:DegenerateInheritedReductions}, and in the endgame we remove equations requiring $f_i$ to be a quadratic nonresidue. However, the endgame requires some further modification. The relevant proposition is as follows:
\begin{proposition}[Modified endgame in the four orbit case]
\label{proposition:FourPieceModifiedEndgame}
Suppose that
\[
A=B=\epsilon C,\qquad \epsilon\in\{\pm1\},
\]
and that $(A,B,C,D)$ is not equivalent to the Cayley parameters. Let
$i\in\{1,2,3,4\}$. Then, for all sufficiently large primes $p$, every sufficiently high-order point of $\mathcal{S}_i$ is connected via edges from $\Gamma$ to a point with a coordinate which may be freely chosen to be hyperbolic or elliptic and whose chosen coordinate has maximal allowable order in $\mathcal{S}_i$, namely one of
\[
p\pm1, \frac{p\pm1}{2},
\]
as determined by the sign invariant defining $\mathcal{S}_i$ and whether or not the coordinate is hyperbolic or elliptic.
\end{proposition}
\begin{proof}
Starting from a point $(x, y, z)$ with (say) $x$ of order at least $p^{1/2 + \delta},$ we seek to connect to a point whose $y$-coordinate has maximal order. Iterating $V_3\circ V_2$ connects it to points whose $y$-coordinates are of the form

$$\alpha_1t^e + \alpha_2 t^{-e} + \alpha_3$$

where $e$ is even and depends on $x = \xi + \xi^{-1}$ via 

$$e = \frac{(p\pm 1)\operatorname{gcd}(2, \text{ord}(\xi))}{\text{ord}(\xi)}.$$

If $x$ is elliptic, we shall connect to a hyperbolic value of $y$, and if $x$ is hyperbolic, we shall connect to an elliptic value of $y$; if we can do this for any such point, then repeating the process a second time will let us connect elliptic $x$ to elliptic $z$ and hyperbolic $x$ to hyperbolic $z$. As in the endgame for the nondegenerate case we shall work by inclusion/exclusion after counting solutions to 

\begin{equation}\label{eq:fourOrbitsDegenEndgame}
    \alpha_1t^{e} + \alpha_2t^{-e} + \alpha_3 = s^{d}  + s^{-d},
\end{equation}

and our demand that we pass from a hyperbolic point to an elliptic point or from an elliptic point to a hyperbolic point allows us to ensure that $(e, d) = 1$ or $2$. This allows us to control the connected components of the associated curve via the method of monodromy, as follows.

Letting $g(s) = s^d + s^{-d}$ and $h(t) = \alpha_1t^e + \alpha_2 t^{-e} + \alpha_3$, the curve is the fiber product $\mathbb{P}_{t}^1\times_{\mathbb{P}_y^1}\mathbb{P}_{s}^1$ with the covers $\mathbb{P}_{t}^1\rightarrow \mathbb{P}_{y}^1$ and $\mathbb{P}_{s}^1\rightarrow \mathbb{P}_{y}^1$ given by
\[\begin{split}
g: \mathbb{P}_{s}^1&\rightarrow \mathbb{P}_{y}^1\; \;\;\;\;\;\;g(s) - y = 0\\
h: \mathbb{P}_{t}^1&\rightarrow \mathbb{P}_{y}^1\; \;\;\;\;\;\; h(t) - y = 0.
\end{split}\]

As in \cite{BGS}, Lemma 12, the branch points for $g$ are $\{-2, 2, \infty\}$ with branch cycles given by
\begin{equation}
\begin{cases} \sigma_{-2} = (12)(34)\dots(2d-1\, 2d)\\
\sigma_2 = (1\,2d)(23)\dots(2d-2\, 2d-1)\\
\sigma_\infty = (135\dots2d-1)(246\dots2d),\end{cases}
\end{equation}
while the branch points for $h$ are $\{-2\sqrt{\alpha_1\alpha_2} + \alpha_3, 2\sqrt{\alpha_1\alpha_2} + \alpha_3, \infty\}$ with branch cycles given by
\begin{equation}
\begin{cases} \sigma_{-2\sqrt{\alpha_1\alpha_2} + \alpha_3} = (12)(34)\dots(2e-1\, 2e)\\
\sigma_{2\sqrt{\alpha_1\alpha_2}+\alpha_3} = (1\,2e)(23)\dots(2e-2\, 2e-1)\\
\sigma_\infty = (135\dots2e-1)(246\dots2e).\end{cases}
\end{equation}
As we have $\alpha_1\alpha_2 = \frac{f_1(x)}{(x^2 - 4)^2}$, the branch points of $h$ away from infinity may be written more simply as
$$\pm 2\frac{x^2 - \frac{A}{2}x - (\epsilon A + 4)}{x^2 - 4} + \frac{-2A - \epsilon Ax}{x^2 - 4} = \frac{\pm 2x^2 \mp Ax - \epsilon Ax \mp 2\epsilon A - 2A \mp 8} {x^2 - 4},$$
where $\epsilon$ is chosen so $C = \epsilon A$. This simplifies to give branch points at $-2\epsilon$ and at
$$2\epsilon\frac{x^2 - Ax - 2\epsilon A - 4}{x^2-4}.$$
In particular, both $g$ and $h$ ramify at $\infty$ and at $-2\epsilon,$ and our analysis of the monodromy groups must be modified. We split the analysis into a few lemmas.
\begin{lemma}\label{lemma:FourPieceCoprimeMonodromy}
    Suppose that $(e, d) = 1.$ Then, after excluding a bounded set of values of $x$, (\ref{eq:fourOrbitsDegenEndgame}) describes a geometrically irreducible curve.
\end{lemma}
\begin{proof}
    We have that $\mathbb{Z}/(ed) \cong\mathbb Z/(e)\times\mathbb Z/(d)$. It follows that repeated application of $\sigma_\infty$ can take any pair of sheets of $g$ and $h$ to any other pair whose indices have the same parities. As $2\epsilon$ is a ramification point for $g$ and not for $h$, and 
$$2\epsilon\frac{x^2 - Ax - 2\epsilon A - 4}{x^2-4}$$
is a ramification point for $h$ but not for $g$, we may adjust the parity of the index of a sheet of $g$ or of $h$ independently. It thus follows that the monodromy group acts transitively. 
\end{proof}
This argument may be adapted to the setting where $(e, d) > 1.$
\begin{lemma} \label{lemma:FourPieceNonCoprimeMonodromy}
    Suppose that $(e, d) >1.$ Then, after excluding a bounded set of values of $x$, (\ref{eq:fourOrbitsDegenEndgame}) describes a geometrically reducible curve, and over $\overline{\mathbb{F}}_p,$ the curve has exactly $(e, d)$ irreducible components. 
\end{lemma}
\begin{proof}
    The associated cover has $4ed$ sheets, each given by a pair of indices $(i, j)$ with $1\leq i\leq 2e$ and $1\leq j \leq 2d.$ We can decompose this collection into four subcollections stable under $\sigma_\infty$ by restricting the parities of each of $i$ and $j$, and then the cyclic group generated by $\sigma_\infty$ will decompose each of those subcollections into exactly $(e, d)$ orbits. Acting by $\sigma_{\pm2}$ and $\sigma_{2\epsilon\frac{x^2 - Ax - 2\epsilon A - 4}{x^2-4}}$ joins each orbit in each subcollection to exactly one orbit in another subcollection, and this is consistent with the decomposition of each subcollection into $(e, d)$ orbits, as $\sigma_\infty$ is a composition of $\sigma_2,$ $\sigma_{-2},$ and $\sigma_{2\epsilon\frac{x^2 - Ax - 2\epsilon A - 4}{x^2-4}}$.
\end{proof} 

As in Section \ref{EndgameSection}, those points where irreducibility fails can be connected to the cage by way of points where irreducibility does not fail via the endgame. As in the two-orbit case, we need to understand the relationship between the Legendre symbol of $y+2\epsilon$ and its order. As $\epsilon = -1$ is possible, the relationship is a little bit more complicated, as spelled out in the following lemma:

\begin{lemma}\label{lemma:FourPieceSignOrderDictionary}
    Suppose that $y$ is hyperbolic, with $y = \eta + \eta^{-1},$ $\eta\in\mathbb{F}_p^\times.$ Then $y+2\epsilon$ is a quadratic residue if and only if $\eta$ is as well.

    Suppose that $y$ is elliptic, with $y = \eta + \eta^{-1},$ $\eta\in\mu_{p+1}.$ If $\epsilon=1,$ then $y+2$ is a quadratic residue if and only if $\eta$ is a square in $\mu_{p+1},$ but if $\epsilon = -1$ then $y-2$ is a quadratic residue if and only if $\eta$ is not a square in $\mu_{p+1}.$
\end{lemma}
\begin{proof}
    In the hyperbolic case, the result follows as in the two-orbit case studied above from the identity

$$y + 2\epsilon = \eta + 2\epsilon + \eta^{-1} = \eta^{-1}(\eta + \epsilon)^2.$$

However, in the elliptic case, the relevant identity becomes

 $$(y+2\epsilon)^\frac{p-1}{2} = \frac{(\eta + \epsilon)^p}{\eta^\frac{p-1}{2}(\eta + \epsilon)} = \frac{\epsilon + \eta^{-1}}{\eta^\frac{p-1}{2}(\eta + \epsilon)} = \frac{\epsilon\eta + 1}{\eta^\frac{p+1}{2}(\eta + \epsilon)} = \epsilon\eta^{-\frac{p+1}{2}},$$
 from which the result clearly follows.
\end{proof}

We are now ready to resume our proof of Proposition \ref{proposition:FourPieceModifiedEndgame}. Our argument splits into a total of eight cases, by whether (1) $\epsilon = \pm1,$ (2) the value of $\chi_p(y+2\epsilon)$ for our starting point $(x, y, z)$, and (3) whether $x$ is hyperbolic or elliptic. The third level of this decomposition is relatively harmless, essentially amounting to our standard modification of the $\mathbb{F}_p$-parameterizations of the curves arising in the endgame, but the first two levels require slightly more adjustments; we therefore break our argument into a total of four lemmas.

\begin{lemma}\label{lemma:fourOrbitsEndgameCase1}
    Suppose that
\[
A=B=+C,
\]
and that $(A,B,C,D)$ is not equivalent to the Cayley parameters. Suppose $(x, y, z)\in\mathcal{S}_{A, B, C, D}^*(p)$ has $\chi_p(y+2) = -1$ and $x$ of order at least $p^{\frac{1}{2} + \delta}$. Then, for all sufficiently large primes $p$, $(x, y, z)$ is connected via edges from $\Gamma$ to a point whose $y$-coordinate has order $p\pm 1$.
\end{lemma}
\begin{proof}
As our starting point $(x, y, z)$ has $y + 2$ a quadratic nonresidue, for any point $(x', y', z')$ in the orbit, by Lemma \ref{lemma:FourPieceSignOrderDictionary}, writing $y' = \eta + \eta^{-1},$ we have that $\eta$ is a nonsquare in $\mathbb{F}_p$ (if $y'$ is hyperbolic) or in $\mu_{p+1}$ (if $y'$ is elliptic). We thus only need to include in our sieving counts for points on curves isomorphic over $\overline{\mathbb{F}}_p$ to

$$\alpha_1t^{e} + \alpha_2t^{-e} + \alpha_3 = s^{d}  + s^{-d}$$

where $d$ is odd, so that $(e, d) = 1$, and the curves are irreducible by Lemma
\ref{lemma:FourPieceCoprimeMonodromy}. If $x$ is hyperbolic, then we seek a point with $y'$ elliptic of maximal order. In order to apply the Weil bound we must re-write our equation as in (\ref{EqForConnectingToElliptic}); after so doing Weil and inclusion/exclusion allow us to connect to a point whose $y$-coordinate has order $p + 1.$ If $x$ is elliptic, then we seek a point with $y'$ hyperbolic of maximal order. The system is once again geometrically irreducible by the same monodromy argument, or equivalently by Lemma \ref{lemma:FourPieceCoprimeMonodromy} after the standard elliptic reparameterization as in the elliptic case of Section \ref{EndgameSection}.
\end{proof}

\begin{lemma}\label{lemma:fourOrbitsEndgameCase2}
    Suppose that
\[
A=B=+C,
\]
and that $(A,B,C,D)$ is not equivalent to the Cayley parameters. Suppose $(x, y, z)\in\mathcal{S}_{A, B, C, D}^*(p)$ has $\chi_p(y+2) = 1$ and $x$ of order at least $p^{\frac{1}{2} + \delta}$. Then, for all sufficiently large primes $p$, $(x, y, z)$ is connected via edges from $\Gamma$ to a point whose $y$-coordinate has order $\frac{p\pm 1}{2}$.
\end{lemma}
\begin{proof}
    By Lemma \ref{lemma:FourPieceSignOrderDictionary}, our assumption that \(y+2\) is a quadratic residue implies that any point $(x', y', z')$ in the orbit has $y'$ of order dividing $\frac{p\pm 1}{2}.$ In this particular case, the relevant curve is reducible over $\mathbb{F}_p,$ not just $\overline{\mathbb{F}}_p$. Indeed, if $x$ is hyperbolic, the curve is described over $\mathbb{F}_{p^2}$ by 

$$\alpha_1t^{e} + \alpha_2t^{-e} + \alpha_3 = s^{d}  + s^{-d}$$

which we rewrite as

$$\alpha_1t^{e}  +2\sqrt{\alpha_1\alpha_2} + \alpha_2t^{-e} = s^{d} + 2 + s^{-d}$$

using

$-2 = -2\sqrt{\alpha_1\alpha_2} + \alpha_3.$ Now, we have that 

$$y + 2 = (\sqrt{\alpha_1} + \sqrt{\alpha_2})^2;$$

as $y + 2$ is a quadratic residue we find that $\sqrt{\alpha_1} + \sqrt{\alpha_2}\in\mathbb{F}_p;$ as $\sqrt{\alpha_1\alpha_2}\in\mathbb{F}_p$ and $\alpha_2\in\mathbb{F}_p$ as well we find the same to be true of $\sqrt{\alpha_1}$ and $\sqrt{\alpha_2}$ by the identity

$$\sqrt{\alpha_2} = \frac{\sqrt{\alpha_1\alpha_2}}{\sqrt{\alpha_1} + \sqrt{\alpha_2}} + \frac{\alpha_2}{\sqrt{\alpha_1} + \sqrt{\alpha_2}}.$$

We thereby obtain the identity

$$(\sqrt{\alpha_1}t^{e/2} + \sqrt{\alpha_2}t^{-e/2})^2 = (s^{d/2} +  s^{-d/2})^2.$$

By Lemma \ref{lemma:FourPieceNonCoprimeMonodromy}, these two factors are the two geometrically irreducible components of the original transition curve:

$$\sqrt{\alpha_1}t^{e/2} + \sqrt{\alpha_2}t^{-e/2} = \pm(s^{d/2} + s^{-d/2}).$$

As we seek to connect to an elliptic point, we shall have to slightly adjust the shape of our equation so that it is defined over $\mathbb{F}_p$, rather than over $\mathbb{F}_{p^2}$, in order to apply the Weil bound; this can be handled somewhat similarly to what we did in (\ref{EqForConnectingToElliptic}) in the two-orbit case, but requires slightly more modification due to our factorization. The equations take the form

\[\begin{split}
    &\sqrt{\alpha_1}t^{e/2} + \sqrt{\alpha_2}t^{-e/2} = \pm2g_{d/2}(\xi, \eta)\\
    &\xi^2 - \omega\eta^2 = 1.
\end{split}\]

We now have twice as many points of order dividing $\frac{p+1}{2};$ as this is uniform over (even) divisors of $p+1$ we still obtain via sieving a point whose $y$-coordinate has order $\frac{p+1}{2}$, which is maximal in the component $\mathcal{S}_i$ given that $y+2$ is a quadratic residue. 

If instead $x$ is elliptic, we seek to move to a hyperbolic point of order $\frac{p-1}{2}.$ We start from the factorization over \(\overline{\mathbb F}_p\); by Lemma \ref{lemma:FourPieceNonCoprimeMonodromy}, its two factors are the geometrically irreducible components:

$$\sqrt{\alpha_1}t^{e/2} + \sqrt{\alpha_2}t^{-e/2} = \pm(s^{d/2} +  s^{-d/2}).$$

Those points of interest on each curve can again be described by a curve over $\mathbb{F}_p$ as in Section \ref{EndgameSection}, via equations

\[\begin{split}
    &\xi^2 - \omega \eta^2 = 1\\
    &u = \pm(2ag_{e/2}(\xi, \eta) + 2b\omega h_{e/2}(\xi, \eta))\\
    &u= s^{d/2} +  s^{-d/2},
\end{split}\]

where $u$ represents a fixed choice of $\sqrt{y + 2}$ and $a$ and $b$ are fixed and satisfy
$$a^2 - \omega b^2 =  \frac{x^2 - \frac{A}{2}x -  (A + 4)}{x^2 - 4};$$
$$y = 2a^2 + 2\omega b^2 + \alpha_3.$$

As before, we can sieve and obtain a point of order $\frac{p-1}{2},$ which is maximal in the $\mathcal{S}_i$ in which we're working.
\end{proof}
 
\begin{lemma} \label{lemma:fourOrbitsEndgameCase3}
    Suppose that
\[
A=B=-C,
\]
and that $(A,B,C,D)$ is not equivalent to the Cayley parameters. Suppose $(x, y, z)\in\mathcal{S}_{A, B, C, D}^*(p)$ has $\chi_p(y-2) = -1$ and $x$ of order at least $p^{\frac{1}{2} + \delta}$. Then, for all sufficiently large primes $p$, $(x, y, z)$ is connected via edges from $\Gamma$ to a point whose $y$-coordinate has order $\frac{p+1}{2}$ (if $x$ is hyperbolic) or $p-1$ (if $x$ is elliptic).
\end{lemma}
\begin{proof}
    Since our starting point $(x, y, z)$ has $y - 2$ a quadratic nonresidue, Lemma \ref{lemma:FourPieceConicSignComponents} implies that for any point $(x', y', z')$ in the orbit, if we represent \(y'=\eta+\eta^{-1}\), then \(\eta\) lies in the nonsquare coset of \(\mathbb F_p^\times\) if \(y'\) is hyperbolic, and in the square coset of \(\mu_{p+1}\) if \(y'\) is elliptic. The nature of the curves whose point-counts go into our sieving now depends on whether or not we seek to pass to a hyperbolic or an elliptic point, and thus on exactly where we start. 
 
 If $x$ is elliptic, then we seek to move to a hyperbolic point $y'$. In this case we only need to include in our sieving counts for points on curves isomorphic over $\overline{\mathbb{F}}_p$ to

$$\alpha_1t^{e} + \alpha_2t^{-e} + \alpha_3 = s^{d}  + s^{-d}$$

where $d$ is odd, so that $(e, d) = 1$, and the curves are irreducible by Lemma
\ref{lemma:FourPieceCoprimeMonodromy}. As our starting point has $x$ elliptic, we must re-write our system as one defined over $\mathbb{F}_p$ instead of over $\mathbb{F}_{p^2};$ this can be done simply by following the elliptic case of Section \ref{EndgameSection}.

If instead $x$ is hyperbolic, then we seek a point with $y'$ elliptic of order $\frac{p+1}{2}$. The relevant curves may be described over $\mathbb{F}_{p^2}$ by 

$$\alpha_1t^{e} + \alpha_2t^{-e} + \alpha_3 = s^{d}  + s^{-d}$$

which we rewrite as

$$\alpha_1t^{e}  - 2\sqrt{\alpha_1\alpha_2} + \alpha_2t^{-e} = s^{d} - 2 + s^{-d}$$

using

$2 = 2\sqrt{\alpha_1\alpha_2} + \alpha_3.$ Now, we have that 

$$y - 2 = (\sqrt{\alpha_1} - \sqrt{\alpha_2})^2;$$

as $y - 2$ is a quadratic nonresidue we find that $\sqrt{\alpha_1} - \sqrt{\alpha_2}\in\mathbb{F}_{p^2}\setminus\mathbb{F}_p;$ as $\sqrt{\alpha_1\alpha_2}\in\mathbb{F}_p$ and $\alpha_2\in\mathbb{F}_p$ as well we find the same to be true of $\sqrt{\alpha_1}$ and $\sqrt{\alpha_2}$ by the identity

$$\sqrt{\alpha_2} = \frac{\sqrt{\alpha_1\alpha_2}}{\sqrt{\alpha_1} - \sqrt{\alpha_2}} - \frac{\alpha_2}{\sqrt{\alpha_1} - \sqrt{\alpha_2}}.$$

We also note by these identities that each of $\sqrt{\alpha_1}, \sqrt{\alpha_2},$ and $\sqrt{\alpha_1} - \sqrt{\alpha_2}$ have their squares in $\mathbb{F}_p$, so fixing $\omega$ to be an arbitrary quadratic nonresidue, we can represent each as a multiple of $\sqrt\omega$ by an element of $\mathbb{F}_p$. We thereby obtain the identity

$$(\sqrt{\alpha_1}t^{e/2} - \sqrt{\alpha_2}t^{-e/2})^2 = (s^{d/2} -  s^{-d/2})^2.$$

By Lemma \ref{lemma:FourPieceNonCoprimeMonodromy}, these two factors are the two geometrically irreducible components:

$$\sqrt{\alpha_1}t^{e/2} - \sqrt{\alpha_2}t^{-e/2} = \pm(s^{d/2} - s^{-d/2}),$$

or equivalently

$$\frac{1}{\sqrt\omega}(\sqrt{\alpha_1}t^{e/2} - \sqrt{\alpha_2}t^{-e/2}) = \pm\frac{1}{\sqrt\omega}(s^{d/2} -  s^{-d/2}).$$

Once again, in order to apply the Weil bound, we have to adjust the shape of our equation so that it and its relevant solutions are defined over $\mathbb{F}_p$, not $\mathbb{F}_{p^2}.$ The adjusted system takes the form

\[\begin{split}
&\xi^2 - \omega\eta^2 = 1\\
    &\sqrt\frac{\alpha_1}{\omega}t^{e/2} - \sqrt{\frac{\alpha_2}{\omega}}t^{-e/2} = \pm2h_{d/2}(\xi, \eta).
\end{split}\]

We now have twice as many points of order dividing $\frac{p+1}{2};$ as this is uniform over (even) divisors of $p+1$ we still obtain via sieving a point whose $y$-coordinate has order $\frac{p+1}{2}$, which is maximal in the component $\mathcal{S}_i$ given that $y-2$ is a quadratic nonresidue.
\end{proof}

\begin{lemma}\label{lemma:fourOrbitsEndgameCase4}
Suppose that
\[
A=B=-C,
\]
and that $(A,B,C,D)$ is not equivalent to the Cayley parameters. Suppose $(x, y, z)\in\mathcal{S}_{A, B, C, D}^*(p)$ has $\chi_p(y-2) = 1$ and $x$ of order at least $p^{\frac{1}{2} + \delta}$. Then, for all sufficiently large primes $p$, $(x, y, z)$ is connected via edges from $\Gamma$ to a point whose $y$-coordinate has order $\frac{p-1}{2}$ (if $x$ is elliptic) or $p+1$ (if $x$ is hyperbolic).
\end{lemma}
\begin{proof}
    Since our starting point has $y - 2$ a quadratic residue, Lemma \ref{lemma:FourPieceSignOrderDictionary} implies that for any point $(x', y', z')$, writing \(y'=\eta+\eta^{-1}\), we have that \(\eta\in\mathbb F_p^\times\) is a quadratic residue if \(y'\) is hyperbolic, while in the elliptic case \(\eta\) lies in the nonsquare coset in \(\mu_{p+1}\). If $x$ is hyperbolic, then we seek to connect to an elliptic $y'$. The relevant curves are isomorphic over $\overline{\mathbb{F}}_p$ to

$$\alpha_1t^{e} + \alpha_2t^{-e} + \alpha_3 = s^{d}  + s^{-d}$$

where $d$ is odd, so that $(e, d) = 1$, and the curves are irreducible by Lemma
\ref{lemma:FourPieceCoprimeMonodromy}. As we seek a point with $y'$ elliptic, in order to apply the Weil bound we must re-write our equation as in (\ref{EqForConnectingToElliptic}); after so doing Weil and inclusion/exclusion allow us to connect to a point whose $y$-coordinate has order $p + 1.$

If, however, our starting point is elliptic, we start from the factorization
over \(\overline{\mathbb F}_p\); by Lemma
\ref{lemma:FourPieceNonCoprimeMonodromy}, its two factors are the geometrically
irreducible components:

$$\sqrt{\alpha_1}t^{e/2} - \sqrt{\alpha_2}t^{-e/2} = \pm(s^{d/2} -  s^{-d/2}).$$

Those points of interest on each curve can again be described by a curve over $\mathbb{F}_p$ as in Section \ref{EndgameSection}, via equations

\[\begin{split}
    &\xi^2 - \omega \eta^2 = 1\\
    &u = \pm(2ag_{e/2}(\xi, \eta) + 2b\omega h_{e/2}(\xi, \eta))\\
    &u= s^{d/2} -  s^{-d/2},
\end{split}\]

where $u$ represents a choice of $\sqrt{y - 2}$ and $a$ and $b$ are fixed and satisfy
$$a^2 - \omega b^2 =  -\frac{x^2 - \frac{A}{2}x + (A - 4)}{x^2 - 4}$$
$$y = 2a^2 + 2\omega b^2 + \alpha_3.$$

As before, we can sieve and obtain a point of order $\frac{p-1}{2},$ which is maximal in the $\mathcal{S}_i$ in which we're working.
\end{proof}

As we have now shown by way of Lemmas \ref{lemma:fourOrbitsEndgameCase1}, \ref{lemma:fourOrbitsEndgameCase2}, \ref{lemma:fourOrbitsEndgameCase3}, and \ref{lemma:fourOrbitsEndgameCase4}, no matter the values of $\epsilon$ and $\chi_p(y +2\epsilon)$, we can connect a hyperbolic point of sufficiently large order to an elliptic point of order $p+1$ or $\frac{p+1}{2}$ and an elliptic point of sufficiently large order to a hyperbolic point of order $p-1$ or $\frac{p-1}{2}$. As mentioned at the start of the proof of this lemma, applying this argument a second time allows us to connect an arbitrary point of sufficiently large order to a point of order $p\pm 1$ or $\frac{p\pm 1}{2}$, which we may freely choose to be hyperbolic or elliptic. 
\end{proof}
Finally, we record the cage connectivity step. Unlike in the two-orbit case, this step is relatively straightforward: the elaborate invariants allow us to coordinate between the components of conic sections of relative maximal order.
\begin{lemma}[Cage connectivity in the four-piece case]
\label{lemma:FourPieceCageConnectivity}
Suppose that
\[
A=B=\epsilon C,\qquad \epsilon\in\{\pm1\},
\]
and that $(A,B,C,D)$ is not equivalent to the Cayley parameters. Let
$i\in\{1,2,3,4\}$. For all sufficiently large primes $p$, the cage inside $\mathcal{S}_i$ is connected.

More precisely, suppose that $(x,y,z)$ and $(x',y',z')$ lie in $\mathcal{S}_i$,
with $x$ and $y'$ of maximal allowable order in $\mathcal{S}_i$. Then, except
possibly for a bounded number of exceptional pairs $(x,y')$, there is a value
$z''$ of maximal allowable order such that
\[
C_3(z'')\cap C_1(x)\cap \mathcal{S}_i\neq\emptyset
\]
and
\[
C_3(z'')\cap C_2(y')\cap \mathcal{S}_i\neq\emptyset.
\]
Furthermore, these two intersections lie in the same
$\langle V_1,V_2\rangle$-orbit inside
$C_3(z'')\cap\mathcal{S}_i$.
\end{lemma}
\begin{proof}
    By Lemma \ref{lemma:FourPieceConicSignComponents}, together with the parameterizations of Lemmas \ref{hyperbolics} and \ref{elliptics}, if $z''$ is chosen to be of order $p\pm 1$ or $\frac{p\pm 1}{2}$, choosing the sign in the latter case so the order is odd, then $\langle V_1, V_2\rangle$ acts transitively on $C_3(z'')\cap\mathcal{S}_i$. By the defining identities for the sets \(\mathcal S_i\) in Theorem
\ref{QuadraticObstructions}, if $\chi_p$ denotes the mod $p$ Legendre symbol, then $\chi_p(z'' + 2)$ and $\chi_p(x + 2\epsilon)$ determine $\chi_p(y'' + 2\epsilon)$ for any $y''$ appearing as a $y$-coordinate of a point in $C_3(z'')\cap\mathcal{S}_i,$ and similarly $\chi_p(z'' + 2)$ and $\chi_p(y' + 2\epsilon)$ determine $\chi_p(x'' + 2\epsilon)$ for any $x''$ appearing as an $x$-coordinate of a point in $C_3(z'')\cap\mathcal{S}_i.$ It follows that if

$$C_3(z'')\cap C_1(x)\neq \emptyset\text{ and}$$
$$C_3(z'')\cap C_2(y')\neq\emptyset$$
then $C_3(z'')$ intersects $C_1(x)$ and $C_2(y')$ in the same $\langle V_1, V_2\rangle$ orbit. Thus, to establish cage connectivity, we merely need to establish irreducibility of the curve cut out by the system

\[\begin{split}
    &z = t^\ell + t^{-\ell}\\
    &z^2 + A_1 z - B_1 v^2 = C_1\\
    &z^2 + A_2 z - B_2w^2 = C_2\\
    \end{split}\]
where $A_1, A_2, B_1, B_2, C_1,$ and $C_2$ are as in Subsection \ref{cageConnectivitySubsect}, or a suitable modification if we must connect via an elliptic $z''.$ But this follows straightforwardly from the polynomial-system construction in Subsection \ref{cageConnectivitySubsect}, at least for most pairs $(x, y').$
\end{proof}

We have now essentially completed the proof of Theorem \ref{DegenerateParameterOrbits}. For completeness, we assemble the lemmas:
\begin{proof}[Proof of Theorem \ref{DegenerateParameterOrbits}]
    By equivalence of parameters, we may assume that
\[
A=B,\qquad 4D+A^2=8C+16.
\]
We exclude the Cayley parameters by hypothesis.

First suppose that \(C\neq \pm A\). By Lemma \ref{lemma:DegenerateInheritedReductions}, the opening and middlegame results from Sections \ref{openingSection} and \ref{middlegameSection} apply inside each \(\mathcal S_i\), and hence connect any point in \(\mathcal S_i\) to a point of \(\mathcal S_i\) with order at least \(p^{\frac12+\delta}\). By Lemma \ref{endgameDegenerate} such points connect to points whose $x$- or $y$-orders are maximal. This reduces the problem to proving connectivity of the corresponding cages inside
\(\mathcal{S}_1\) and \(\mathcal{S}_2\). 
Lemma \ref{lemma:TwoPieceS2CageConnectivity} proves cage connectivity inside \(\mathcal{S}_2\). If \(A\neq0\), Lemma \ref{lemma:TwoPieceS1CageConnectivityNonzeroA} proves cage connectivity inside \(\mathcal{S}_1\), while if \(A=0\), Lemma \ref{lemma:TwoPieceS1CageConnectivityAZero} proves the same conclusion. Thus \(\Gamma\) acts transitively on each of \(\mathcal{S}_1\) and \(\mathcal{S}_2\).

Now suppose that \(C=\pm A\). As in the \(C\neq\pm A\) case, by Lemma \ref{lemma:DegenerateInheritedReductions}, the opening and middlegame results from Sections \ref{openingSection} and \ref{middlegameSection} apply inside each \(\mathcal S_i\), connecting any point in \(\mathcal S_i\) to a point of order at least \(p^{\frac12+\delta}\). Proposition \ref{proposition:FourPieceModifiedEndgame} gives the required modified endgame inside each \(\mathcal{S}_i\), connecting any point of order at least $p^{\frac{1}{2} + \delta}$ to a point whose chosen coordinate has maximal allowable order in \(\mathcal S_i\). Lemma \ref{lemma:FourPieceCageConnectivity} then proves cage connectivity inside each \(\mathcal{S}_i\), from which it follows that \(\Gamma\) acts transitively on each of \(\mathcal{S}_1,\mathcal{S}_2,\mathcal{S}_3,\) and \(\mathcal{S}_4\).
\end{proof}

Together with Theorem \ref{QuadraticObstructions}, we have shown that there are either exactly two or four large orbits, and described when each occurs in terms of the parameters $A, B, C,$ and $D$.
For completeness, we also sketch how the same ingredients prove Theorem  \ref{th:DegenerateEndgamePlusMiddlegame} (the degenerate analogue of Theorem \ref{th:EndgamePlusMiddlegame}) which holds for all sufficiently large $p$ and is independent of the opening.
 
\begin{proof}[Proof sketch of Theorem \ref{th:DegenerateEndgamePlusMiddlegame}]
Let $(A,B,C,D)$ be degenerate modulo $p$, not equivalent to the Cayley parameters, and normalized so that $A=B$ and $4D+A^2=8C+16$; let $\mathcal{S}_i(p)$ be one of the $\Gamma$-invariant sets of Theorem \ref{QuadraticObstructions} (two or four of them). For each $i$ we define the cage $\mathcal{K}_i$ inside $\mathcal{S}_i(p)$ as in Subsection \ref{cageConnectivitySubsect}, with the modifications of Section \ref{ModificationsDegenerate}, and run the argument inside $\mathcal{S}_i(p)$:
\begin{enumerate}
    \item \textit{Cage connectivity and endgame.} By the lemmas on cage connectivity (the two-orbit lemmas, resp.\ Lemma \ref{lemma:FourPieceCageConnectivity}), $\mathcal{K}_i$ is connected inside $\mathcal{S}_i(p)$; by Lemma \ref{endgameDegenerate}, resp.\ Proposition \ref{proposition:FourPieceModifiedEndgame}, every point of $\mathcal{S}_i(p)$ of order at least $p^{\frac12+\delta}$, aside from boundedly many exceptional values of the relevant coordinate, is joined to $\mathcal{K}_i$ without leaving $\mathcal{S}_i(p)$. In these arguments the equations involving the perfect squares $f_1,f_2$ (and $f_3$ if $A=B=\pm C$) are discarded, as in Section \ref{ModificationsDegenerate}.
    \item \textit{Middlegame.} By Lemma \ref{lemma:DegenerateInheritedReductions}, the middlegame (Section \ref{middlegameSection}) applies inside $\mathcal{S}_i(p)$: every non-double-fixed point of $\mathcal{S}_i(p)$ of order at least $p^\epsilon$ is joined to $\mathcal{K}_i$. In the degenerate case there are about $p$ double fixed points, but all but at most $12$ of them have a coordinate equal to $\pm2$, and each lies in $\mathcal{S}_i(p)$ together with its unique neighbor $P$ under the Vieta involution which does not fix it. If $P$ were also a double fixed point, the pair would be a $\Gamma$-orbit of size two, removed in passing to $\mathcal{S}^*(p)$; so $P$ is not double fixed, and has the same coordinate of high order as the point we started from, unless that coordinate has order less than $p^\epsilon$. Hence all but $O(p^{2\epsilon})$ of these points are joined to $\mathcal{K}_i$ through $P$.

    \item \textit{Passing from $\Gamma'$ to $\Gamma$.} In the two-orbits case, where the arguments of Section \ref{ModificationsDegenerate} use the larger groups $\Gamma_{xy}$ or $\Gamma'$ (Lemma \ref{lemma:Degenerate2OrbitsComparisonOfGroups}, and the use of $\tau$ and $\operatorname{neg}\circ\tau$ when $A=B=0$), the transfer to $\Gamma$ is carried out exactly as in Proposition \ref{prop:ComparisonLargeComponent}: the transpositions and double sign flips involved fix points of high order lying in the large component, so they stabilize it. In the four-orbit case we never used the extra symmetries in $\Gamma'$, and so this step is not necessary.
    \item \textit{Counting.} Let $\mathcal{C}_i(p)$ be the $\Gamma$-orbit of $\mathcal{K}_i$. As in the nondegenerate case, $\mathcal{S}_i(p)\setminus\mathcal{C}_i(p)$ is contained in the union of the finite orbits in $\mathcal{E}(p)$, the points of order less than $p^\epsilon$, and the exceptional double fixed points handled in (2), which has $O(p^{4\epsilon})$ elements; replacing $\epsilon$ by $\epsilon/4$ gives $|\mathcal{S}_i(p)\setminus\mathcal{C}_i(p)|\ll_\epsilon p^\epsilon$. \qedhere
\end{enumerate}
\end{proof}

\section{Acknowledgments}
The author is deeply grateful to his advisor Alex Gamburd for introducing him to the problem, guiding him through the related literature, and for numerous helpful conversations. This paper would not exist without Alex Gamburd's mentorship. The author would also like to thank Andrew Obus for helpful conversations, especially related to some of the irreducibility proofs in the endgame and to the writeup of Proposition \ref{prop:SpreadingOut}; Yosef Berman, for pushing him to use computer explorations early and often and for assisting him in doing so; and Ajmain Yamin and Brian Kingsbury for additional help with the computer-assisted portions of this project. The author completed this work while partially supported by a Graduate Center Fellowship at the CUNY Graduate Center.

The author discussed the problem of achieving parameter independence of the density-zero set of exceptional primes in the main theorem with the free version of the AI model ChatGPT. In particular, some of its remarks suggested the encoding of finite orbit graphs as schemes over $\mathbb Z$ to relate bounded orbits in characteristic zero and large positive characteristic, inspiring the approach developed in Proposition \ref{prop:SpreadingOut}. The free version of the AI model ChatGPT was also used to produce TikZ code for the Markoff tree (Figure \ref{MarkoffTree}), as well as to produce and edit TikZ code for Figures \ref{ExceptionalFiniteOrbit1}, \ref{ExceptionalFiniteOrbit2}, and \ref{ExceptionalFiniteOrbit3} from hand-drawn diagrams produced by the author. In all cases, the author verified manually that the diagrams rendered by the TikZ code were correct. Finally, the free version of the AI model Claude and the ``Pro'' version of ChatGPT (accessed via a free trial) were used for review and editing of this paper, including for proofreading, checking arguments and exposition, re-writing some arguments and exposition under the author's directoin, and suggesting ways to lemmatize pre-existing arguments which had been written in a single, overly long block. In all cases, the author either manually corrected issues flagged by the AI models, or manually copied in edits to the language suggested by the AI models; no changes were made to the manuscript without the author's direct oversight. The author takes full and sole responsibility for the contents of this manuscript.

\appendix
\section{The Polynomial \texorpdfstring{$\Delta$}{Delta}}\label{appendix:Delta}
The polynomial $\Delta$ discussed in the body of the article is given by
\[\begin{split}
\Delta(A, B, C, D) &= 4A^3B^3C^3-27A^4B^4-6A^4B^2C^2-6A^2B^4C^2-27A^4C^4-6A^2B^2C^4\\
&-27B^4C^4+18A^3B^3CD+18A^3BC^3D+18AB^3C^3D+A^2B^2C^2D^2\\
&+192A^5BC+24A^3B^3C+192AB^5C+24A^3BC^3+24AB^3C^3+192ABC^5\\
&-144A^4B^2D-144A^2B^4D-144A^4C^2D-88A^2B^2C^2D-144B^4C^2D\\
&-144A^2C^4D-144B^2C^4D+80A^3BCD^2+80AB^3CD^2+80ABC^3D^2\\
&+4A^2B^2D^3+4A^2C^2D^3+4B^2C^2D^3+256A^6+192A^4B^2+192A^2B^4+256B^6\\
&+192A^4C^2-1328A^2B^2C^2+192B^4C^2+192A^2C^4+192B^2C^4+256C^6\\
&+1024A^3BCD+1024AB^3CD+1024ABC^3D-128A^4D^2-784A^2B^2D^2\\
&-128B^4D^2-784A^2C^2D^2-784B^2C^2D^2-128C^4D^2+352ABCD^3\\
&+16A^2D^4+16B^2D^4+16C^2D^4-3328A^3BC-3328AB^3C-3328ABC^3\\
&+2048A^4D+1984A^2B^2D+2048B^4D+1984A^2C^2D+1984B^2C^2D+2048C^4D\\
&+1408ABCD^2-768A^2D^3-768B^2D^3-768C^2D^3+64D^5-2048A^4-3840A^2B^2\\
&-2048B^4-3840A^2C^2-3840B^2C^2-2048C^4-11776ABCD+5632A^2D^2\\
&+5632B^2D^2+5632C^2D^2-1024D^4+2048ABC-12288A^2D-12288B^2D\\
&-12288C^2D+6144D^3+4096A^2+4096B^2+4096C^2-16384D^2+16384D.
\end{split}\]
\section{Macaulay2 Code}\label{appendix:Macaulay2}
In this section, we provide code for all of the computations done in this paper using Macaulay2 \cite{Macaulay2}. First, to compute $\Delta$ as in Lemma \ref{ZariskiClose}, run:

\begin{lstlisting}
S=QQ[a_1,a_2,a_3]
R=QQ[A,B,C,D]
phi=map(S,R,{-2*a_1-a_2*a_3,-2*a_2-a_1*a_3,
        -2*a_3-a_1*a_2,-2*a_1*a_2*a_3-a_1^2-a_2^2-a_3^2})
I=ker phi
\end{lstlisting}

To prove the equalities $\Delta_1 = \Delta_2 = \Delta_3$ in Lemma \ref{DeltaUniversality}, we run the following code:

    \begin{lstlisting}
R = ZZ[x, y, z, A, B, C, D]
f1 = x^4 - A*x^3 - (D+4)*x^2 + (4*A + B*C)*x + (4*D + B^2 + C^2)
f2 = y^4 - B*y^3 - (D+4)*y^2 + (4*B + A*C)*y + (4*D + A^2 + C^2)
f3 = z^4 - C*z^3 - (D+4)*z^2 + (4*C + A*B)*z + (4*D + A^2 + B^2)
Delta1 = discriminant(f1, x)
Delta2 = discriminant(f2, y)
Delta3 = discriminant(f3, z)
Delta1 == Delta2
Delta1 == Delta3
\end{lstlisting}

To check the final equality $\Delta_1 = \Delta$ in Lemma \ref{DeltaUniversality}, run the following code:
\begin{lstlisting}
S=QQ[a_1,a_2,a_3]
R=QQ[A,B,C,D]
T = R[x]
phi=map(S,R,{-2*a_1-a_2*a_3,-2*a_2-a_1*a_3,
        -2*a_3-a_1*a_2,-2*a_1*a_2*a_3-a_1^2-a_2^2-a_3^2})
I=ker phi
f = x^4 - A*x^3 - (D+4)*x^2 + (4*A + B*C)*x + (4*D + B^2 + C^2)
Delta = discriminant(f, x)
ideal(Delta) == I
\end{lstlisting}

To check the divisibility asserted in Proposition \ref{DegeneratesAndDelta}, run the following code and notice the first term in the product expansion:
\begin{lstlisting}
R = ZZ[A, B, C, D, X]
S = ZZ[a, c, d, x]
f = X^4 - A*X^3 - (D+4)*X^2 + (4*A + B*C)*X + (4*D + B^2 + C^2)
Delta = discriminant(f, X)
phi = map(S, R, {a, a, c, d, x})
factor(phi(Delta))
\end{lstlisting}

Finally, to check the congruence asserted in the course of the proof of Theorem \ref{QuadraticObstructions}, run the following code and note the program returns ``0.''
\begin{lstlisting}
R = ZZ[x, y, z, A, C]
f = (A*x - 2*x^2 - 2*x*z + 2*C - 4*x - 4*z)^2
g = (x*y + 2*x + 2*y + C - A + 4)*(x^2*z + 2*x*z - x*y + A*x + 2*x - 2*y + C + A + 4)
h = 4*x*y*z + 4*A*x + 4*A*y + 4*C*z - A^2 + 8*C + 16 - 4*x^2 - 4*y^2 - 4*z^2
4*g-f - (x^2 + 4*x + 4)*h    
\end{lstlisting}

\end{document}